\documentclass[3p,12pt]{elsarticle}

\usepackage{amsfonts}
\usepackage{amssymb}
\usepackage{mathrsfs}
\usepackage{amsmath}
\usepackage{graphicx}
\usepackage{multirow}
\usepackage{amsthm}
\usepackage{subfig}
\usepackage{caption}
\usepackage{color}
\usepackage[hidelinks]{hyperref}
\numberwithin{equation}{section}

\newtheorem{theorem}{Theorem}[section]
\newtheorem{lemma}{Lemma}[section]
\newtheorem{example}{Example}[section]
\newtheorem{remark}{Remark}[section]
\newtheorem{definition}{Definition}[section]
\newtheorem{algo}[theorem]{Algorithm}
\journal{}
\journal{}

\allowdisplaybreaks
\begin{document}

\begin{figure}[!htbp]   
    \centering
    \includegraphics[height=10cm,width=16cm]{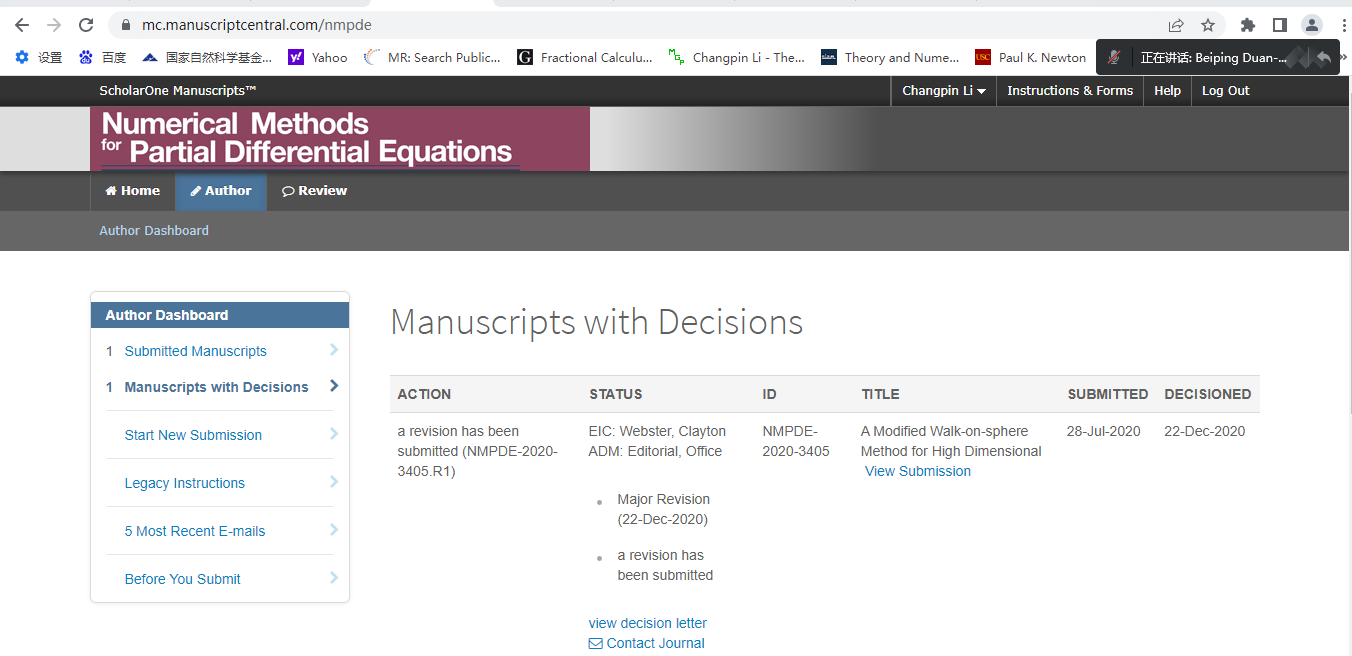}
\end{figure}

\begin{frontmatter}
\title{A Modified Walk-on-sphere Method for High Dimensional Fractional Poisson Equation{\footnote
{\footnotesize Li and Wang were supported by the National Natural Science Foundation of China under grant
nos. 11926319 and 11926336.  Zhang was partially supported by  ARO/MURI grant W911NF-15-1-0562.
}
}
}
\author{Caiyu Jiao\textsuperscript{a}}
\author{Changpin Li\textsuperscript{a}{$^{,\ast}$}}
\author{Hexiang Wang\textsuperscript{b}}
\author{Zhongqiang Zhang\textsuperscript{c}}
\address{\textsuperscript{a}Department of Mathematics, Shanghai University, Shanghai 200444, China}
\address{\textsuperscript{b}School of Mathematics and Statistics, Kashi University, Kashi 844006, Xinjiang, China}
\address{\textsuperscript{c}Department of Mathematical Sciences, Worcester Polytechnic Institute, Worcester, MA 01609 USA}
\cortext[cor1]{Corresponding author. E-mail: lcp@shu.edu.cn}
\begin{abstract}{  }
We develop walk-on-sphere method for fractional Poisson equations with Dirichilet boundary conditions in high dimensions.
 The walk-on-sphere method is based on probabilistic representation of the fractional Poisson equation. We propose
 efficient quadrature rules to {evaluate integral representation} in the ball and apply rejection sampling method
 to drawing from the computed probabilities in general domains. Moreover, we provide an estimate of the number of
 walks in the mean value for the method when the domain is a ball. We show that the number of walks is increasing
 in the fractional order and the distance of the starting point to the origin. We also give the relationship between
 the Green function of fractional Laplace equation and that of the classical Laplace equation. Numerical results
 for problems in 2-10 dimensions verify our theory and the efficiency of the modified walk-on-sphere method.
\end{abstract}

\begin{keyword}
walk on spheres, fractional Laplacian, modified walk-on-sphere method, inexact sampling
\end{keyword}
\end{frontmatter}

\section{Introduction}
The fractional Laplacian, $(-\Delta)^{s}$, is a prototypical operator for modeling nonlocal and anomalous phenomenon
which incorporates long range interactions \cite{Li&Yi2019}. It arises in many areas of applications, including models
for turbulent flows, porous media flows, pollutant transport, quantum mechanics, stochastic dynamics, and finance
\cite{Das2011,Herrmann2011,Hilfer2000,Oldham&Spanier1974}.

Numerical methods for fractional Laplacian operator and differential equations with fractional Laplacian operator
have been investigated in dozens of few papers, such as in finite difference methods \cite{Duo&Wyk&Zhang2018,Huang&Oberman2014,
Li&Cai2019}, spectral methods \cite{AcoBBM18,Zhang19,HaoZhang20}, finite element methods \cite{AcoBor15,AinsworthG18,Marta13}
and probabilistic methods \cite{Gao&Duan2014,Kyprianou&Osojnik2018}. See review papers \cite{BonBNe18,Lischke&Pang2020,DElDGe21}
for more details. All these methods are nonlocal and thus expensive in high dimensions, except the probabilistic methods.
While the most economical method is with quasi-linear complexity in number of physical nodes \cite{AinsworthG18} using
finite element methods in 2D, no numerical results are reported for Poisson equation with fractional Laplacian over general
bounded domain in high dimensions such as in three or much higher dimensions.

Probabilistic methods (usually implemented with Monte Carlo methods, say, for example, \cite{kloeden,lay}) for partial
differential equations with/without fractional Laplacian are based on the probabilistic representation of the Laplacian/fractional
Laplacian, see e.g. \cite{Applebaum2009}. These methods do not require any discretization in space. In one of such methods,
{walk-on-sphere method} (e.g. \cite{Muller1956}) does not even require discretization in time or even the diffusion trajectory
of the stochastic process.  Such probabilistic methods are particularly advantageous when the geometry domain $\Omega$ is very
complex or if the solution of the partial differential equation is required at  a relatively small number of points.

Though Monte Carlo methods need $\mathcal{O}(M^{2})$ walks to achieve standard deviation $\mathcal{O}(M^{-1})$, it is a reliable
method in arbitrarily high dimensions. In addition, they can be efficiently implemented on massively parallel computers.

In this work, we develop efficient probabilistic  methods in high dimensions for the following fractional Poisson equation
on an open bounded domain with an extended Dirichlet boundary value condition (see e.g. in \cite{Ros-Oton&Serra2014}):
 \begin{equation}\label{eq:possionproblem}
 \left\{
 \begin{aligned}
 &(-\Delta)^{s}u(\texttt{{\rm{x}}})=f(\texttt{{\rm{x}}}), &&\, \texttt{{\rm{x}}}\in\Omega,
 \\
 &u(\texttt{{\rm{x}}})=g(\texttt{{\rm{x}}}), &&\, \texttt{{\rm{x}}}\in \mathbb{R}^{n}\backslash \Omega,
 \end{aligned}
 \right.
 \end{equation}
 where $s\in(0,1)$, $n\geq1$ and we use the integral definition defined by a singular integral
 \cite{Pozrikidis2016} which coincides with Riesz derivative definition on the whole space \cite{Cai&Li2019},
 \begin{equation}\label{1pvfractionLaplacian}
 (-\Delta)^{s}u(\texttt{{\rm{x}}})=C(n,s)\,
 {\rm P.V.}\int_{\mathbb{R}^{n}}\frac{u(\texttt{{\rm{x}}})-u(\texttt{{\rm{y}}})}{|\texttt{{\rm{x}}}-\texttt{{\rm{y}}}|^{n+2s}}
 {\rm d}\texttt{{\rm{y}}}.
 \end{equation}
 Here ${\rm P.V.}$ stands for the Cauchy principle value and the constant $C(n,s)$ is given by \cite{Cai&Li2019}
 \begin{equation}
 C(n,s)=\left(\,\,\int_{\mathbb{R}^{n}}\frac{1-\cos\zeta_{1}}
 {|\zeta|^{n+2s}}{\rm d}\zeta\right)^{-1}=\frac{{s}2^{2s}\Gamma(\frac{n}{2}+s)}{{\pi}^{\frac{n}{2}}\Gamma(1-s)}
 \end{equation}
 with $\zeta_{1}$ being the first component of $\zeta=(\zeta_{1},\zeta_{2},\ldots,\zeta_{n})\in\mathbb{R}^{n}$
 and $\Gamma$ representing the Gamma function.

We develop our probabilistic method along the line of walk on spheres developed in   \cite{DeLaurentis&Romero1990,Elepov&Mihailov1973,Hwang&Mascagni2003,Kyprianou&Osojnik2018,Muller1956,
Sabelfeld1991}. We use the modified walk-on-sphere method based on Poisson kernel and
Green function to solve equation \eqref{eq:possionproblem}, which is also called ``one point random estimation"
(OPRE) method  {\cite{DeLaurentis&Romero1990}}. Specifically, every jump of one particle can be simulated with a  certain  probability until this particle is out of domain $\Omega$ (see Figure 1) and all of the particles'
processes compose the approximate solution.

The difficulty  in the implementation of this method is the computation of the probability in high dimensions,
which hasn't  been addressed in literature. We will introduce some quadrature methods in Sections 2-4 and use
rejection sampling {method} to draw samples whenever the exact sampling is not feasible. Another practical issue when implementing this approach is to estimate the average number of steps for it. When $s\rightarrow 1$, the Green function of fractional Laplacian equation is reduced to that of integer-order case. This issue is discussed in
Section \ref{sec:bound-steps-walks}.

\begin{figure}[!htbp]
	\centering
	\includegraphics[height=5.5cm, width=5.5cm]{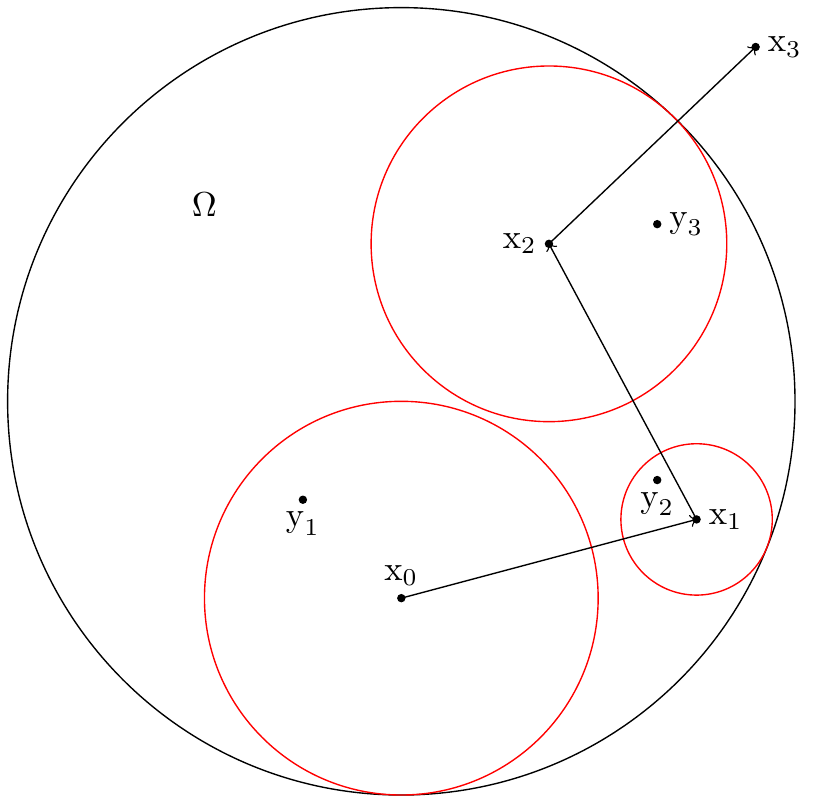}
	\captionsetup{font={footnotesize}}
	\caption{Steps of the modified walk-on-sphere algorithm until exiting the domain $\Omega$} \label{1fig1}
	
\end{figure}

The contributions of this work are summarized as follows.
\begin{itemize}
\item We apply quadrature rules to the singular representation of walk on spheres and then numerically solve
equation (\ref{eq:possionproblem}) in $n$-dimensional ball. We present some convergence analysis of the proposed
approach. Compared with \cite{Kyprianou&Osojnik2018}, we give the deterministic numerical method to solve equation
in $n$-dimensional ball.
	
\item We provide the modified walk-on-sphere method to numerically solve equation (\ref{eq:possionproblem}) on
general domains in high dimensions. We find approximate probabilities of walk on spheres and draw samples from
these probabilities. Compared with \cite{Kyprianou&Osojnik2018}, the current work can be applied to arbitrary
dimensions by rejection sampling method and reduces the computational time because of OPRE method.
We illustrate the efficiency of the proposed approach using a numerical example of fractional Poisson equation
in ten dimensions.
	
\item We give an upper bound of the number of walks in the modified walk-on-sphere method for fractional Poisson
equation in a ball. The upper bound is a function increasing with respect to the fractional order $s$ and the
distance of the starting point  $x_0$ to the origin. When $s\rightarrow 1$, Green function for the fractional
Laplacian equation degenerates into that of the classical Laplace equation.
\end{itemize}

The rest of this paper is outlined as follows. In Section \ref{sec:prob-reps-frac}, we present the probabilistic
representation for the homogeneous problem \eqref{eq:possionproblem} where the domain $\Omega$ is a ball. We also
present quadrature rules to approximate the integrals in the representation.

In Section \ref{sec:mod-walk-on-spheres}, we present the modified walk-on-sphere algorithm for the equation
(\ref{eq:possionproblem}) on an open bounded domain $\Omega$ in one dimension and high dimensions. For high dimensional
problems, we present a simple and efficient rejection sampling method based on the truncated Gaussian distribution
to draw samples from the probabilities of random walks.

In Section \ref{sec:bound-steps-walks}, we derive an estimate  of number of walks for the method when the problem
is considered on a ball. We show that  the number of walks is increasing with respect to the fractional order and the distance of the starting point to the origin. We also give the relationship of Green functions between the fractional Laplacian and the classical Laplace equation.

In Section \ref{sec:num-exm}, numerical experiments are performed to confirm the convergent order of the proposed
numerical evaluations in Section \ref{sec:prob-reps-frac} and {the efficiency of the modified walk-on-sphere method
in Section \ref{sec:mod-walk-on-spheres}}. Finally, we summarize our work in the last section.

\section{Probabilistic representation for fractional Poisson equation} \label{sec:prob-reps-frac}
In this section, we give an integral representation of $u(\texttt{{\rm{x}}})$ in \eqref{eq:possionproblem} where $\Omega$
is a ball centered at $\texttt{{\rm{x}}}\in \mathbb{R}^{n}$. The representation formula of the homogeneous equation
is discretized by using quadrature formula and the corresponding error estimates are derived for $n$ dimensional case.

\subsection{Fractional Poisson equation on a ball}

We start from equation \eqref{eq:possionproblem} where $\Omega$ is a ball centered at the origin with radius $r>0$, i.e.,
\begin{equation}\label{eq:possionball}
\left\{
\begin{aligned}
&(-\Delta)^{s}u(\texttt{{\rm{x}}})=f(\texttt{{\rm{x}}}), \,&& \texttt{{\rm{x}}}\in \mathbb{B}_{r},
\\
&u(\texttt{{\rm{x}}})=g(\texttt{{\rm{x}}}),\, && \texttt{{\rm{x}}}\in \mathbb{R}^{n}\backslash \mathbb{B}_{r}.
\end{aligned}
\right.
\end{equation}
To give the integral representation for $u(\texttt{{\rm{x}}})$, we introduce the following definitions.

\begin{definition}{\rm{(\cite{Bucur2016})}}
Let $r>0$ be fixed. For any $\texttt{{\rm{x}}} \in \mathbb{B}_{r}$ and any $\texttt{{\rm{y}}} \in \mathbb{R}^{n}\backslash
\overline{\mathbb{B}}_{r}$, the Poisson kernel $P_{r}$ is defined by
\begin{equation}
P_{r}(\texttt{{\rm{x}}},\texttt{{\rm{y}}})=\alpha(n,s)\left(\frac{r^{2}-|\texttt{{\rm{x}}}|^{2}}{|\texttt{{\rm{y}}}|^{2}-
r^{2}}\right)^{s}\frac{1}{|\texttt{{\rm{x}}}-\texttt{{\rm{y}}}|^{n}},
\end{equation}
where the constant $\alpha(n,s)$ is given by
\begin{equation}
\alpha(n,s)=\frac{\Gamma(\frac{n}{2})\sin(\pi s)}{{\pi}^{\frac{n}{2}+1}}.
\end{equation}
\end{definition}

\begin{definition}{\rm{(\cite{Bucur2016})}}
Let $r>0$ be fixed. For any $\texttt{{\rm{x}}},\texttt{{\rm{y}}}\in \mathbb{B}_{r}$ and $\texttt{{\rm{x}}}\neq \texttt{{\rm{y}}}$,
Green function $G$ is defined by
\begin{equation}\label{Greenfunction}
G(\texttt{{\rm{x}}},\texttt{{\rm{y}}})=\left\{
\begin{aligned}
&\kappa(1,\frac{1}{2})\log\left({\frac{r^2-xy+\sqrt{(r^2-x^2)(r^2-y^2)}}{r|x-y|}}\right), \, &&n=1,
\\
&\kappa(n,s) |\texttt{{\rm{x}}}-\texttt{{\rm{y}}}|^{2s-n}\int_{0}^{r^{\ast}(\texttt{{\rm{x}}},\texttt{{\rm{y}}})}
\frac{t^{s-1}}{(t+1)^{\frac{n}{2}}}{\rm d}t,\, &&n \geq 2,
\end{aligned}
\right.
\end{equation}
where
\begin{equation}
r^{\ast}(\texttt{{\rm{x}}},\texttt{{\rm{y}}})=\frac{(r^{2}-|\texttt{{\rm{x}}}|^{2})(r^{2}-|\texttt{{\rm{y}}}|^{2})}{r^{2}
|\texttt{{\rm{x}}}-\texttt{{\rm{y}}}|^{2}},
\end{equation}
and $\kappa(n,s)$ denotes a normalization constant
\begin{equation}
\kappa(n,s)=\left\{
\begin{aligned}
&\frac{1}{\pi}, \, && n=1,
\\
&\frac{\Gamma (\frac{n}{2})}{2^{2s} {\pi}^{\frac{n}{2}}{\Gamma^{2}(s)}}, \, &&n \geq 2.
\end{aligned}
\right.
\end{equation}
\end{definition}

\noindent Then the representation formula for \eqref{eq:possionball} is stated in the following theorem.
\begin{theorem}{\rm{(\cite{Bucur2016})}}\label{integralrepre}
Let $r>0$, $f\in C^{2s+\varepsilon}(\mathbb{B}_{r})\cap C(\overline{\mathbb{B}}_{r})$ for sufficiently small $\varepsilon >0$
and $g\in L_{s}^{1}(\mathbb{R}^{n}) \cap C({\mathbb{R}^{n}})$. Then there exists a unique continuous solution to \eqref{eq:possionball}
which is given by
\begin{equation}\label{reprformula}
  u(\texttt{{\rm{x}}})=\int_{\mathbb{R}^{n}\backslash \mathbb{B}_{r}}P_{r}(\texttt{{\rm{x}}},\texttt{{\rm{y}}})g(\texttt{{\rm{y}}})
  {\rm d}\texttt{{\rm{y}}}+\int_{\mathbb{B}_{r}}G(\texttt{{\rm{x}}},\texttt{{\rm{y}}})f(\texttt{{\rm{y}}}){\rm d}\texttt{{\rm{y}}}.
\end{equation}
\end{theorem}

From Theorem \ref{integralrepre}, we can derive the representation formula for problem \eqref{eq:possionproblem}
with $\Omega$ being an arbitrary ball, centered at $\texttt{{\rm{x}}}_{0}\in \mathbb{R}^{n}$, namely
\begin{equation}\label{eq:possionball1}
\left\{
\begin{aligned}
 &(-\Delta)^{s}u(\texttt{{\rm{x}}})=f(\texttt{{\rm{x}}}), \,&& \texttt{{\rm{x}}}\in \mathbb{B}_{r}(\texttt{{\rm{x}}}_{0}),
 \\
&u(\texttt{{\rm{x}}})=g(\texttt{{\rm{x}}}), \, && \texttt{{\rm{x}}}\in \mathbb{R}^{n}\backslash \mathbb{B}_{r}(\texttt{{\rm{x}}}_{0}).
\end{aligned}
\right.
\end{equation}

\noindent{Through variable translation and replacement, we obtain}
\begin{equation}
  u(\texttt{{\rm{x}}})=\int_{\mathbb{R}^{n}\backslash \mathbb{B}_{r}(\texttt{{\rm{x}}}_{0})}P_{r}(\texttt{{\rm{x}}}-\texttt{{\rm{x}}}_{0},\texttt{{\rm{y}}}-
  \texttt{{\rm{x}}}_{0})g(\texttt{{\rm{y}}}){\rm d}\texttt{{\rm{y}}}+\int_{\mathbb{B}_{r}(\texttt{{\rm{x}}}_{0})}G(\texttt{{\rm{x}}}-\texttt{{\rm{x}}}_{0},
  \texttt{{\rm{y}}}-\texttt{{\rm{x}}}_{0})f(\texttt{{\rm{y}}}){\rm d}\texttt{{\rm{y}}}.
\end{equation}

\subsection{Numerical method  for \eqref{eq:possionball} using the Poisson kernel}

In this subsection, we first derive the numerical method for the following Dirichlet problem,
\begin{equation}\label{eq:Dirichlet problem}
\left\{
 \begin{aligned}
 &(-\Delta)^{s}u(\texttt{{\rm{x}}})=0,\, && \texttt{{\rm{x}}}\in \mathbb{B}_{r},
 \\ &u(\texttt{{\rm{x}}})=g(\texttt{{\rm{x}}}), \,&& \texttt{{\rm{x}}}\in \mathbb{R}^{n}\backslash \mathbb{B}_{r}.
 \end{aligned}
 \right.
\end{equation}

From Theorem \ref{integralrepre},
\begin{equation}
  u(\texttt{{\rm{x}}})=\int_{\mathbb{R}^{n}\backslash \mathbb{B}_{r}}P_{r}(\texttt{{\rm{x}}},\texttt{{\rm{y}}})
  g(\texttt{{\rm{y}}}){\rm d}\texttt{{\rm{y}}},
\end{equation}
provided that $g(\texttt{{\rm{x}}})$ satisfies the condition {in Theorem \ref{integralrepre}}.

To compute this integral in the above formula,  we use change of variables by utilizing the hyperspherical
coordinates with radius $\rho >r$, angles $\varphi_{1},\varphi_{2},\cdots ,\varphi_{n-2} \in [0,\pi]$, and
$\theta \in [0,2\pi]$. Then, it holds {for $n\geq3$} that
\begin{equation}\label{3hypercoordinate}
\left\{
 \begin{aligned}
  &y_{1} =\rho\sin{\varphi_{1}}\sin{\varphi_{2}}\cdots\sin{\varphi_{n-2}}\sin{\theta},
  \\
  &y_{2} =\rho\sin{\varphi_{1}}\sin{\varphi_{2}}\cdots\sin{\varphi_{n-2}}\cos{\theta},
  \\
  &y_{3} =\rho\sin{\varphi_{1}}\sin{\varphi_{2}}\cdots\cos{\varphi_{n-2}},
  \\
  &\cdots
  \\
  &y_{n-1} =\rho \sin{\varphi_{1}}\cos{\varphi_{2}},
  \\
  &y_{n} =\rho\cos{\varphi_{1}}.
 \end{aligned}
 \right.
\end{equation}
The Jacobian of the transformation is given by
$\rho^{n-1}\sin^{n-2}{(\varphi_{1})}\sin^{n-3}{(\varphi_{2})}\cdots\sin{(\varphi_{n-2})}$.
Here we discuss the case with $n\geq3$. Two dimensional case can be derived similarly
so is omitted here or is left for the interested reader.
Without loss of generality and up to rotations, we assume {$\texttt{{\rm{x}}}=e_{n}=|\texttt{{\rm{x}}}|(0,0,\ldots,1)$, so (see Figure \ref{fig1})
\begin{equation}
 |\texttt{{\rm{x}}}-\texttt{{\rm{y}}}|^{2}=
  \rho^{2}+|\texttt{{\rm{x}}}|^{2}-2\rho |\texttt{{\rm{x}}}|\cos\varphi_{1},   \qquad        n\geq 3.
\end{equation}
\begin{figure}[!htbp]
    \centering
    \includegraphics[height=6cm,width=6cm]{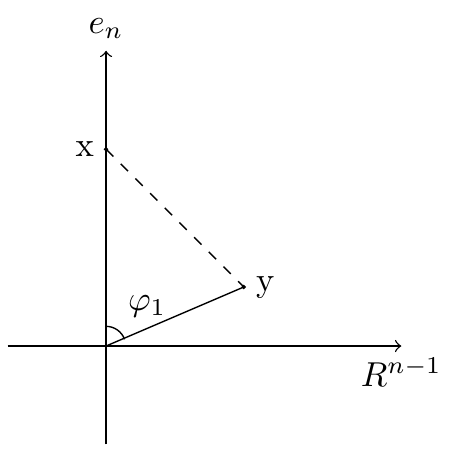}
 \captionsetup{font={footnotesize}}
 \caption{Distance between $\texttt{{\rm{x}}}$ and $\texttt{{\rm{y}}}$.}\label{fig1}
\end{figure}
}
Now we have
\begin{equation}
 \begin{aligned}
   u(\texttt{{\rm{x}}})&=\alpha(n,s)(r^2-|\texttt{{\rm{x}}}|^{2})^{s}\int_{0}^{\pi}\ldots\int_{0}^{2\pi}
   \int_{r}^{\infty} (\rho^{2}-r^{2})^{-s}\frac{g(\rho,\theta,\varphi_{1},\ldots,\varphi_{n-2})}{({\rho^{2}+
   |\texttt{{\rm{x}}}|^{2}-2\rho |\texttt{{\rm{x}}}|\cos\varphi_{1}})^{\frac{n}{2}}}
   \\
  &\quad\times{\rho^{n-1}\sin^{n-2}{(\varphi_{1})}\sin^{n-3}{(\varphi_{2})}\cdots\sin{(\varphi_{n-2})}}{\rm d}\rho\,
  {\rm d}\theta\,{\rm d}\varphi_{1}\cdots{\rm d}\varphi_{n-2}.
  \end{aligned}
\end{equation}
To compute the improper integral,  we perform the substitution
$\rho=\frac{r}{{\rho}'}$ and rename $\rho'$ as $\rho$,
\begin{equation}
 \begin{aligned}
  u(\texttt{{\rm{x}}})&=\alpha(n,s)(r^2-|\texttt{{\rm{x}}}|^2)^{s}r^{n-2s}\int_{0}^{\pi}\ldots\int_{0}^{2\pi}
  \int_{0}^{1} {\rho}^{2s-1}(1-\rho^{2})^{-s} \frac{g(\frac{r}{\rho},\theta,\varphi_{1},\ldots,\varphi_{n-2})}{({r^{2}+{\rho}^2|\texttt{{\rm{x}}}|^{2}-2r\rho |\texttt{{\rm{x}}}|\cos\varphi_{1}})^{\frac{n}{2}}}
  \\
  &\quad\times{\sin^{n-2}{(\varphi_{1})}\sin^{n-3}{(\varphi_{2})}\cdots\sin{(\varphi_{n-2})}}{\rm d}\rho\,
  {\rm d}\theta\,{\rm d}\varphi_{1}\cdots{\rm d}\varphi_{n-2}.
  \end{aligned}
\end{equation}

When $s\in (0,\frac{1}{2})$, we separate integral into two parts as follows
\begin{equation}\label{disJ1}
 \begin{aligned}
  &u(\texttt{{\rm{x}}}) =\alpha(n,s)(r^2-|\texttt{{\rm{x}}}|^2)^{s}r^{n-2s}\int_{0}^{\pi}\ldots{\int_{0}^{2\pi}}
  \left({\int_{0}^{\frac{1}{2}}}+{\int_{\frac{1}{2}}^{1}}\right){\rho}^{2s-1}(1-\rho^{2})^{-s}
  \\
  &\times\frac{g(\frac{r}{\rho},\theta,\varphi_{1},\ldots,\varphi_{n-2})}{({r^{2}+{\rho}^2|\texttt{{\rm{x}}}|^{2}-2r\rho |\texttt{{\rm{x}}}|\cos\varphi_{1}})^{\frac{n}{2}}}
  {\sin^{n-2}{(\varphi_{1})}\sin^{n-3}{(\varphi_{2})}\cdots\sin{(\varphi_{n-2})}}{\rm d}\rho\,{\rm d}\theta\,
  {\rm d}\varphi_{1}\cdots{\rm d}\varphi_{n-2}
  \\
  &=c(n,s)(r^2-|\texttt{{\rm{x}}}|^2)^{s}r^{n-2s}\big[I_{1}(\texttt{{\rm{x}}})+I_{2}(\texttt{{\rm{x}}})\big].
 \end{aligned}
\end{equation}
 Through change of variable $\rho=\frac{{\rho}'}{2}$ and $\theta=2\pi{\theta}'$, $\varphi_{1}=\pi\varphi'_{1}$,...,
  $\varphi_{n-2}=\pi\varphi'_{n-2}$, $I_{1}(\texttt{{\rm{x}}})$ can be rewritten as
\begin{equation}\label{I1forJ1}
 \begin{aligned}
  I_{1}(\texttt{{\rm{x}}})&={\pi}^{n-1}\int_{0}^{1}\int_{0}^{1}\ldots\int_{0}^{1}\left(\frac{{\rho}'}{2}\right)^{2s-1}
  \left[1-\left(\frac{{\rho}'}{2}\right)^{2}\right]^{-s}\frac{g(\frac{2r}{\rho'},2\pi\theta',\pi\varphi'_{1},\ldots,
  \pi\varphi'_{n-2})}{({r^{2}+({\frac{\rho'}{2}})^2|\texttt{{\rm{x}}}|^{2}-r\rho'|\texttt{{\rm{x}}}|\cos(\pi\varphi_{1}')})^{\frac{n}{2}}}
  \\
  &\quad\times{\sin^{n-2}{(\pi\varphi'_{1})}\sin^{n-3}{(\pi\varphi'_{2})}\cdots\sin{(\pi\varphi'_{n-2})}}{\rm d}\rho'\,{\rm d}\theta'\,{\rm d}\varphi'_{1}\cdots{\rm d}\varphi'_{n-2}
  \\
  &={\pi}^{n-1}\int_{0}^{1}\int_{0}^{1}\ldots\int_{0}^{1}\left(\frac{{\rho}}{2}\right)^{2s-1}\omega_{1}(\texttt{{\rm{x}}},
  \rho,\theta,\varphi_{1},\ldots,\varphi_{n-2}){\rm d}\rho\,{\rm d}\theta\,{\rm d}\varphi_{1}\cdots{\rm d}\varphi_{n-2},
 \end{aligned}
\end{equation}
where
\begin{equation}
 \begin{aligned}
  \omega_{1}(\texttt{{\rm{x}}},\rho,\theta,\varphi_{1},\ldots \varphi_{n-2})&=\left[1-\left( \frac{\rho}{2}\right)^2\right]^{-s}\frac{g(\frac{2r}{\rho},2\pi\theta,\pi\varphi_{1},\ldots,\pi\varphi_{n-2})}{({r^{2}+
  {(\frac{\rho}{2}})^2|\texttt{{\rm{x}}}|^{2}-r\rho|\texttt{{\rm{x}}}|\cos\pi\varphi_{1}})^{\frac{n}{2}}}
  \\
  &\quad \times
  {\sin^{n-2}{(\pi\varphi_{1})}\sin^{n-3}{(\pi\varphi_{2})}\cdots\sin{(\pi\varphi_{n-2})}}
 \end{aligned}
\end{equation}
  By the affine transformations $\rho=\frac{1}{2}\rho'+\frac{1}{2}$ and $\theta=2\pi{\theta}'$, $\varphi_{1}=\pi\varphi'_{1}$,..., $\varphi_{n-2}=\pi\varphi'_{n-2}$, $I_{2}(\texttt{{\rm{x}}})$ is given by
\begin{equation}\label{I2forJ1}
 \begin{aligned}
  I_{2}(\texttt{{\rm{x}}})&={\pi}^{n-1}\int_{0}^{1}\ldots\int_{0}^{1}\left[1-\left(\frac{{\rho'}+1}{2}\right)^{2}\right]^{-s}
  \frac{g(\frac{2r}{\rho'+1},2\pi\theta',\pi\varphi'_{1},\ldots,\pi\varphi'_{n-2})}{({r^{2}+({\frac{\rho'+1}{2}})^2
  |\texttt{{\rm{x}}}|^{2}-r(\rho'+1)|\texttt{{\rm{x}}}|\cos\pi\varphi'_{1}})^{\frac{n}{2}}}
  \\
  &\times \left(\frac{{\rho}'+1}{2}\right)^{2s-1}{\sin^{n-2}{(\pi\varphi'_{1})}\sin^{n-3}{(\pi\varphi'_{2})}\cdots
  \sin{(\pi\varphi'_{n-2})}}{\rm d}\rho'\,{\rm d}\theta'\,{\rm d}\varphi'_{1}\cdots{\rm d}\varphi'_{n-2}
  \\
  &={\pi}^{n-1}\int_{0}^{1}\ldots\int_{0}^{1}\left(\frac{1-{\rho}}{2}\right)^{-s}\omega_{2}(\texttt{{\rm{x}}},\rho,\theta,
  \varphi_{1},\ldots \varphi_{n-2}){\rm d}\rho\,{\rm d}\theta\,{\rm d}\varphi_{1}\cdots{\rm d}\varphi_{n-2},
 \end{aligned}
\end{equation}
where
\begin{equation}
 \begin{aligned}
  \omega_{2}(\texttt{{\rm{x}}},\rho,\theta,\varphi_{1},\ldots \varphi_{n-2})&=\left(\frac{3+\rho}{2}\right)^{-s}\left(\frac{\rho+1}{2}\right)^{2s-1}\frac{g(\frac{2r}{\rho+1},
  2\pi\theta,\pi\varphi_{1},\ldots,\pi\varphi_{n-2})}{({r^{2}+({\frac{\rho+1}{2}})^2|\texttt{{\rm{x}}}|^{2}-
  r(\rho+1)|\texttt{{\rm{x}}}|\cos\pi\varphi_{1}})^{\frac{n}{2}}}
  \\
  &\quad \times
  {\sin^{n-2}{(\pi\varphi_{1})}\sin^{n-3}{(\pi\varphi_{2})}\cdots\sin{(\pi\varphi_{n-2})}}.
 \end{aligned}
\end{equation}

We set the uniform grid $\{\rho_{i}\}_{i=1}^{N}$ for $\rho$, $\{\theta_{j}\}_{j=1}^{M}$ for $\theta$,
$\{(\varphi_{m})_{k_{m}}\}_{{k_{m}=1}}^{M_{m}}$ for $\varphi_{m}$, $m=1,2,\ldots, n-2$, and $\{t_{k}\}_{k=1}^{L}$
for $t$. Here $\rho_{i}=ih_\rho$, $\theta_{j}=jh_{\theta}$, $(\varphi_{m})_{k_{m}}=k_{m}h_{m}$, and $t_{k}=kh_{t}$
with $h_{\rho}=\frac{1}{N}$, $h_{\theta}=\frac{1}{M}$, $h_{m}=\frac{1}{M_{m}}$, and $h_{t}=\frac{1}{L}$. We also
define the interpolation operator by recursive formula,
\begin{flalign}\label{intepolation operator}
 \begin{aligned}
 &I_{[\rho_{i-1},\rho_{i}]}v(\rho,\theta,\varphi_{1},\ldots, \varphi_{n-2})
 \\
& =\frac{1}{h_{\rho}}\left[(\rho_{i}-\rho)v(\rho_{i-1},\theta,\varphi_{1},\ldots, \varphi_{n-2})+(\rho-
\rho_{i-1})v(\rho_{i},\theta,\varphi_{1},\ldots, \varphi_{n-2})\right],
 \\
 &I_{[\rho_{i-1},\rho_{i}]\times[\theta_{j-1},\theta_{j}]}v(\rho,\theta,\varphi_{1},\ldots, \varphi_{n-2})=I_{[\theta_{j-1},\theta_{j}]}I_{[\rho_{i-1},\rho_{i}]}v(\rho,\theta,\varphi_{1},\ldots, \varphi_{n-2}),
 \\
 &\ldots
 \\
 &I_{[\rho_{i-1},\rho_{i}]\times[\theta_{j-1},\theta_{j}]\times[(\varphi_{1})_{k_{1}-1},(\varphi_{1})_{k_{1}}]\ldots
 \times[(\varphi_{n-2})_{k_{n-2}-1},(\varphi_{n-2})_{k_{n-2}}]}v(\rho,\theta,\varphi_{1},\ldots, \varphi_{n-2})
 \\
 &=I_{[(\varphi_{n-2})_{k_{n-2}-1},(\varphi_{n-2})_{k_{n-2}}]}I_{[\rho_{i-1},\rho_{i}]\times[\theta_{j-1},\theta_{j}]\ldots
 \times[(\varphi_{n-3})_{k_{n-3}-1},(\varphi_{n-3})_{k_{n-3}}]}
 v(\rho,\theta,\varphi_{1},\ldots, \varphi_{n-2}).
 \end{aligned}
\end{flalign}
Then utilizing the interpolation operator to approximate $\omega_{1}(\texttt{{\rm{x}}},\rho,\theta,\varphi_{1},\ldots,
\varphi_{n-2})$ on each interval yields
\begin{equation}
\begin{aligned}
  I_{1}(\texttt{{\rm{x}}})&\approx \pi^{n-1}\sum_{i=1}^{N}\sum_{j=1}^{M}\sum_{k_{1}=1}^{M_{1}}\ldots\sum_{k_{n-2}=1}^{M_{n-2}}
  \int_{(\varphi_{n-2})_{k_{n-2}-1}}^{(\varphi_{n-2})_{k_{n-2}}}
  \ldots\int_{(\varphi_1)_{k_{1}-1}}^{(\varphi_1)_{k_{1}}}\int_{\theta_{j-1}}^{\theta_{j}}\int_{\rho_{i-1}}^{\rho_{i}}
  \left(\frac{\rho}{2}\right)^{2s-1}
  \\
  &\times I_{[\rho_{i-1},\rho_{i}]\times[\theta_{j-1},\theta_{j}]\ldots\times[((\varphi_{n-2})_{k_{n-2}-1},(\varphi_{n-2})_{k_{n-2}}]}
  [\omega_{1}(\texttt{{\rm{x}}},\rho,\theta,\varphi_{1},\ldots, \varphi_{n-2})]{\rm d}\rho\,{\rm d}\theta\,
  {\rm d}\varphi_{1}\cdots{\rm d}\varphi_{n-2}
  \\
  &=\left(\frac{\pi}{2}\right)^{n-1}\frac{h_{\theta}\prod_{m=1}^{n-2}h_{m}}{h_{\rho}}\sum_{i=1}^{N}\sum_{j=1}^{M}
  \sum_{k_{1}=1}^{M_{1}}\ldots\sum_{k_{n-2}=1}^{M_{n-2}}
  \Bigg\{\sum_{j'=j-1}^{j}\sum_{k_{1}'=k_{1}-1}^{k_{1}}\ldots\sum_{k_{n-2}'=k_{n-2}-1}^{k_{n-2}}
  \\
  &[A_1(i)\omega_{1}(\texttt{{\rm{x}}},\rho_{i-1},\theta_{j'},(\varphi_1)_{k'_{1}},\ldots, (\varphi_{n-2})_{k'_{n-2}})+A_2(i)\omega_{1}(\texttt{{\rm{x}}},\rho_{i},\theta_{j'},(\varphi_1)_{k'_{1}},\ldots,
  (\varphi_{n-2})_{k'_{n-2}})]\Bigg\}
  \\
  &{\triangleq} I_{1}^{h}(\texttt{{\rm{x}}}),
\end{aligned}
\end{equation}
where
\begin{equation}
 \left\{
 \begin{aligned}
  A_{1}(i)&=2^{1-2s}\left[\frac{\rho_{i}}{2s}(\rho_{i}^{2s}-\rho_{i-1}^{2s})
  -\frac{1}{2s+1}(\rho_{i}^{2s+1}-\rho_{i-1}^{2s+1})\right],
  \\
  A_{2}(i)&=2^{1-2s}\left[\frac{1}{2s+1}(\rho_{i}^{2s+1}-\rho_{i-1}^{2s+1})
+\frac{\rho_{i-1}}{2s}(\rho_{i}^{2s}-\rho_{i-1}^{2s})\right].
\\
 \end{aligned}
 \right.
\end{equation}

\noindent Similarly, $I_{2}(\texttt{{\rm{x}}})$ can be approximated as follows,
\begin{equation}
\begin{aligned}
 I_{2}(\texttt{{\rm{x}}})&\approx \pi^{n-1}\sum_{i=1}^{N}\sum_{j=1}^{M}\sum_{k_{1}=1}^{M_{1}}\ldots\sum_{k_{n-2}=1}^{M_{n-2}}
 \int_{(\varphi_{n-2})_{k_{n-2}-1}}^{(\varphi_{n-2})_{k_{n-2}}}
 \ldots\int_{(\varphi_1)_{k_{1}-1}}^{(\varphi_1)_{k_{1}}}\int_{\theta_{j-1}}^{\theta_{j}}
 \int_{\rho_{i-1}}^{\rho_{i}}\left(\frac{1-\rho}{2}\right)^{-s}
 \\
  &\times I_{[\rho_{i-1},\rho_{i}]\times[\theta_{j-1},\theta_{j}]\ldots\times[(\varphi_{n-2})_{k_{n-2}-1},(\varphi_{n-2})_{k_{n-2}}]}
  [\omega_{2}(\texttt{{\rm{x}}},\rho,\theta,\varphi_{1},\ldots, \varphi_{n-2})]{\rm d}\rho\,{\rm d}\theta\,
  {\rm d}\varphi_{1}\cdots{\rm d}\varphi_{n-2}
  \\
  &=\left(\frac{\pi}{2}\right)^{n-1}\frac{h_{\theta}\prod_{m=1}^{n-2}h_{m}}{h_{\rho}}\sum_{i=1}^{N}\sum_{j=1}^{M}
  \sum_{k_{1}=1}^{M_{1}}\ldots\sum_{k_{n-2}=1}^{M_{n-2}}
  \Bigg\{\sum_{j'=j-1}^{j}\sum_{k_{1}'=k_{1}-1}^{k_{1}}\ldots\sum_{k_{n-2}'=k_{n-2}-1}^{k_{n-2}}
  \\
  &[B_1(i)\omega_{2}(\texttt{{\rm{x}}},\rho_{i-1},\theta_{j'},(\varphi_{1})_{k'_{1}},\ldots, (\varphi_{n-2})_{k'_{n-2}})
  +B_2(i)\omega_{2}(\texttt{{\rm{x}}},\rho_{i},\theta_{j'},(\varphi_{1})_{k'_{1}},\ldots, (\varphi_{n-2})_{k'_{n-2}})]\Bigg\}
  \\
 &{\triangleq} I_{2}^{h}(x),
\end{aligned}
\end{equation}
where
\begin{equation}\label{coefficient of B}
 \left\{
 \begin{aligned}
 B_{1}(i)=2^s \bigg\{\frac{1-\rho_{i}}{1-s}&\big[(1-\rho_{i})^{1-s}-(1-\rho_{i-1})^{1-s}\big]
 \\
 &+\frac{1}{2-s}\big[(1-\rho_{i-1})^{2-s}-(1-\rho_{i})^{2-s}\big]\bigg\},
 \\
 B_{2}(i)=2^s \bigg\{\frac{1}{2-s}&\big[(1-\rho_{i})^{2-s}-(1-\rho_{i-1})^{2-s}\big]
 \\
 &+\frac{1-\rho_{i-1}}{1-s}\big[(1-\rho_{i})^{1-s}-(1-\rho_{i-1})^{1-s}\big]\bigg\}.
 \end{aligned}
 \right.
\end{equation}
Then we derive {Scheme I for the} approximation of the solution $u(\texttt{{\rm{x}}})$ when $s\in(0,\frac{1}{2})$,
\begin{equation}\label{righthandside1}
 u(\texttt{{\rm{x}}})\approx c(n,s)(r^2-|\texttt{{\rm{x}}}|^2)^{s}r^{n-2s}\left[I_{1}^{h}(\texttt{{\rm{x}}})+I_{2}^{h}(\texttt{{\rm{x}}})\right]
 \triangleq u_{h}(\texttt{{\rm{x}}}).
\end{equation}

\noindent When $s\in[\frac{1}{2},1)$, Scheme I is given as follows,
\begin{equation}\label{righthandside2}
 \begin{aligned}
 u(\texttt{{\rm{x}}}) &\approx {\left(\frac{\pi}{2}\right)^{n-1}}\frac{h_{\theta}\prod_{m=1}^{n-2}h_{m}}{h_{\rho}}{2^{-s}}\alpha(n,s)(r^2-
 |\texttt{{\rm{x}}}|^2)^{s}r^{n-2s} \sum_{i=1}^{N}\sum_{j=1}^{M}\sum_{k_{1}=1}^{M_{1}}\ldots\sum_{k_{n-2}=1}^{M_{n-2}}
 \\
 &\Bigg\{\sum_{j'=j-1}^{j}\sum_{k_{1}'=k_{1}-1}^{k_{1}}\ldots\sum_{k_{n-2}'=k_{n-2}-1}^{k_{n-2}}
 B_{1}(i)\big[\omega_{2}(\texttt{{\rm{x}}},2\rho_{i-1}-1,\theta_{j'},(\varphi_{1})_{k_{1}'},\ldots, (\varphi_{n-2})_{k_{n-2}'})
 \big]
 \\
 &+B_{2}(i)\big[\omega_{2}(\texttt{{\rm{x}}},2\rho_{i}-1,\theta_{j'},(\varphi_{1})_{k_{1}'},\ldots, (\varphi_{n-2})_{k_{n-2}'})\big]\Bigg\},
 \end{aligned}
\end{equation}
where $B_{1}(i)$ and $B_{2}(i)$ are defined in equation {(\ref{coefficient of B})}.

\begin{remark}
	{The complexity of the quadrature rule is $\mathcal{O}(NM\prod_{i=1}^{n-2}M_{i})$. Specially, when
	$h=h_{\rho}=h_{\theta}=h_{1}=\ldots=h_{n-2}=\frac{1}{N}$, the complexity is $\mathcal{O}(N^{n})$.}
\end{remark}

When the equation has nonconstant source term with homogeneous boundary value, i.e. $f(\texttt{{\rm{x}}})\neq0$ and
$g(\texttt{{\rm{x}}})=0$ in equation (\ref{eq:possionball}), {we take the two dimensional case as an example.
It follows from
\begin{equation}
\begin{aligned}
  u(\texttt{{\rm{x}}})&=\int_{\mathbb{B}_{r}}G(\texttt{{\rm{x}}},\texttt{{\rm{y}}})f(\texttt{{\rm{y}}}){\rm d}\texttt{{\rm{y}}}
=\kappa(n,s)r^{n-2s}(r^2-|\texttt{{\rm{x}}}|^2)^{s}
\\
&\times\int_{0}^{1}\left(\int_{\mathbb{B}_{r}\setminus{S_{h}}}+\int_{S_{h}}\right)(r^2-|\texttt{{\rm{y}}}|^2)^{s}
[(r^2-|\texttt{{\rm{x}}}|^2)(r^2-{|\texttt{{\rm{y}}}|^2})t+r^2|\texttt{{\rm{x}}}
-\texttt{{\rm{y}}}|^2)]^{{-1}}t^{s-1}f(\texttt{{\rm{y}}}){\rm d}t{\rm d}\texttt{{\rm{y}}}
\end{aligned}
\end{equation}
in two dimensions that for $\texttt{{\rm{x}}} = (|\texttt{{\rm{x}}}|,0)\neq(0,0)$
\begin{equation}
 \begin{aligned}
  u(\texttt{{\rm{x}}})&\approx\kappa(n,s)r^{n-2s}(r^2-|\texttt{{\rm{x}}}|^2)^{s}\int_{0}^{1}t^{s-1}
  \Big(\int_{\phi}^{2\pi-\phi}\int_{0}^{r}+\int_{-\phi}^{\phi}\big(\int_{0}^{\frac{|\texttt{{\rm{x}}}|-h}{\cos\theta}}
  +\int_{\frac{|\texttt{{\rm{x}}}|+h}{\cos\theta}}^{r}\big)\Big)\omega_{31}(\rho,\theta,t){\rm d}t{\rm d}\rho{\rm d}\theta
  \\
  &+\int_{0}^{1}t^{s-1}\Big(\int_{0}^{\phi}\int_{0}^{\frac{2h}{\cos\theta}}\omega_{32}(\rho,\theta,t){\rm d}\rho{\rm d}\theta+\int_{-\phi}^{0}\int_{0}^{\frac{2h}{\cos\theta}}\omega_{33}(\rho,\theta,t){\rm d}\rho{\rm d}\theta\Big){\rm d}t.
 \end{aligned}
\end{equation}
Here {$S_{h}$ is a square with side length $h$ centered at \texttt{{\rm{x}}}, $\phi=\arctan\frac{h}{|\texttt{{\rm{x}}}|-h}$ and $\omega_{31},\omega_{32},\omega_{33}$ are
integrand in the corresponding fields of Figure \ref{fig2}},
\begin{equation}
 \begin{aligned}
\omega_{31}(\rho,\theta,t)&=(r^2-\rho^2)^{2s}\frac{f(\rho\cos\theta,\rho\sin\theta)}{(r^2-|\texttt{{\rm{x}}}|^2)
(r^2-\rho^2)t+r^2[(\rho\cos\theta-|\texttt{{\rm{x}}}|)^2+(\rho\sin\theta)^2]},
\\
\omega_{32}(\rho,\theta,t)&=(r^2-\rho^2)^{2s}\frac{f(\rho\cos\theta+|\texttt{{\rm{x}}}|-h,\rho\sin\theta+h)}
{(r^2-|\texttt{{\rm{x}}}|^2)(r^2-\rho^2)t+r^2[(\rho\cos\theta-h)^2+(\rho\sin\theta+h)^2]},
\\
\omega_{33}(\rho,\theta,t)&=(r^2-\rho^2)^{2s}\frac{f(\rho\cos\theta+|\texttt{{\rm{x}}}|-h,\rho\sin\theta-h)}
{(r^2-|\texttt{{\rm{x}}}|^2)(r^2-\rho^2)t+r^2[({\rho}\cos\theta-h)^2+(\rho\sin\theta-h)^2]}.
\\
 \end{aligned}
\end{equation}
\begin{figure}[!htbp]
    \centering
    \includegraphics[height=6cm,width=6cm]{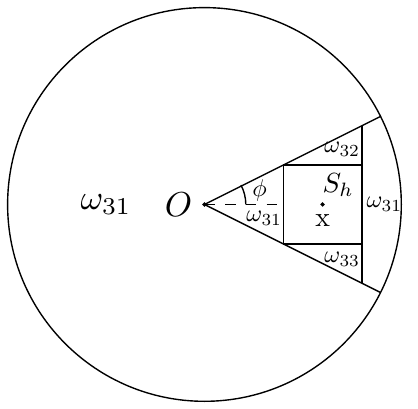}
 \captionsetup{font={footnotesize}}
 \caption{Two dimensional case.}\label{fig2}

\end{figure}	
Then we have
\begin{equation}\label{boundarycondition1}
 \begin{aligned}
  u_{h}(\texttt{{\rm{x}}})&=\kappa(2,s)r^{2-2s}(r^2-|\texttt{{\rm{x}}}|^2)^{s}\sum_{i=1}^{N}\sum_{j=1}^{M}\sum_{k=1}^{L}
  C(k)\Bigg[(2\pi-2\phi)r\omega_{31}(r\rho_{i-\frac{1}{2}},(2\pi-2\phi)\theta_{j-\frac{1}{2}},t_{k-\frac{1}{2}})
  \\
  &+\frac{2\phi(|\texttt{{\rm{x}}}|-h)}{\cos(2\phi\theta_{j-\frac{1}{2}}-\phi)}\omega_{31}\left(\frac{(|\texttt{{\rm{x}}}|-h)
  \rho_{i-\frac{1}{2}}}{\cos(2\phi\theta_{j-\frac{1}{2}}-\phi)},2\phi\theta_{j-\frac{1}{2}-\phi},t_{k-\frac{1}{2}}\right)+
  2\phi\Big(r-\frac{|\texttt{{\rm{x}}}|+h}{\cos(2\phi{\theta_{j-\frac{1}{2}}}-\phi)}\Big)
  \\
  &\times\omega_{31}\left(\Big(r-\frac{|\texttt{{\rm{x}}}|+h}{\cos(2\phi{\theta_{j-\frac{1}{2}}}-\phi)}\Big)\rho_{i-
  \frac{1}{2}}+\frac{|\texttt{{\rm{x}}}|+h}{\cos(2\phi{\theta_{j-\frac{1}{2}}}-\phi)},2\phi\theta_{\frac{1}{2}}-\phi,t_{k-\frac{1}{2}}\right)
  \\
  &+\frac{2h\phi}{\cos(\phi\theta_{j-\frac{1}{2}})}\omega_{32}\left(\frac{2h\rho_{i-\frac{1}{2}}}{\cos(\phi\theta_{j-
  \frac{1}{2}})},\phi\theta_{j-\frac{1}{2}},t_{k-\frac{1}{2}}\right)
  \\
  &+\frac{2h\phi}{\cos(\phi\theta_{j-\frac{1}{2}}-\phi)}\omega_{33}\left(\frac{2h\rho_{i-\frac{1}{2}}}{\cos(\phi\theta_{j-
  \frac{1}{2}})},\phi\theta_{j-\frac{1}{2}}-\phi,t_{k-\frac{1}{2}}\right)
\Bigg],
 \end{aligned}
 \end{equation}
where $C(k)=\frac{1}{s}(t_{k}^{s}-t_{k-1}^{s})$.

When $\texttt{{\rm{x}}}=(0,0)$, we have
\begin{equation}\label{boundarycondition2}
 u_{h}(\texttt{{\rm{x}}})=2\pi(r-h)\kappa(2,s)\sum_{i=1}^{N}\sum_{j=1}^{M}\sum_{k=1}^{L}C(k)\omega_{4}((r-h)\rho_{i-
 \frac{1}{2}},2\pi\theta_{j-\frac{1}{2}},t_{k-\frac{1}{2}}),
\end{equation}
where
\begin{equation}
  \omega_{4}(\rho,\theta,t)=(r^2-\rho^2)^{2s}\frac{f(\rho,\theta)}{(r^2-\rho^2)t+\rho^2}.
\end{equation}

For higher-dimensional case, the numerical scheme can be similarly derived so is omitted here.
}

\subsection{Error estimates of the quadrature rules}

In this subsection, we provide error estimates for our numerical method in discretizing the representation
formula of the solution $u(\texttt{{\rm{x}}})$. Firstly we have the following lemma which can be readily derived.
\begin{lemma}\label{lemmaforerror}
 Let I be the interpolation operator defined in (\ref{intepolation operator}) on the domain $\mathbb{K}=[0,\ell_{1}]
 \times[0, \ell_{2}]\times\ldots \times[0,\ell_{n}]$. Then for any functions
$v\in C_{b}^{2}(\mathbb{K})$, we have the error estimate,
 \begin{equation}
  \|v-Iv\|_{L^{\infty}}\leq \frac{1}{8}\left(\sum_{i=1}^{n}\ell_{i}^{2}\left\|\frac{\partial^{2}v}{\partial{x}_{i}^{2}}
  \right\|_{L^{\infty}}\right).
 \end{equation}
\end{lemma}

 In $n$ $(n\geq 2)$ dimensional cases, suppose $g(x)\in C_{b}^{2}(\mathbb{R}^{n}\backslash \mathbb{B}_{r})$. Then when
 $s\in (0,\frac{1}{2})$, we obtain
\begin{equation}
 |u(\texttt{{\rm{x}}})-u_{h}(\texttt{{\rm{x}}})|\leq c(n,s)(r^2-|\texttt{{\rm{x}}}|^{2})^{s}r^{n-2s}[|I_{1}(\texttt{{\rm{x}}})-I_{1}^{h}(\texttt{{\rm{x}}})|+|I_{2}(\texttt{{\rm{x}}})-
 I_{2}^{h}(\texttt{{\rm{x}}})|].
\end{equation}
From Lemma \ref{lemmaforerror} and the fact $\texttt{{\rm{x}}}\neq \texttt{{\rm{y}}}$, we derive
\begin{equation}
 \begin{aligned}
  |I_{1}(\texttt{{\rm{x}}})-I_{1}^{h}(\texttt{{\rm{x}}})|&=\Bigg|\pi^{n-1}\sum_{i=1}^{N}\sum_{j=1}^{M}\ldots
  \sum_{k_{n-2}=1}^{M_{n-2}}\int_{(\varphi_{n-2})_{k_{n-2}-1}}^{(\varphi_{n-2})_{k_{n-2}}}\ldots\int_{\theta_{j-1}}^{\theta_{j}}
  \int_{\rho_{i-1}}^{\rho_{i}}\Big[\omega_{1}(\texttt{{\rm{x}}},\rho,\theta,\varphi_{1},\ldots, \varphi_{n-2})
  \\
  &-I_{[\rho_{i-1},\rho_{i}]\times[\theta_{j-1},\theta_{j}]\ldots\times[(\varphi_{n-2})_{k_{n-2}-1},(\varphi_{n-2})_{k_{n-2}}]}
  \omega_{1}(\texttt{{\rm{x}}},\rho,\theta,\varphi_{1},\ldots, \varphi_{n-2})\Big]
  \\
  &\times\left(\frac{\rho}{2}\right)^{2s-1} {\rm d}\rho\,{\rm d}\theta\,{\rm d}\varphi_{1}\cdots{\rm d}\varphi_{n-2}\Bigg|
  \\
  &\leq \frac{\pi}{s2^{2s+3}}\left({h_{\rho}}^{2}\left\|\frac{\partial^{2}\omega_{1}}{\partial{\rho}^{2}}\right\|_{L^{\infty}}
  +{h_{\theta}^2}\left\|\frac{\partial^{2}\omega_{1}}{\partial{\theta}^{2}}\right\|_{L^{\infty}}+\sum_{r=1}^{n-2}{h_{r}^2}
  \left\|\frac{\partial^{2}\omega_{1}}{\partial{\varphi_{r}^{2}}}\right\|_{L^{\infty}}\right)
  \\
  &\leq C\left(h_{\rho}^{2}+h_{\theta}^{2}+\sum_{r=1}^{n-2}h_{r}^2\right).
 \end{aligned}
\end{equation}
Similarly,
\begin{equation}
  |I_{2}(\texttt{{\rm{x}}})-I_{2}^{h}(\texttt{{\rm{x}}})| \leq C\left(h_{\rho}^{2}+h_{\theta}^{2}+\sum_{r=1}^{n-2}h_{r}^2\right).
\end{equation}
When $s\in [\frac{1}{2},1)$,  the truncated errors is still $\mathcal{O}(h_{\rho}^{2}+h_{\theta}^{2}+\sum_{r=1}^{n-2}h_{r}^2)$.
Thus we have the theorem as follows.
\begin{theorem}
 Let $r>0$, $g \in C_{b}^{2}(\mathbb{R}^{n}\backslash \mathbb{B}_{r})$ and $s\in(0,1)$. Then it holds that
 \begin{equation}
  |u(\texttt{{\rm{x}}})-u_{h}(\texttt{{\rm{x}}})| \leq C\left(h_{\rho}^{2}+h_{\theta}^{2}+\sum_{r=1}^{n-2}h_{r}^2\right),
 \end{equation}
where C is a positive constant.
\end{theorem}
{When $f(\texttt{{\rm{x}}})\in C_{b}^{2}(\mathbb{B}_{r}\backslash S_{h})$, it is easily to derive the error estimates for
equation (\ref{boundarycondition1}), $\mathcal{O}(h^{2s}+h_{t}+h_{\rho}^{2}+h_{\theta}^{2})$.

  Obviously, the above quadrature for $n$-dimensional fractional Laplacian is often difficult to be implemented in computer
  simulations if $n>3$. So the Monte Carlo method is very likely the best choice for numerical experiments. In the following,
  we introduce modified walk-on-sphere method, which is one of Monte Carlo methods.}

\section{Modified walk-on-sphere method} \label{sec:mod-walk-on-spheres}

In this section, we utilize the modified  walk-on-sphere  method to solve equation (\ref{eq:possionproblem})
with $\Omega$ being an arbitrary domain instead of a ball. Assume $f(\texttt{{\rm{x}}})$ and $g(\texttt{{\rm{x}}})$
satisfy the condition in Theorem \ref{integralrepre}. Then it is known from Section 2 that the solution to (\ref{eq:possionproblem}),
$u(\texttt{{\rm{x}}}_{0})$, $\texttt{{\rm{x}}}_{0} \in \Omega$, has the integral representation
\begin{equation}\label{3integralpre}
  u(\texttt{{\rm{x}}}_{0})=\int_{\mathbb{R}^{n} \backslash \mathbb{B}_{r_{0}}(\texttt{{\rm{x}}}_{0})}p_{1,r_{0}}(\texttt{{\rm{x}}},\texttt{{\rm{x}}}_{0})u(\texttt{{\rm{x}}}){\rm d}\texttt{{\rm{x}}}
 +a(\texttt{{\rm{x}}}_{0})\int_{\mathbb{B}_{r_{0}}(\texttt{{\rm{x}}}_{0})}p_{2,r_{0}}(\texttt{{\rm{y}}},\texttt{{\rm{x}}}_{0})f(\texttt{{\rm{y}}}){\rm d}\texttt{{\rm{y}}},
\end{equation}
where $\mathbb{B}_{r_{0}}(\texttt{{\rm{x}}}_{0})$ is the largest ball contained in $\Omega$ with radius
$r_{0}=\sup\{r:\mathbb{B}(\texttt{{\rm{x}}}_{0},r)\subset \Omega\}$, centered at $\texttt{{\rm{x}}}_{0}$,
\begin{equation}
\left\{
 \begin{aligned}
  &p_{1,r_{0}}(\texttt{{\rm{x}}},\texttt{{\rm{x}}}_{0})=P_{r_{0}}(0,\texttt{{\rm{x}}}-\texttt{{\rm{x}}}_{0}),
  \\
  &p_{2,r_{0}}(\texttt{{\rm{y}}},\texttt{{\rm{x}}}_{0})=\frac{1}{a(\texttt{{\rm{x}}}_{0})}G(0,\texttt{{\rm{y}}}-\texttt{{\rm{x}}}_{0}),
  \\
 \end{aligned}
 \right.
\end{equation}
and
\begin{equation}
\left\{
 \begin{aligned}
  a(\texttt{{\rm{x}}}_{0})&=\int_{\mathbb{B}_{r_{0}}(\texttt{{\rm{x}}}_{0})}G(0,\texttt{{\rm{y}}}-
  \texttt{{\rm{x}}}_{0}){\rm d}\texttt{{\rm{y}}}
  \\
  &=\kappa(n,s)\int_{\mathbb{B}_{r_{0}}(\texttt{{\rm{x}}}_{0})}|\texttt{{\rm{y}}}-\texttt{{\rm{x}}}_{0}|^{2s-n}
  \int_{0}^{\frac{r_{0}^{2}-|\texttt{{\rm{y}}}-\texttt{{\rm{x}}}_{0}|^2}{|\texttt{{\rm{y}}}-\texttt{{\rm{x}}}_{0}|^2}}
  \frac{t^{s-1}}{(t+1)^{\frac{n}{2}}}{\rm d}t{\rm d}\texttt{{\rm{y}}}
  \\
&=\omega_{n-1}\kappa(n,s)\int_{0}^{r_{0}}{\rho}^{2s-1}\int_{0}^{\frac{r_{0}^{2}-{\rho}^2}{{\rho}^2}}
\frac{t^{s-1}}{(t+1)^{\frac{n}{2}}}{\rm d}t{\rm d}{\rho}
\\
&=\omega_{n-1}\kappa(n,s)\int_{0}^{\infty}\frac{t^{s-1}}{(t+1)^{\frac{n}{2}}}\int_{0}^{
\frac{r_{0}}{(t+1)^{\frac{1}{2}}}}{\rho}^{2s-1}{\rm d}{\rho}{\rm d}t
\\
&=\kappa(n,s)B(s,\frac{n}{2})\frac{\omega_{n-1}}{2s}r_{0}^{2s}.
 \end{aligned}
 \right.
\end{equation}
Here $a(\texttt{{\rm{x}}}_{0})$ is the normalizing constant such that $\int_{\mathbb{B}_{r_{0}}(\texttt{{\rm{x}}}_{0})}p_{2,r_{0}}(\texttt{{\rm{y}}},\texttt{{\rm{x}}}_{0}){\rm d}\texttt{{\rm{y}}}=1.$
And $\omega_{n-1}$ denotes the $(n-1)$-dimensional measure of the unit sphere {$S^{n-1}$}.
We recall from \cite{Bucur2016} that
\begin{equation}
\int_{\mathbb{R}^{n} \backslash \mathbb{B}_{r_{0}}(\texttt{{\rm{x}}}_{0})}p_{1,r_{0}}(\texttt{{\rm{x}}},
\texttt{{\rm{x}}}_{0}){\rm d}\texttt{{\rm{x}}}=1.
\end{equation}
Both $p_{1,r_{0}}(\texttt{{\rm{x}}},\texttt{{\rm{x}}}_{0})$ and $p_{2,r_{0}}(\texttt{{\rm{y}}},\texttt{{\rm{x}}}_{0})$
can be viewed as probability density function for the random variables $\texttt{{\rm{X}}}$ outside the ball
$\mathbb{B}_{r_{0}}(\texttt{{\rm{x}}}_{0})$ and $\texttt{{\rm{Y}}}$ inside the ball, respectively.
So we can suppose $\texttt{{\rm{X}}}$ and $\texttt{{\rm{Y}}}$ denote random variables outside the ball with density
$p_{1,r_{0}}(\texttt{{\rm{x}}},\texttt{{\rm{x}}}_{0})$ and inside the ball with $p_{2,r_{0}}(\texttt{{\rm{y}}},\texttt{{\rm{x}}}_{0})$,
accordingly. Then the integral representation (\ref{3integralpre}) can be rewritten as
\begin{equation}\label{3probabilitypre}
 u(\texttt{{\rm{x}}}_{0})=\mathbb{E}u(\texttt{{\rm{X}}})+a(\texttt{{\rm{x}}}_{0})\mathbb{E}f(\texttt{{\rm{Y}}}).
\end{equation}
Here $\mathbb{E}$ indicates expected operation. The first term describes a mean value with respect to $p_{1,r_{0}}(\texttt{{\rm{y}}},
\texttt{{\rm{x}}}_{0})$ outside the ball and represents the average score upon exiting the ball $\mathbb{B}_{r_{0}}(\texttt{{\rm{x}}}_{0})$.
The second term is a weighted average taken with respect to density $p_{2,r_{0}}(\texttt{{\rm{x}}},\texttt{{\rm{x}}}_{0})$ and
represents the expected contribution from sources inside the ball.

Both $p_{1,r_{0}}(\texttt{{\rm{x}}},\texttt{{\rm{x}}}_{0})$ and $p_{2,r_{0}}(\texttt{{\rm{y}}},\texttt{{\rm{x}}}_{0})$
can be used to construct transition probabilities for a Markov chain. The transition from an initial point
$\texttt{{\rm{X}}}_{0}=\texttt{{\rm{x}}}_{0}$ is performed by selecting a point $\texttt{{\rm{X}}}_{1}$ outside the ball $\mathbb{B}_{r_{0}}(\texttt{{\rm{x}}}_{0})$ with density $p_{1,r_{0}}(\texttt{{\rm{x}}},\texttt{{\rm{x}}}_{0})$ and by
generating a random variable $\texttt{{\rm{Y}}}_{1}$ with density $p_{2,r_{0}}(\texttt{{\rm{y}}},\texttt{{\rm{x}}}_{0})$.
Given the position $\texttt{{\rm{X}}}_{k}=\texttt{{\rm{x}}}_{k}$ at the $k$-th step, the transition to the $(k+1)$-th step
is carried out by choosing $\texttt{{\rm{X}}}_{k+1}$ with $p_{1,r_{k}}(\texttt{{\rm{x}}},\texttt{{\rm{x}}}_{k})$, $r_{k}=\sup \{r:\mathbb{B}(\texttt{{\rm{x}}}_{k},r)\subset \Omega\}$, outside the ball $\mathbb{B}_{r_{k}}(\texttt{{\rm{x}}}_{k})$ and by
selecting $\texttt{{\rm{Y}}}_{k+1}$ according to the density $p_{2,r_{k}}(\texttt{{\rm{y}}},\texttt{{\rm{x}}}_{k})$($\texttt{{\rm{X}}}_{k+1}$
and $\texttt{{\rm{Y}}}_{k+1}$ are independent of each other). The walk on spheres is simulated by repeating this procedure
until the particle exits the domain $\Omega$. For the point $\texttt{{\rm{X}}}_{k}$, the solution $u(\texttt{{\rm{x}}})$
must satisfy (\ref{3probabilitypre}). We obtain
\begin{equation}\label{3probabilitypre4k}
 u(\texttt{{\rm{X}}}_{k})=\mathbb{E}[u(\texttt{{\rm{X}}}_{k+1})|\texttt{{\rm{X}}}_{k}]+a(\texttt{{\rm{X}}}_{k})
 \mathbb{E}[f(\texttt{{\rm{Y}}}_{k+1})|\texttt{{\rm{X}}}_{k}].
\end{equation}
Here the conditional expectations are used as the densities of  $\texttt{{\rm{X}}}_{k+1}$ and $\texttt{{\rm{Y}}}_{k+1}$ are
determined by the position of $\texttt{{\rm{X}}}_{k}$.

Generally speaking, the connection between the solution to the fractional Laplacian Dirichlet problem and the random process
follows from the telescoping summation
\begin{equation}
 \begin{aligned}
 u(\texttt{{\rm{x}}}_{0})=\mathbb{E}u(\texttt{{\rm{X}}}_{0})&=\mathbb{E}u(\texttt{{\rm{X}}}_{l})+\mathbb{E}
 \left\{\sum_{k=0}^{l-1}[u(\texttt{{\rm{X}}}_{k})-u(\texttt{{\rm{X}}}_{k+1})]\right\}
 \\
 &=\mathbb{E}u(\texttt{{\rm{X}}}_{l})+ \mathbb{E}
\left\{\sum_{k=0}^{l-1}\big(u(\texttt{{\rm{X}}}_{k})-\mathbb{E}[u(\texttt{{\rm{X}}}_{k+1})|\texttt{{\rm{X}}}_{k}]\big)\right\},
 \end{aligned}
\end{equation}
where $\texttt{{\rm{X}}}_{0}=\texttt{{\rm{x}}}_{0}$. In the last equality, we have used the fact that
\begin{equation}
 \mathbb{E}u(\texttt{{\rm{X}}}_{k+1})=\mathbb{E}\{ \mathbb{E}[u(\texttt{{\rm{X}}}_{k+1})|\texttt{{\rm{X}}}_{k}]\}.
\end{equation}
Applying identity (\ref{3probabilitypre4k}) yields
\begin{equation}\label{3intepresolution}
 \begin{aligned}
  u(\texttt{{\rm{x}}}_{0})&=\mathbb{E}u(\texttt{{\rm{X}}}_{l})+\mathbb{E}\left\{\sum_{k=0}^{l-1}a(\texttt{{\rm{X}}}_{k})
  \mathbb{E}[f(\texttt{{\rm{Y}}}_{k+1})|\texttt{{\rm{X}}}_{k}] \right\}
  \\
  &=\mathbb{E}u(\texttt{{\rm{X}}}_{l})+\mathbb{E}\left[\sum_{k=0}^{l-1}a(\texttt{{\rm{X}}}_{k})f(\texttt{{\rm{Y}}}_{k+1})\right].
 \end{aligned}
\end{equation}
Now suppose the process jumps out of the boundary on the $l$-th step. Then all of the terms on the right-hand side of
equation (\ref{3intepresolution}) would be known and $u(\texttt{{\rm{X}}}_{l})$ can be replaced by $g(\texttt{{\rm{X}}}_{l})$.
This suggests that $u(\texttt{{\rm{x}}}_{0})$ be the mean value of the exit points plus a weighted average from internal contribution.

Monte Carlo method makes use of the preceding observation to estimate $u(\texttt{{\rm{x}}}_{0})$. According to the density $p_{1,r_{k}}(\texttt{{\rm{x}}},\texttt{{\rm{x}}}_{k})$, $\texttt{{\rm{X}}}_{k+1}$ will jump out of the ball $B_{r_{k}}(\texttt{{\rm{x}}}_{k})$
so that the particle will exit the domain in a finite number of steps. At the conclusion of each walk, we compute the random sample
\begin{equation}\label{3randomsample}
 Z_{i}=g(\texttt{{\rm{X}}}_{l}^{i})+\sum_{k=0}^{l-1}a(\texttt{{\rm{X}}}_{k}^{i})f(\texttt{{\rm{Y}}}_{k+1}^{i}),
\end{equation}
where $i$ denotes the $i$-th experiment. By identity (\ref{3randomsample}), we have
$u(\texttt{{\rm{x}}}_{0})=\mathbb{E}Z_{i}$. An estimate for the mean of $Z_{i}$ is given by the statistic
\begin{equation}
 S_{N}=\frac{1}{N}\sum_{i=1}^{N}Z_{i},
\end{equation}
where $N$ is the number of trials. By the law of large numbers,
\begin{equation}
 \lim\limits_{N \rightarrow \infty}S_{N}=u(\texttt{{\rm{x}}}_{0}).
\end{equation}
The central limit theorem gives $\mathcal{O}(1/N)$ upper bounds on the variance of
the $N$-term sum, which serves as a numerical error estimate.

\subsection{Modifying the walk-on-sphere method via approximate sampling methods}
  For the Monte Carlo sampling, we need to sample $\texttt{{\rm{X}}}_{k+1}$ and $\texttt{{\rm{Y}}}_{k+1}$ according
to their probability density functions. Given the position $\texttt{{\rm{X}}}_{k}=\texttt{{\rm{x}}}_{k}=(x_{k,1},x_{k,2},\ldots,x_{k,n})$,
we obtain { probability measure for $\texttt{{\rm{X}}}_{k+1}$},
\begin{equation}{\label{eq:pdf4x}}
 \begin{aligned}
  \mathbb{P}_{\texttt{{\rm{x}}}_{k}}(\texttt{{\rm{X}}}_{k+1}\in {\rm d}\texttt{{\rm{x}}})&=p_{1,r_{k}}(\texttt{{\rm{x}}},
  \texttt{{\rm{x}}}_{k}){\rm d}\texttt{{\rm{x}}}
  =P_{r_{k}}(0,\texttt{{\rm{x}}}-\texttt{{\rm{x}}}_{k}){\rm d}\texttt{{\rm{x}}}, \quad \texttt{{\rm{x}}}\in
  \mathbb{R}^{n}\backslash \mathbb{B}_{r_{k}}(\texttt{{\rm{x}}}_{k})\\
    &=c(n,s)\left(\frac{r_{k}^{2}}{|\texttt{{\rm{x}}}-\texttt{{\rm{x}}}_{k}|^{2}-r_{k}^{2}}\right)^{s}
    \frac{1}{|\texttt{{\rm{x}}}-\texttt{{\rm{x}}}_{k}|^{n}}{\rm d}\texttt{{\rm{x}}}.
 \end{aligned}
\end{equation}
We change variables by using the hyperspherical coordinates with radius $\rho >r_{k}$, angles
$\varphi_{1},\varphi_{2},\cdots ,\varphi_{n-2} \in [0,\pi]$, and $\theta \in [0,2\pi]$. In this case, it holds that
\begin{equation}\label{3hypercoordinate}
 \left\{
 \begin{aligned}
  &x_{1} =x_{k,1}+\rho\sin{\varphi_{1}}\sin{\varphi_{2}}\cdots\sin{\varphi_{n-2}}\sin{\theta},
  \\
  &x_{2} =x_{k,2}+\rho\sin{\varphi_{1}}\sin{\varphi_{2}}\cdots\sin{\varphi_{n-2}}\cos{\theta},
  \\
  &x_{3} =x_{k,3}+\rho\sin{\varphi_{1}}\sin{\varphi_{2}}\cdots\cos{\varphi_{n-2}},
  \\
  &\cdots
  \\
  &x_{n-1}=x_{k,n-1}+\rho\sin{\varphi_{1}}\cos{\varphi_{2}},
  \\
  &x_{n} =x_{k,n}+\rho\cos{\varphi_{1}}.
 \end{aligned}
 \right.
\end{equation}
The Jacobian of the transformation is given by
$\rho^{n-1}\sin^{n-2}{\varphi_{1}}\sin^{n-3}{\varphi_{2}}\cdots\sin{\varphi_{n-2}}$.
Then we derive
\begin{equation}\label{3densityfuntion}
 \begin{aligned}
  &\mathbb{P}_{\texttt{{\rm{x}}}_{k}}(\texttt{{\rm{X}}}_{k+1}\in {\rm d}\texttt{{\rm{x}}})
  \\
  &=2\frac{\pi^{n/2}}{\Gamma(n/2)}c(n,s)r_{k}^{2s}(\rho^{2}-r_{k}^{2})^{-s}\frac{{\rm d}\rho}{\rho}
  \times \frac{{\rm d}\theta}{2\pi}\times\frac{\sin^{n-2}\varphi_{1}}{I_{n-2}}{\rm d}{\varphi_{1}}\times
  \cdots\times \frac{\sin\varphi_{n-2}}{I_{1}}{\rm d}{\varphi_{n-2}},
 \end{aligned}
\end{equation}
where $\rho \geq r_{k}$,
\begin{equation}\label{4In}
 I_{n}=\int_{0}^{\pi}\sin^{n}\varphi{\rm d}\varphi=\left\{
  \begin{aligned}
   &\frac{n-1}{n}\frac{n-3}{n-2}\cdots \frac{1}{2}I_{0},\qquad  \,n \, {\rm is \, even},
   \\
   &\frac{n-1}{n}\frac{n-3}{n-2}\cdots \frac{1}{2}I_{1},\qquad  \, n \,{\rm is \, odd},
  \end{aligned}
  \right.
\end{equation}
with $I_{0}=\pi$, $I_{1}=2$, and
\begin{equation}
  \prod_{k=1}^{n-2}\int_{0}^{\pi}\sin^{k}{\varphi}{\rm d}\varphi
 =I_{1}I_{2}\cdots I_{n-2}=\frac{\pi^{n/2-1}}{\Gamma(n/2)}.
\end{equation}

  From (\ref{3densityfuntion}), we observe that $\theta$ is sampled uniformly on $[0,2\pi]$,
whereas we can sample the radius $\rho$ via the inverse transform sampling method \cite{Kyprianou&Osojnik2018}.
For $U \sim ([0,1])$,
\begin{equation}
 \rho=F^{-1}(U)=r_{k}(I^{-1}(1-U;s,1-s))^{-1/2},
\end{equation}
where $I(x;z,w)$, { $z,w>0$,} is the incomplete beta function
\begin{equation}
 I(x;z,w)=\frac{1}{B(z,w)}\int_{0}^{x}u^{z-1}(1-u)^{w-1}{\rm d}u,\quad x\in[0,1].
\end{equation}

For $\varphi_{1},\varphi_{2},\cdots,\varphi_{n-2}$, we use again the inverse transform method to simulate them.
Denote $I_{n}^{\ast}(\phi)=\int_{0}^{\phi}\sin^{n}\varphi{\rm d}\varphi$ and $I_{n}(\phi)=\frac{1}{I_{n}}
I_{n}^{\ast}(\phi)$, $\phi\in[0,\pi]$. Then we have
\begin{equation}\label{4Inx}
 I_{n}(\phi)=\frac{1}{I_{n}}\left\{
  \begin{aligned}
  &-\frac{1}{n}\sin^{n-1}{\phi}\cos{\phi}+\sum_{i=1}^{n/2-1}-\frac{1}{n-2i}\sin^{n-2i-1}\phi\cos{\phi}
  \prod_{j=0}^{i-1}\frac{n-2j-1}{n-2j}
  \\
  &+\frac{(n-1)!!}{n!!}I_{0}^{*}(\phi), \qquad \qquad \qquad \qquad \qquad \qquad \qquad n\, {\rm is\, even},
  \\
  &-\frac{1}{n}\sin^{n-1}{\phi}\cos{\phi}+\sum_{i=1}^{(n-1)/2-1}-\frac{1}{n-2i}\sin^{n-2i-1}\phi\cos{\phi}
  \prod_{j=0}^{i-1}\frac{n-2j-1}{n-2j}
  \\
  &+\frac{(n-1)!!}{n!!}I_{1}^{*}(\phi), \qquad \qquad \qquad \qquad \qquad \qquad \qquad n\, {\rm is\, odd},
  \end{aligned}
  \right.
\end{equation}
where $I_{0}^{*}(\phi)=\phi$, $I_{1}^{*}(\phi)=1-\cos{\phi}$. For random number $U\sim([0,1])$, we have
\begin{equation}\label{3simulatetheta}
 \varphi_{i}=I_{n-i-1}^{-1}(U).
\end{equation}
\noindent We can readily get $\varphi_{n-2}=\arccos(1-2U)$, $U \sim ([0,1])$. When $i \neq n-2$,
it is complicated to get the inverse density function and we use rejection sampling method to generate samples
$\varphi_{n-m-1}$ from {a target PDF $\frac{1}{I_{m}}p(x)=\frac{1}{I_{m}}\sin^{m}{x}$, $m=2,\ldots ,n-2$.}
The standard RS algorithm \cite{Luca&David2019} allows us to draw samples exactly from the target PDF {$p_{0}(x)$}.

\begin{algo}[\cite{Luca&David2019}]
	 Step 1. Choose an alternative simpler proposal PDF $q_{0}(x)$.
	 \\
	 Step 2. Draw $x'\sim q_{0}(x)$ and $\omega'\sim U([0,1])$.
	 \\
	Step  3. If $\omega'\leq \frac{p_{0}(x')}{Kq_{0}(x')}$, then $x'$ is accepted. Otherwise, $x'$ is discarded.
	 \\
	 Step  4. Repeated steps 2-3 until the desired number of samples has been obtained.
\end{algo}

 Since the target PDF {for $\varphi_{n-m-1}$ $p_{0}(x)=\frac{1}{I_{m}}p_{m}(x)=\frac{1}{I_{m}}\sin^{m}{x}$, $m=2,\ldots ,n-2$} is unimodal and symmetric, the proposal PDF is given by truncated Gaussian density,
 {\begin{equation}\label{truncated Gaussion density}
  q_{0}(x)=C_{q,m}q_{m}(x)=C_{q,m}\beta\exp(-\alpha_{m}(x-\mu)^2),\quad x \in (0,\pi),
 \end{equation}
 where the determination of parameters $\alpha_{m},\beta,\mu$ and the normalizing constant $C_{q,m}$ is in such a way
 that we obtain  a proposal function for applying the rejection sampling method, i.e. $K=I_{m}/C_{q,m}$, such that
 \begin{equation}\label{PDFinequality}
  q_{m}(x)\geq p_{m}(x)
 \end{equation}}
 It is easy to get $\beta=1,\mu=\frac{\pi}{2}$. For parameter $\alpha_{m}$, noting that equation (\ref{PDFinequality})
 implies
 {\begin{equation}
  \ln{q_{m}(x)}=-\alpha_{m}(x-\frac{\pi}{2})^2\geq m\ln{\sin(x)},
 \end{equation}
 so
 \begin{equation}
  \alpha_{m} \leq -\frac{m\ln{\sin(x)}}{(x-\frac{\pi}{2})^2}, \quad x\in(0,\pi).
 \end{equation}
 In order to obtain best positive fit between the proposal and the target, we must set $\alpha_{m}=\lim_{x\rightarrow\frac{\pi}{2}}\frac{-m\ln{\sin(x)}}{(x-\frac{\pi}{2})^2}=\frac{m}{2}$ (see Figure \ref{3fig1}).}
 Thus, samples must be drawn from the selected proposal PDF, $C_{q,m}e^{-\frac{m}{2}(x-\frac{\pi}{2})^2}$, $x \in(0,\pi)$.
 For the truncated Gaussians, the technique available in the literature \cite{Chopin2009} allows us to draw samples efficiently. Finally, the acceptance rate which is the key performance measure for rejection sampling method is
 \begin{equation}
  \eta=\int_{0}^{\pi}\frac{p_{m}(x)}{q_{m}(x)}q_{0}(x){\rm d}x=\frac{I_{m}}{\sqrt{\frac{2\pi}{m}}\,{\rm erf}\left(\sqrt{\frac{2m}{4}}\pi\right)},
 \end{equation}
 where ${\rm erf}(x)$ denotes the error function. It is obvious that $\eta$ is a monotonically increasing function
 converging to $1$ with respect to $m$ and the acceptance rate is more than 91\% (see Figure \ref{3fig1}).
\begin{figure}[!htbp]
    \begin{minipage}[t]{0.5\linewidth}
    \centering
    \includegraphics[height=6cm,width=8cm]{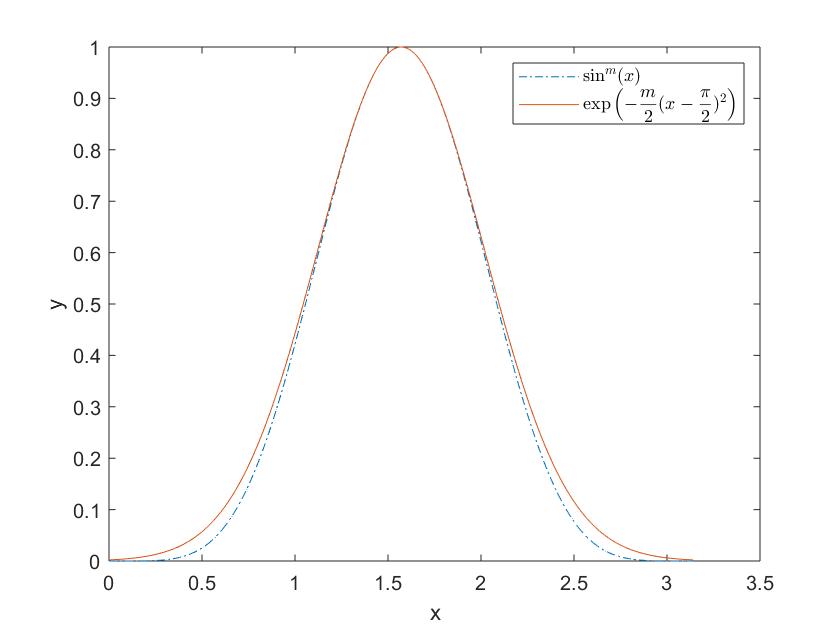}
        \end{minipage}
    \begin{minipage}[t]{0.5\linewidth}
    \centering
    \includegraphics[height=6cm,width=8cm]{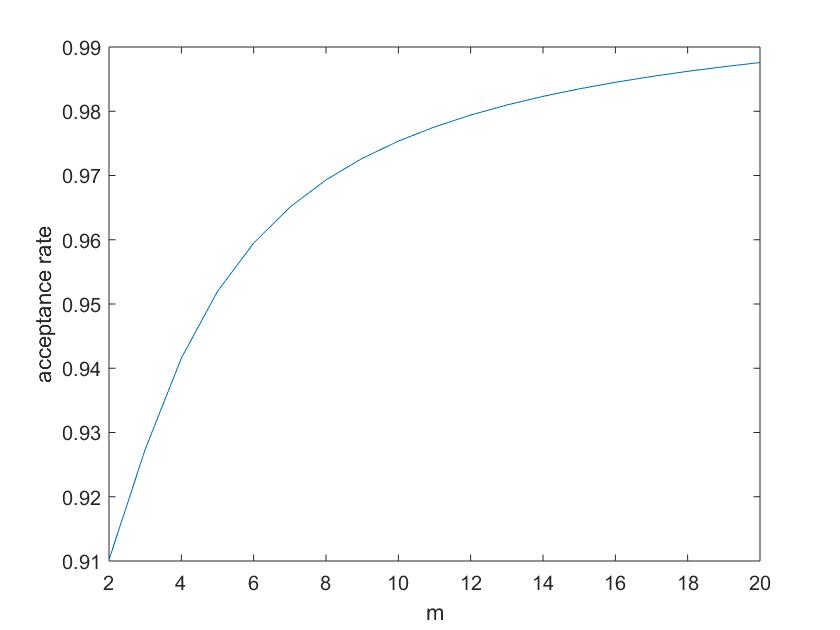}
    \end{minipage}
 \captionsetup{font={footnotesize}}
 \caption{The left figure shows little difference between two functions because of the determination of the parameters.
 The right figure shows the acceptance rate $\eta$ converges to 1.}\label{3fig1}
\end{figure}

Based on the simulation of $\rho,\theta,\varphi_{1},\varphi_{2},\cdots$,$\varphi_{n-2}$, we derive $\texttt{{\rm{x}}}$ so
that the random variable $\texttt{{\rm{X}}}_{k+1}$ can be simulated by $\texttt{{\rm{X}}}_{k+1}=\texttt{{\rm{x}}}$.

  For random variable $\texttt{{\rm{Y}}}_{k+1}$, it is complicated to sample $\texttt{{\rm{Y}}}_{k+1}$ by using density $p_{2,r_{k}}(\texttt{{\rm{y}}},\texttt{{\rm{x}}}_{k})$.
Thus, we rewrite the second part in (\ref{3integralpre}) in the form of
\begin{equation}\label{green_function1}
 \begin{aligned}
  &\int_{\mathbb{B}_{r_{0}}(\texttt{{\rm{x}}}_{0})}f(\texttt{{\rm{y}}})G(0,\texttt{{\rm{y}}}-\texttt{{\rm{x}}}_{0}){\rm d}\texttt{{\rm{y}}}
  \\
 =&\kappa(n,s) \int_{\mathbb{B}_{r_{0}}(\texttt{{\rm{x}}}_{0})}f(\texttt{{\rm{y}}})|\texttt{{\rm{y}}}-\texttt{{\rm{x}}}_{0}|^{2s-n}
 \int_{0}^{\frac{r_{0}^{2}-|\texttt{{\rm{y}}}-\texttt{{\rm{x}}}_{0}|^{2}}{|\texttt{{\rm{y}}}-\texttt{{\rm{x}}}_{0}|^{2}}}
 \frac{t^{s-1}}{(t+1)^{\frac{n}{2}}}{\rm d}t{\rm d}\texttt{{\rm{y}}}.
 \end{aligned}
\end{equation}
Performing substitution $t=\frac{1-t'}{t'}$ yields that
\begin{equation}\label{green_function2}
 \begin{aligned}
   &\int_{\mathbb{B}_{r_{0}}(\texttt{{\rm{x}}}_{0})}f(\texttt{{\rm{y}}})G(0,\texttt{{\rm{y}}}-\texttt{{\rm{x}}}_{0}){\rm d}\texttt{{\rm{y}}}
   \\
  =&\kappa(n,s) \int_{\mathbb{B}_{r_{0}}(\texttt{{\rm{x}}}_{0})}f(\texttt{{\rm{y}}})|\texttt{{\rm{y}}}-\texttt{{\rm{x}}}_{0}|^{2s-n}
  \int_{\frac{|\texttt{{\rm{y}}}-\texttt{{\rm{x}}}_{0}|^{2}}{r_{0}^{2}}}^{1}
  (1-t')^{s-1}(t')^{\frac{n}{2}-s-1}{\rm d}t'{{\rm d}\texttt{{\rm{y}}}}
  \\
  =&b(\texttt{{\rm{x}}}_{0})\mathbb{E}\left\{\left[1-I\left(\frac{|\texttt{{\rm{Y}}}-\texttt{{\rm{x}}}_{0}|^{2}}{r_{0}^{2}};
  \frac{n}{2}-s,s\right)\right]f(\texttt{{\rm{Y}}})\right\},
 \end{aligned}
\end{equation}
where
\begin{equation}
 \begin{aligned}
   b(\texttt{{\rm{x}}}_{0})&=\kappa(n,s)B(\frac{n}{2}-s,s)\int_{B_{r_{0}}(\texttt{{\rm{x}}}_{0})}|\texttt{{\rm{y}}}-
   {\texttt{{\rm{x}}}}_0|^{2s-n}{\rm d}\texttt{{\rm{y}}}
           &=\kappa(n,s)B(\frac{n}{2}-s,s)\frac{r_{0}^{2s}\pi^{n/2}}{s\Gamma(n/2)}
 \end{aligned}
\end{equation}
and the probability density function for \texttt{{\rm{Y}}} is
\begin{equation}{\label{eq:pdf4y}}
p_{2,r_{0}}^{\ast}(\texttt{{\rm{y}}},\texttt{{\rm{x}}}_{0})=\frac{s\Gamma(n/2)}{r_{0}^{2s}\pi^{n/2}}|\texttt{{\rm{y}}}-
\texttt{{\rm{x}}}_{0}|^{2s-n}.
\end{equation}
Thus, the equation (\ref{3intepresolution}) becomes
\begin{equation}\label{simulation}
 u(\texttt{{\rm{x}}}_{0})=\mathbb{E}u(\texttt{{\rm{X}}}_{l})+\sum_{k=1}^{l-1}\mathbb{E}\left\{b(\texttt{{\rm{X}}}_{k})
 \left[1- I\left(\frac{|\texttt{{\rm{Y}}}_{k+1}-\texttt{{\rm{X}}}_{k}|^{2}}{r_{k}^{2}};\frac{n}{2}-s,1-s\right)\right]
 f(\texttt{{\rm{Y}}}_{k+1})
 \right\},
\end{equation}
where the random variable $\texttt{{\rm{Y}}}_{k+1}$ obeys the density $p_{2,r_{k}}^{\ast}(\texttt{{\rm{y}}},\texttt{{\rm{x}}}_{k})$
for the given position $\texttt{{\rm{X}}}_{k}=\texttt{{\rm{x}}}_{k}$.

Using the hyperspherical coordinates (\ref{3hypercoordinate}) for PDF and replacing $\texttt{{\rm{x}}}$ with $\texttt{{\rm{y}}}$
yield
\begin{equation}
 \begin{aligned}
  &\mathbb{P}_{\texttt{{\rm{x}}}_{k}}(\texttt{{\rm{Y}}}_{k+1} \in {\rm d}\texttt{{\rm{y}}})
 =&\frac{2s}{r_{k}^{2s}}\rho^{2s-1}{\rm d}\rho\times \frac{{\rm d}\theta}{2\pi}
 \times \frac{\sin^{n-2}\varphi_{1}}{I_{n-2}}{\rm d}\varphi_{1}\times\cdots
 \times  \frac{\sin\varphi_{n-2}}{I_{1}}{\rm d}\varphi_{n-2},\qquad \rho<r_{k}.
 \end{aligned}
\end{equation}
Notice that $\theta\sim U([0,2\pi])$, $\rho=r_{k}R^{1/2s}$, and $R\sim U([0,1])$ enable us to simulate $\theta$ and $\rho$.
Moreover, we sample $\varphi_{1}$ by inverse sample method and $\varphi_{i},i=2,3,\ldots, n-2$, by rejection sampling method
given before. Then we can simulate $\texttt{{\rm{Y}}}_{k+1}=\texttt{{\rm{y}}}$.

\subsection{One-dimensional fractional Poisson equation}
When $n=1$, the integral   $I\left(\frac{|\texttt{{\rm{Y}}}_{k+1}-\texttt{{\rm{X}}}_{k}|^{2}}{r_{k}^{2}};\frac{1}{2}-s,s\right)$
in (\ref{green_function1}) is infinite if {$\frac{1}{2}< s<1$}. When $0<s<\frac{1}{2}$ and $n=1$,
we can still use the method (\ref{simulation}) with the probability of the moving point's direction $\theta$ is $\mathbb{P}\{\theta=+1\}=\mathbb{P}\{\theta=-1\}=\frac{1}{2}$.
To deal with this troublesome integral, we rewrite (\ref{green_function1}) as follows:
 \begin{equation}\label{simulation_s_1/2_1}
  \begin{aligned}
   &\int_{x_{0}-r_{0}}^{x_{0}+r_{0}}f(y)G(0,y-x_{0}){\rm d}y
   \\
  =&\kappa(1,s)\int_{x_{0}-r_{0}}^{x_{0}+r_{0}}f(y)|y-x_{0}|^{2s-1}\int_{0}^{\frac{r_{0}^{2}-|y-x_{0}|^2}{|y-x_{0}|^2}}
  \frac{t^{s-1}}{(t+1)^{\frac{1}{2}}}{\rm d}t{\rm d}y
  \\
  =&\kappa(1,s)\int_{0}^{r_{0}}[f(x_{0}+\rho)+f(x_{0}-\rho)]{\rho}^{2s-1}\int_{0}^{\frac{r_{0}^{2}-{\rho}^2}{{\rho}^2}}
  \frac{t^{s-1}}{(t+1)^{\frac{1}{2}}}{\rm d}t{\rm d}\rho
  \\
  =&2\kappa(1,s)\int_{0}^{r_{0}}\mathbb{E}_{\theta}[f(x_{0}+\rho\theta)]{\rho}^{2s-1}\int_{0}^{\frac{r_{0}^{2}-{\rho}^2}{{\rho}^2}}
  \frac{t^{s-1}}{(t+1)^{\frac{1}{2}}}{\rm d}t{\rm d}\rho,
  \end{aligned}
 \end{equation}
 where we perform substitution $\rho=|y-x_{0}|$ and $\theta$ {denotes the random variable} with $\mathbb{P}\{\theta=+1\}=\mathbb{P}\{\theta=-1\}=\frac{1}{2}$.
 Via a change of variable $t'={\rho}^2t$ for the inner integral, we obtain
 \begin{equation}
  \begin{aligned}
   \int_{0}^{r_{0}^{2}-{\rho}^2}\left( \frac{t'}{{\rho}^2}\right)^{s-1}\left(\frac{t'}{{\rho}^2}+1\right)^{-\frac{1}{2}}{\rm d}
   \frac{t'}{{\rho}^2}
  ={\rho}^{1-2s}\int_{0}^{r_{0}^{2}-{\rho}^2}t^{s-1}(t+{\rho}^2)^{-\frac{1}{2}}{\rm d}t.
  \end{aligned}
 \end{equation}

It follows from (\ref{simulation_s_1/2_1}) that
 \begin{equation}
  \begin{aligned}
  &2\kappa(1,s)\int_{0}^{r_{0}}\mathbb{E}_{\theta}[f(x_{0}+\rho\theta)]\int_{0}^{r_{0}^{2}-{\rho}^2}t^{s-1}
  (t+{\rho}^2)^{-\frac{1}{2}}{\rm d}t{\rm d}\rho
  \\
  =&2r_{0}\kappa(1,s)\mathbb{E}_{\rho}\left\{\mathbb{E}_{\theta}[f(x_{0}+\rho\theta)]\int_{0}^{r_{0}^{2}-{\rho}^2}t^{s-1}
  (t+{\rho}^2)^{-\frac{1}{2}}{\rm d}t\right\}
  \\
  =&{2r_{0}\kappa(1,s)\mathbb{E}_\texttt{{\rm{Y}}}\left[f(\texttt{{\rm{Y}}})\int_{0}^{r_{0}^{2}
  -{|\texttt{{\rm{Y}}}-\texttt{{\rm{x}}}_{0}|}^2}t^{s-1}(t+{|\texttt{{\rm{Y}}}-\texttt{{\rm{x}}}_{0}|}^2)
  ^{-\frac{1}{2}}{\rm d}t\right]}
  \end{aligned}
 \end{equation}
where $\texttt{{\rm{Y}}}=\texttt{{\rm{x}}}_{0}+\rho\theta$ with $\rho \sim U([0,r_{0}])$ and $\mathbb{P}\{\theta=+1\}=\mathbb{P}\{\theta=-1\}=\frac{1}{2}$. Then
we show that the integral in the expectation can be represented by the hypergeometric function. For simplicity,
let $a=|\texttt{{\rm{Y}}}-\texttt{{\rm{x}}}_{0}|^2$ and $b=r_{0}^{2}-a$. We obtain
\begin{equation}
 \begin{aligned}
&\int_{0}^{b}t^{s-1}(t+a)^{-\frac{1}{2}}{\rm d}t=\int_{0}^{1}(bt)^{s-1}(a+bt)^{-\frac{1}{2}}{\rm d}bt
\\
=&a^{-\frac{1}{2}}b^{s}\int_{0}^{1}t^{s-1}\left(1+\frac{b}{a}t\right)^{-\frac{1}{2}}{\rm d}t
\\
=&a^{-\frac{1}{2}}b^{s}\left[\int_{0}^{1}{t}^{s-1}\left(1+\frac{b}{a}t\right)^{\frac{1}{2}}{\rm d}t-\frac{b}{a}\int_{0}^{1}{t}^{s}\left(1+\frac{b}{a}t\right)^{-\frac{1}{2}}{\rm d}t\right]
\\
=&\frac{a^{-\frac{3}{2}}b^{s}}{s(s+1)}\left[a(s+1)\,{_{2}F_{1}\left(-0.5,s;s+1;-\frac{b}{a}\right)}-bs\,{_{2}F_{1}
\left(0.5,s+1;s+2;-\frac{b}{a}\right)} \right],
 \end{aligned}
\end{equation}
where ${_{2}F_{1}(b,c;d;z)}$ is the hypergeometric function given by the analytic continution
\begin{equation}
_{2}F_{1}(a,b,c;x)=\frac{\Gamma(c)}{\Gamma(b)\Gamma(c-b)}\int_{0}^{1}t^{b-1}(1-t)^{c-b-1}(1-tx)^{-a}{\rm d}t.
\end{equation}
\noindent Finally, equation (\ref{3intepresolution}) can be changed into
\begin{equation}
 \begin{aligned}
 u(\texttt{{\rm{x}}}_{0})=&\mathbb{E}u(\texttt{{\rm{X}}}_{l})+2\kappa(1,s)\sum_{k=1}^{l-1}\mathbb{E}\left\{ r_{k}f(\texttt{{\rm{Y}}}_{k+1})\int_{0}^{r_{k}^{2}-|\texttt{{\rm{Y}}}_{k+1}-\texttt{{\rm{X}}}_{k}|^2}t^{s-1}
 (t+|\texttt{{\rm{Y}}}_{k+1}-\texttt{{\rm{X}}}_{k}|^2)^{-\frac{1}{2}}{\rm d}t\right\}
 \\
         =&\mathbb{E}u(\texttt{{\rm{X}}}_{l})+2\kappa(1,s)\sum_{k=1}^{l-1}\mathbb{E}\Bigg\{ r_{k}f(\texttt{{\rm{Y}}}_{k+1})\frac{1}{|\texttt{{\rm{Y}}}_{k+1}-\texttt{{\rm{X}}}_{k}|^3s(s+1)}
         \\
         &\times \bigg[|\texttt{{\rm{Y}}}_{k+1}-\texttt{{\rm{X}}}_{k}|^2(s+1)\,{_{2}F_{1}\left(-0.5,s;s+1;\frac{r_{k}^{2}
         -|\texttt{{\rm{Y}}}_{k+1}-\texttt{{\rm{X}}}_{k}|^2}{|\texttt{{\rm{Y}}}_{k+1}-\texttt{{\rm{X}}}_{k}|^2}\right)}
         \\
         &-(r_{k}^{2}-|\texttt{{\rm{Y}}}_{k+1}-\texttt{{\rm{X}}}_{k}|^2)s\,{_{2}F_{1}\left(0.5,s+1;s+2;\frac{r_{k}^{2}
         -|\texttt{{\rm{Y}}}_{k+1}-\texttt{{\rm{X}}}_{k}|^2}{|\texttt{{\rm{Y}}}_{k+1}-\texttt{{\rm{X}}}_{k}|^2}\right)}
         \bigg]  \Bigg\}.
 \end{aligned}
\end{equation}

When $n=2s$, i.e., $s=\frac{1}{2}$, we have
\begin{equation}
 \begin{aligned}
u(\texttt{{\rm{x}}}_{0})=\mathbb{E}u(\texttt{{\rm{X}}}_{l})&+2\kappa(1,\frac{1}{2})\sum_{k=1}^{l-1}\mathbb{E}\left\{ r_{k}f(\texttt{{\rm{Y}}}_{k+1})\log\left(\frac{r_{k}+\sqrt{r_{k}^{2}-|\texttt{{\rm{Y}}}_{k+1}-\texttt{{\rm{X}}}_{k}|^2}}
{|\texttt{{\rm{Y}}}_{k+1}-\texttt{{\rm{X}}}_{k}|}\right)\right\}.
\end{aligned}
\end{equation}
Here, $\texttt{{\rm{Y}}}_{k+1}=\texttt{{\rm{x}}}_{k}+\rho\theta$ with $\rho \sim U(0,r_{k})$ and $\mathbb{P}\{\theta=1\}=\mathbb{P}\{\theta=-1\}=\frac{1}{2}$.
{And the random variable $\texttt{{\rm{X}}}_{k}$, $k=1,2,\cdots ,l$ can be derived by the method introduced in high dimension which will undergo a long jump.

\subsection{Summary of the modified walk-one-spheres method}
 We summarize the method in the following algorithm for {equation \eqref{eq:possionproblem}}, {$n\geq2$}:
\begin{algo}\label{algo:mod-walk-on-spheres}
	
Assign fractional order $s$, the domain $\Omega$, the point $\texttt{{\rm{x}}}_{0}\in \mathbb{R}^{n}$, and the number of samples N.
	
Step 1. {Sample $\texttt{{\rm{X}}}_{1}$ and $\texttt{{\rm{Y}}}_{1}$ by probability density functions
$p_{1,r_{0}}(\texttt{{\rm{x}}},\texttt{{\rm{x}}}_{0})$ in (\ref{eq:pdf4x}) and $p_{2,r_{0}}^{\ast}(\texttt{{\rm{y}}},\texttt{{\rm{x}}}_{0})$ in (\ref{eq:pdf4y})
based on $\texttt{{\rm{X}}}_{0}=\texttt{{\rm{x}}}_{0}$, respectively.}

Step 2.   If the latest $\texttt{{\rm{X}}}_{k}$ is out of $\Omega$, go to Step $4$. Otherwise, go to Step $3$.

Step 3. 
{Sample $\texttt{{\rm{X}}}_{k+1}$ and $\texttt{{\rm{Y}}}_{k+1}$by probability density functions
$p_{1,r_{k}}(\texttt{{\rm{x}}},\texttt{{\rm{x}}}_{k})$ in (\ref{eq:pdf4x}) and $p_{2,r_{k}}^{\ast}(\texttt{{\rm{y}}},\texttt{{\rm{x}}}_{k})$ in (\ref{eq:pdf4y})} based on $\texttt{{\rm{X}}}_{k}
=\texttt{{\rm{x}}}_{k}$, respectively and go back to Step $2$.

Step 4. {
Calculate $u(x)=g(\texttt{{\rm{X}}}_{n})+\sum\limits_{k=0}^{l-1}\mathbb{E}\left\{b(\texttt{{\rm{X}}}_{k})
 \left[1- I\left(\frac{|\texttt{{\rm{Y}}}_{k+1}-\texttt{{\rm{X}}}_{k}|^{2}}{r_{k}^{2}};\frac{n}{2}-s,1-s\right)\right]
 f(\texttt{{\rm{Y}}}_{k+1})\right\}$.}

Step 5. Implement {Step 1-4} for $N$ times. Then calculate $u(\texttt{{\rm{x}}})\approx \frac{1}{N}\sum\limits_{i=1}^{N}u^{i}(\texttt{{\rm{x}}})$.
\end{algo}

\section{Bounds on expected steps of walks on spheres}\label{sec:bound-steps-walks}
  In the section, we focus on the problem (\ref{eq:possionball})  when the dimensionality $n\geq2$.
{We give the upper bound on the   number of steps $l$ in expectation as follows}.
{
\begin{theorem}\label{Thm4.1} Consider the   problem  \eqref{eq:possionball} when  $n\geq2$.
 For random walks originating at $\texttt{{\rm{x}}}_{0}\in \mathbb{B}_{r}$, the expectation of the number of steps $l$ in
 Algorithm \ref{algo:mod-walk-on-spheres} is
 \\
 \begin{equation}\label{equ4.1}
 \begin{aligned}
  \mathbb{E}(l)\leq
             &2^{4s+1}{\pi^{\frac{n}{2}}}\frac{\Gamma(s+\frac{n}{2})\Gamma(s+1)}{\Gamma^{2}(\frac{n}{2})}C_{4}(n,s)
             \Bigg\{\frac{2^{s}}{s_{1}}+(r-|\texttt{{\rm{x}}}_{0}|)^s
  \\
       &\times\left[r^{n-{\frac{1+s_{1}+s}{2}}} B\left(1-\frac{1+s_{1}+s}{2},n\right)\right]^{\frac{2(s_{1}+s)}{1+s_{1}+s}}
       \Bigg(\frac{1}{\frac{(1+s_{1}+s)(s_{1}-n)}{1-s_{1}-s}+n}
  \\
       &\times\left[(r+|\texttt{{\rm{x}}}_{0}|)^{\frac{(1+s_{1}+s)(s_{1}-n)}{1-s_{1}-s}+n}
         -\left(\frac{r-|\texttt{{\rm{x}}}_{0}|}{2}\right)^{\frac{(1+s_{1}+s)(s_{1}-n)}{1-s_{1}-s}+n}
         \right]\Bigg)^{\frac{1-s_{1}-s}{1+s_{1}+s}}\Bigg\},
 \end{aligned}
 \end{equation}
where $s_{1}=s$, $s\in\left(0,\frac{1}{3}\right]$ and $s_{1}=\frac{1-s}{2}$, $s\in\left(\frac{1}{3},1\right)$. The right-hand-side function is increasing with respect to $s$ and $|\texttt{{\rm{x}}}_{0}|$, respectively.
\end{theorem}
}
Before the proof of Theorem \ref{Thm4.1}, we need {to introduce} following lemmas.
 Observe that the number  of walks  $l$  depends only on the domain $\Omega$ and is independent of  $f(\texttt{{\rm{x}}})$
 and $g(\texttt{{\rm{x}}})$. Thus we may consider the problem, for fixed $r>0$, given by
\begin{equation}\label{eq:4possionball}
\left\{
\begin{aligned}
 &(-\Delta)^{s}u(\texttt{{\rm{x}}})=f(\texttt{{\rm{x}}}),\, && \texttt{{\rm{x}}}\in \mathbb{B}_{r},
 \\
 &u(\texttt{{\rm{x}}})=0, \, && \texttt{{\rm{x}}}\in \mathbb{R}^{n} \backslash \mathbb{B}_{r},
\end{aligned}
\right.
\end{equation}
where $s\in (0,1)$ and $f(\texttt{{\rm{x}}})=(d(\texttt{{\rm{x}}}))^{-2s}$ and $d(\texttt{{\rm{x}}})=r-|\texttt{{\rm{x}}}|$
denotes the minimum distance from $\texttt{{\rm{x}}}$ to the boundary $\partial \mathbb{B}_{r}$. For any $\texttt{{\rm{x}}}_{k}
\in \mathbb{B}_{r}$, $d(\texttt{{\rm{x}}}_{k})=r_{k}$.

\begin{figure}[!htbp]
  \centering
  \includegraphics[height=5.5cm, width=5.5cm]{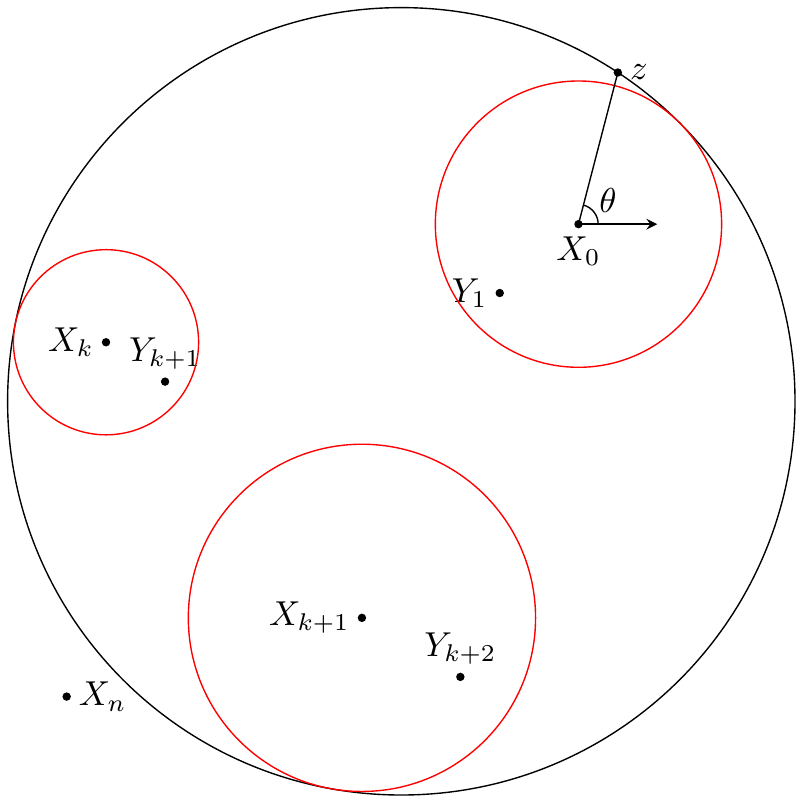}
  \captionsetup{font={footnotesize}}
  \caption{Diagram for expected stopping steps. $\texttt{{\rm{X}}}_{n}$ is the stopping point.}
    \label{4fig1}

\end{figure}

Recall that the Green function $G_{\mathbb{B}_{r}}(\texttt{{\rm{x}}},\texttt{{\rm{y}}})$   in Section 1 is given by
\begin{equation}
 \begin{aligned}
 G_{\mathbb{B}_{r}}(\texttt{{\rm{x}}},\texttt{{\rm{y}}})=2^{-2s}{\pi}^{-n/2}\Gamma(s)^{-2}\Gamma\left(\frac{n}{2}\right)
  |\texttt{{\rm{x}}}-\texttt{{\rm{y}}}|^{2s-n}\int_{0}^{r^{\ast}(\texttt{{\rm{x}}},\texttt{{\rm{y}}})}
  \frac{t^{s-1}}{(t+1)^{\frac{n}{2}}}{\rm d}t.
 \end{aligned}
\end{equation}

\begin{lemma}{\rm({\cite{Chen&Song1998}}})\label{lemma4.1}
 For any ball centred at the origin, $\mathbb{B}_{r}\subset \mathbb{R}^{n}$, we have
 \begin{equation}
  G_{\mathbb{B}_{r}}(\texttt{{\rm{x}}},\texttt{{\rm{y}}})\leq C_{1}{(n,s)}\frac{({d}(\texttt{{\rm{x}}}))^{s}({d}(\texttt{{\rm{y}}}))^s}{|\texttt{{\rm{x}}}-\texttt{{\rm{y}}}|^{n}}, \quad \texttt{{\rm{x}}},\texttt{{\rm{y}}}\in \mathbb{B}_{r},
 \end{equation}
where $C_{1}(n,s)=\pi^{-n/2}\Gamma{(s)}^{-1}\Gamma{(s+1)}^{-1}\Gamma{(\frac{n}{2})}$.
\end{lemma}

\begin{lemma}{\rm({\cite{Chen&Song1998}}})\label{lemma4.2}
 For any ball centred at the origin, $\mathbb{B}_{r}\subset \mathbb{R}^{n}$, we have
 \begin{equation}
  G_{\mathbb{B}_{r}}(\texttt{{\rm{x}}},\texttt{{\rm{y}}})\leq C_{3}{(n,s)}\frac{({d}(\texttt{{\rm{x}}}))^{s}}{{({d}(\texttt{{\rm{y}}}))^s}|\texttt{{\rm{x}}}-\texttt{{\rm{y}}}|^{n-2s}}, \quad \texttt{{\rm{x}}},\texttt{{\rm{y}}}\in \mathbb{B}_{r},
 \end{equation}
where $C_{3}(n,s)=2^{2s}\max\left\{C_{1}(n,s),C_{2}(n,s)\right\}$ and $C_{2}(n,s)=2^{-2s}\pi^{-n/2}\Gamma(\frac{n}{2}-s)\Gamma(s)^{-1}$.
\end{lemma}

Then we have the result in the following.
\begin{lemma}\label{lemma4.3}
For any ball centred at the origin, $\mathbb{B}_{r}\subset \mathbb{R}^{n}$, and $s_{1}\in (0,2s)$,
 \begin{equation}
  G_{\mathbb{B}_{r}}(\texttt{{\rm{x}}},\texttt{{\rm{y}}})\leq {C_{4}{(n,s)}}({d}(\texttt{{\rm{x}}}))^{s}\frac{({d}(\texttt{{\rm{y}}}))^{s-s_{1}}}{|\texttt{{\rm{x}}}
  -\texttt{{\rm{y}}}|^{n-s_{1}}}, \quad \texttt{{\rm{x}}},\texttt{{\rm{y}}}\in \mathbb{B}_{r},
 \end{equation}
 where ${C_{4}{(n,s)}}=\pi^{-n/2}\Gamma(s)^{-1}
 \max\left\{\frac{2^{2s}\Gamma({\frac{n}{2}})}{\Gamma(1+s)},\Gamma(\frac{n}{2}-s)\right\}$.
\end{lemma}
\begin{proof}
  Clearly, Lemma \ref{lemma4.3} holds for $\texttt{{\rm{x}}}=\texttt{{\rm{y}}}$. For $\texttt{{\rm{x}}},\texttt{{\rm{y}}} \in
  \mathbb{B}_{r}$, by Lemmas \ref{lemma4.1} and \ref{lemma4.2},
we obtain
 \begin{equation}
  \begin{aligned}
   G_{\mathbb{B}_{r}}(\texttt{{\rm{x}}},\texttt{{\rm{y}}})& \leq \left[C_{1}{(n,s)}\frac{({d}(\texttt{{\rm{x}}}))^{s}({d}(\texttt{{\rm{y}}}))^s}{|\texttt{{\rm{x}}}-\texttt{{\rm{y}}}|^{n}}\right] \wedge \left[C_{3}{(n,s)}\frac{({d}(\texttt{{\rm{x}}}))^{s}}
   {{({d}(\texttt{{\rm{y}}}))^s}|\texttt{{\rm{x}}}-\texttt{{\rm{y}}}|^{n-2s}}\right]\\
   & \leq \max\{C_{1}(n,s),C_{3}(n,s)\}({d}(\texttt{{\rm{x}}}))^{s}\frac{({d}(\texttt{{\rm{y}}}))^{s-s_{1}}}{|\texttt{{\rm{x}}}-\texttt{{\rm{y}}}|^{n-s_{1}}}
   \left\{   \Big(\frac{{d}(\texttt{{\rm{y}}})}{|\texttt{{\rm{x}}}-\texttt{{\rm{y}}}|}\Big)^{s_{1}} \wedge \Big(\frac{{|\texttt{{\rm{x}}}-\texttt{{\rm{y}}}|}}{{d}(\texttt{{\rm{y}}})}\Big)^{2s-s_{1}}\right\}
   \\
   & \leq {C_{4}(n,s)}({d}(\texttt{{\rm{x}}}))^{s}\frac{({d}(\texttt{{\rm{y}}}))^{s-s_{1}}}
   {|\texttt{{\rm{x}}}-\texttt{{\rm{y}}}|^{n-s_{1}}},
  \end{aligned}
 \end{equation}
 where $a \wedge b=\min\{a,b\}$ and
  \begin{equation}
  \begin{aligned}
  {C_{4}(n,s)}&=\max\big\{C_{1}(n,s),2^{2s}\max\{C_{1}(n,s),C_{2}(n,s)\}\big\}
  =2^{2s}\max\{C_{1}(n,s),C_{2}(n,s)\}
  \\
  &=\pi^{-n/2}\Gamma{(s)}^{-1}\max\left\{\frac{2^{2s}\Gamma({\frac{n}{2}})}{\Gamma(1+s)},
 \Gamma\left(\frac{n}{2}-s\right)\right\}.
  \end{aligned}
  \end{equation}
\end{proof}

\noindent\textit{ Proof of Theorem \ref{Thm4.1}.} \quad From Lemma \ref{lemma4.3} we know that $\int_{\mathbb{B}_{r}}f(\texttt{{\rm{y}}})G(\texttt{{\rm{x}}}_{0},\texttt{{\rm{y}}}){\rm d}\texttt{{\rm{y}}}$ with $\texttt{{\rm{x}}}_{0}\in \mathbb{B}_{r}$
is finite and  is the solution to problem (\ref{eq:4possionball}), i.e.,
\begin{equation}\label{solutionfor4.1}
  u(\texttt{{\rm{x}}}_{0})=\int_{\mathbb{B}_{r}}f(\texttt{{\rm{y}}})G(\texttt{{\rm{x}}}_{0},\texttt{{\rm{y}}}){\rm d}\texttt{{\rm{y}}}, \quad \texttt{{\rm{x}}}_{0}\in \mathbb{B}_{r}.
\end{equation}

\noindent By using equation (\ref{3intepresolution}), we derive
\begin{equation}
 \begin{aligned}
  u(\texttt{{\rm{x}}}_{0})=\int_{\mathbb{B}_{r}}f(\texttt{{\rm{y}}})G(\texttt{{\rm{x}}}_{0},\texttt{{\rm{y}}}){\rm d}\texttt{{\rm{y}}}
   &=\mathbb{E}g(\texttt{{\rm{X}}}_{l})+\mathbb{E}\left[\sum_{k=0}^{l-1}a(\texttt{{\rm{x}}}_{k})f(\texttt{{\rm{Y}}}_{k+1})\right]
   \\
   &=\mathbb{E}\left[\sum_{k=0}^{l-1}a(\texttt{{\rm{X}}}_{k})f(\texttt{{\rm{Y}}}_{k+1})\right],
 \end{aligned}
\end{equation}
where $\mathbb{E}u(\texttt{{\rm{X}}}_{l})=\mathbb{E}g(\texttt{{\rm{X}}}_{l})=0$ have been used since $\texttt{{\rm{X}}}_{l}$
is in the outside of the ball. For $k=0,1,\cdots, l-1$, recalling
\begin{equation}
 a(\texttt{{\rm{X}}}_{k})=\kappa(n,s)B(s,\frac{n}{2})\frac{\omega_{n-1}}{2s}{d}(\texttt{{\rm{X}}}_{k})^{2s},
\end{equation}
the definition of $f(\texttt{{\rm{x}}})=({d}(\texttt{{\rm{x}}}))^{-2s}$ in equation \eqref{eq:4possionball}
and the fact that ${d}(\texttt{{\rm{Y}}}_{k+1})\leq 2{d}(\texttt{{\rm{X}}}_{k})$ (see Figure \ref{4fig1}) lead to the
lower bound
\begin{equation}
 a(\texttt{{\rm{X}}}_{k})f(\texttt{{\rm{Y}}}_{k+1})=\kappa(n,s)B\left(s,\frac{n}{2}\right)
 \frac{\omega_{n-1}}{2s}\left[{\frac{{d}(\texttt{{\rm{X}}}_{k})}{{d}(\texttt{{\rm{Y}}}_{k+1})}}\right]^{2s}
 \geq {\kappa(n,s)B\left(s,\frac{n}{2}\right)\frac{\omega_{n-1}}{2s}\left(\frac{1}{2}\right)^{2s}}.
\end{equation}
Then we have
\begin{equation}\label{4expectedstepast}
 \kappa(n,s)B\left(s,\frac{n}{2}\right)\frac{\omega_{n-1}}{2s}\left(\frac{1}{2}\right)^{2s}\mathbb{E}(l)
 \leq \int_{\mathbb{B}_{r}}f(\texttt{{\rm{y}}})G(\texttt{{\rm{x}}}_{0},\texttt{{\rm{y}}}){\rm d}\texttt{{\rm{y}}}.
\end{equation}
For the right hand side of the inequality, we utilize Lemma \ref{lemma4.3} and partition domain $B_{r}$ into two parts,
{\begin{equation}\label{estimate4green}
 \begin{aligned}
  \int_{\mathbb{B}_{r}}f(\texttt{{\rm{y}}})G(\texttt{{\rm{x}}}_{0},\texttt{{\rm{y}}}){\rm d}\texttt{{\rm{y}}}
  &=\kappa(n,s)\int_{\mathbb{B}_{r}}f(\texttt{{\rm{y}}})|\texttt{{\rm{y}}}-\texttt{{\rm{x}}}_{0}|^{2s-n}\int_{0}^{r^{\ast}
  (\texttt{{\rm{x}}}_{0},\texttt{{\rm{y}}})}\frac{t^{s-1}}{(t+1)^{\frac{n}{2}}}{\rm d}t{\rm d}\texttt{{\rm{y}}}
  \\
  &\leq C_{4}(n,s)({d}(\texttt{{\rm{x}}}_{0}))^{s}\int_{\mathbb{B}_{r}}f(\texttt{{\rm{y}}})
  \frac{({d}(\texttt{{\rm{y}}}))^{s-s_{1}}}{|\texttt{{\rm{x}}}_{0}-\texttt{{\rm{y}}}|^{n-s_{1}}}{\rm d}\texttt{{\rm{y}}}
  \\
  &=C_{4}(n,s)({d}(\texttt{{\rm{x}}}_{0}))^{s}\left[\int\limits_{\mathbb{B}_{h}(\texttt{{\rm{x}}}_{0})}
  +\int\limits_{\mathbb{B}_{r}\backslash {\mathbb{B}_{h}(\texttt{{\rm{x}}}_{0})}}\right]f(\texttt{{\rm{y}}})
  \frac{({d}(\texttt{{\rm{y}}}))^{s-s_{1}}}{  |\texttt{{\rm{x}}}_{0}-\texttt{{\rm{y}}}|^{n-s_{1}}}{\rm d}\texttt{{\rm{y}}}
  \\
  &=C_{4}(n,s)({d}(\texttt{{\rm{x}}}_{0}))^{s}[I_{1}+I_{2}],
 \end{aligned}
\end{equation}
}
where we set {$h=\frac{r-|\texttt{{\rm{x}}}_{0}|}{2}$, $s_{1}=s$, $s\in \left(0,\frac{1}{3}\right]$ and $s_{1}=\frac{1-s}{2}$, $s\in(\frac{1}{3},1)$}. For $I_{1}$, we obtain
{\begin{equation}\label{est4I1}
 \begin{aligned}
  I_{1}&=\int_{\mathbb{B}_{h}(\texttt{{\rm{x}}}_{0})}({d}(\texttt{{\rm{y}}}))^{-s-s_{1}}|\texttt{{\rm{y}}}-
  \texttt{{\rm{x}}}_{0}|^{s_{1}-n}{\rm d}\texttt{{\rm{y}}}
  \\
       &\leq (r-|\texttt{{\rm{x}}}_{0}|-h)^{-s-s_{1}}\int_{\mathbb{B}_{h}(\texttt{{\rm{x}}}_{0})}
       |\texttt{{\rm{y}}}-\texttt{{\rm{x}}}_{0}|^{s_{1}-n}{\rm d}\texttt{{\rm{y}}}
  \\
       &=(r-|\texttt{{\rm{x}}}_{0}|-h)^{-s-s_{1}}\frac{\omega_{n-1}}{s_{1}}h^{s_{1}}
  \\
       &=\frac{2^{s}\omega_{n-1}}{s_{1}}(r-|\texttt{{\rm{x}}}_{0}|)^{-s}.
 \end{aligned}
\end{equation}
}
For $I_{2}$, we have
{
\begin{equation}\label{est4I2}
 \begin{aligned}
  I_{2}&=\int_{\mathbb{B}_{r}\backslash {\mathbb{B}_{h}(\texttt{{\rm{x}}}_{0})}}{d}(\texttt{{\rm{y}}})^{-s-s_{1}}|\texttt{{\rm{x}}}_{0}-
  \texttt{{\rm{y}}}|^{s_{1}-n}{d}\texttt{{\rm{y}}}
  \\
       &\leq \left(\int_{\mathbb{B}_{r}\backslash {\mathbb{B}_{h}(\texttt{{\rm{x}}}_{0})}} [{d}(\texttt{{\rm{y}}})^{-s-s_{1}}]^{\frac{1+s_{1}+s}{2(s+s_{1})}}{\rm d}\texttt{{\rm{y}}}\right)^{\frac{2(s+s_{1})}{1+s_{1}+s}}
       \left(\int_{\mathbb{B}_{r}\backslash{\mathbb{B}_{h}(\texttt{{\rm{x}}}_{0})}}
       (|\texttt{{\rm{x}}}_{0}-\texttt{{\rm{y}}}|^{s_{1}-n})^{\frac{1+s_{1}+s}{1-s_{1}-s}}{\rm d}\texttt{{\rm{y}}}
       \right)^{\frac{1-s_{1}-s}{1+s_{1}+s}}
  \\
       &=(I_{2,1})^{\frac{2(s_{1}+s)}{1+s_{1}+s}}(I_{2,2})^{\frac{1-s_{1}-s}{1+s_{1}+s}}.
 \end{aligned}
\end{equation}
}
For $I_{2,1}$, we use the polar coordinates, hence
{
\begin{equation}
 \begin{aligned}
  I_{2,1} &\leq \omega_{n-1}\int_{0}^{r}(r-\rho)^{-\frac{1+s_{1}+s}{2}}\rho^{n-1} {\rm d}\rho
  \\
          &=\omega_{n-1} r^{n-{\frac{1+s_{1}+s}{2}}}\int_{0}^{1}(1-\rho)^{\frac{1+s_{1}+s}{2}}\rho^{n-1}{\rm d}\rho
          \\
          &=\omega_{n-1} r^{n-{\frac{1+s_{1}+s}{2}}}B\left(1-\frac{1+s_{1}+s}{2},n\right).
 \end{aligned}
\end{equation}
}
For $I_{2,2}$, we use polar coordinates $(\rho,\theta)$ with $\texttt{{\rm{x}}}_{0}$ being treated as the origin. Let us
consider a ray that originates from $\texttt{{\rm{x}}}_{0}$ and has angle $\theta$, which intersect $\partial \mathbb{B}_{r}$
on $\texttt{{\rm{z}}}$ (see Figure \ref{4fig1}). Then we define $r(\theta)=|\texttt{{\rm{z}}}-\texttt{{\rm{x}}}_{0}|$ and
the integral {$I_{2,2}$} can be rewritten as
{
\begin{equation}
 \begin{aligned}
  I_{2,2}&=\int_{\mathbb{B}_{r}\backslash{\mathbb{B}_{h}(\texttt{{\rm{x}}}_{0})}}
  |\texttt{{\rm{x}}}_{0}-\texttt{{\rm{y}}}|^{\frac{(1+s_{1}+s)(s_{1}-n)}{1-s_{1}-s}}
  {\rm d}\texttt{{\rm{y}}}
  \\
         &=\omega_{n-1}\int_{h}^{r(\theta)}\rho^{\frac{(1+s_{1}+s)(s_{1}-n)}{1-s_{1}-s}+n-1}{\rm d}\rho
  \\
         &\leq\frac{\omega_{n-1}}{\frac{(1+s_{1}+s)(s_{1}-n)}{1-s_{1}-s}+n}
         \left[(r+|\texttt{{\rm{x}}}_{0}|)^{\frac{(1+s_{1}+s)(s_{1}-n)}{1-s_{1}-s}+n}
         -\left(\frac{r-|\texttt{{\rm{x}}}_{0}|}{2}\right)^{\frac{(1+s_{1}+s)(s_{1}-n)}{1-s_{1}-s}+n}
         \right]
 \end{aligned}
\end{equation}
}
Bringing (\ref{estimate4green}), (\ref{est4I1}), and (\ref{est4I2}) into inequality (\ref{4expectedstepast})
and noticing $\kappa(n,s)>0$ yield that
{
\begin{equation*}
  \begin{aligned}
        &\kappa(n,s)B(s,\frac{n}{2})\frac{\omega_{n-1}}{2s}\left(\frac{1}{2}\right)^{2s}\mathbb{E}(l)
  \\
    \leq& C_{4}(n,s) (r-|\texttt{{\rm{x}}}_{0}|)^s\Bigg\{\frac{2^{s}\omega_{n-1}}{s_{1}}(r-|\texttt{{\rm{x}}}_{0}|)^{-s}
  +\left[\omega_{n-1}r^{n-{\frac{1+s_{1}+s}{2}}}B\left(1-\frac{1+s_{1}+s}{2},n\right)\right]
  ^{\frac{2(s_{1}+s)}{1+s_{1}+s}}
  \\
        &\times\left(\frac{\omega_{n-1}}{\frac{(1+s_{1}+s)(s_{1}-n)}{1-s_{1}-s}+n}
         \left[(r+|\texttt{{\rm{x}}}_{0}|)^{\frac{(1+s_{1}+s)(s_{1}-n)}{1-s_{1}-s}+n}
         -\left(\frac{r-|\texttt{{\rm{x}}}_{0}|}{2}\right)^{\frac{(1+s_{1}+s)(s_{1}-n)}{1-s_{1}-s}+n}
         \right]\right)^{\frac{1-s_{1}-s}{1+s_{1}+s}}\Bigg\}.
 \end{aligned}
\end{equation*}
}
Thus, we arrive at the following estimate:
{
  \begin{equation}
  \begin{aligned}
\mathbb{E}(l)\leq
             &2^{4s+1}{\pi^{\frac{n}{2}}}\frac{\Gamma(s+\frac{n}{2})\Gamma(s+1)}{\Gamma^{2}(\frac{n}{2})}C_{4}(n,s)
             \Bigg\{\frac{2^{s}}{s_{1}}+(r-|\texttt{{\rm{x}}}_{0}|)^s
  \\
       &\times\left[r^{n-{\frac{1+s_{1}+s}{2}}} B\left(1-\frac{1+s_{1}+s}{2},n\right)\right]^{\frac{2(s_{1}+s)}{1+s_{1}+s}}
       \Bigg(\frac{1}{\frac{(1+s_{1}+s)(s_{1}-n)}{1-s_{1}-s}+n}
  \\
       &\times\left[(r+|\texttt{{\rm{x}}}_{0}|)^{\frac{(1+s_{1}+s)(s_{1}-n)}{1-s_{1}-s}+n}
         -\left(\frac{r-|\texttt{{\rm{x}}}_{0}|}{2}\right)^{\frac{(1+s_{1}+s)(s_{1}-n)}{1-s_{1}-s}+n}
         \right]\Bigg)^{\frac{1-s_{1}-s}{1+s_{1}+s}}\Bigg\}.
 \end{aligned}
\end{equation}
Here $s_{1}=s$, for $s\in\left(0,\frac{1}{3}\right]$ and $s_{1}=\frac{1-s}{2}$ for $s\in\left(\frac{1}{3},1\right)$. So inequality (\ref{equ4.1}) is shown.\\
}
\indent We are now in position to bound the expected number of  steps $l$ before stopping. {Let $a(s)=-\left[\frac{(1+s_{1}+s)(s_{1}-n)}{1-s_{1}-s}+n\right]=\frac{s^2+4ns+2s+4n-3}{2(1-s)}$ and $\rho=\frac{|\texttt{{\rm{x}}}_{0}|}{r}$.} Then, {for $s\in(\frac{1}{3},1)$, we have}
\begin{equation}\label{mon_proof}
 \begin{aligned}
 &\mathbb{E}(l)\leq \max\{C_{5}(n,s),C_{6}(n,s)\}\Bigg\{\frac{2^{s+1}}{1-s}+(1-\rho)^s\left[B\left(
 \frac{1-s}{4},2\right)\right]^{\frac{2+2s}{3+s}}
 \\
  &\qquad \quad\times\left({-\frac{1}{a(s)}
\Big[(1+\rho)^{-a(s)}-2^{a(s)}({1-\rho})^{-a(s)}\Big]}\right)^{\frac{1-s}{3+s}}\Bigg\},
 \end{aligned}
\end{equation}
where
\begin{equation}
 \begin{aligned}
  C_{5}(n,s)&=\frac{2^{6s+1}}{B\left(s,\frac{n}{2}\right)},
  \\
  C_{6}(n,s)&=2^{4s+1}s\frac{\Gamma\left(s+\frac{n}{2}\right)\Gamma\left(\frac{n}{2}-s\right)}
  {\Gamma^{2}\left(\frac{n}{2}\right)}.
 \end{aligned}
\end{equation}
Let $v(s,\rho)$ be the right hand side of equation (\ref{mon_proof}). {Since we have
\begin{equation}
 \begin{aligned}
  &(1-\rho)^s\left({-\frac{1}{a(s)}\Big[(1+\rho)^{-a(s)}-2^{a(s)}({1-\rho})^{-a(s)}\Big]}\right)^{\frac{1-s}{3+s}}
 \\
 =&\left(\frac{1}{a(s)}\right)^{\frac{1-s}{3+s}}\left[2^{a(s)}(1-\rho)^{\frac{s(3+s)}{1-s}-a(s)}
 -(1+\rho)^{-a(s)}(1-\rho)^{\frac{s(3+s)}{1-s}}\right]^{\frac{1-s}{3+s}}
 \end{aligned}
\end{equation}
and $\frac{s(3+s)}{1-s}-a(s)<0$, it can be readily checked  that $v(s,\rho)$ is a
monotonically increasing function with respect to $\rho$ so that we discuss monotonicity for $s$. For $C_{5}(n,s)$,
it is easily obtain $C_{5}(n,s)$ monotonically increases with respect to $s$.}
Observe that  $C_{6}(n,s)$ can be written as follows.
\begin{equation}
 C_{6}(n,s)=2^{4s+1}sB\left(\frac{n}{2}+s,\frac{n}{2}-s\right)\frac{\Gamma(n)}{\Gamma^2\left(\frac{n}{2}\right)}.
\end{equation}
Differentiating Beta function with respect to $s$, we obtain
\begin{equation}
 \begin{aligned}
 &B'\left(\frac{n}{2}+s,\frac{n}{2}-s\right)=\int_{0}^{1}x^{s+\frac{n}{2}-1}(1-x)^{\frac{n}{2}-s-1}\ln\left(\frac{x}{1-x}\right)
 {\rm d}x
 \\
 &=\int_{0}^{\frac{1}{2}}x^{\frac{n}{2}-s-1}(1-x)^{\frac{n}{2}-s-1}\left[x^{2s}\ln\left(\frac{x}{1-x}\right)+(1-x)^{2s}
 \ln\left(\frac{1-x}{x}\right)\right]{\rm d}x.
 \\
 \end{aligned}
\end{equation}
Since the integrand in the square bracket is nonnegative, $B\left(\frac{n}{2}+s,\frac{n}{2}-s\right)$ increases monotonically
such that both $C_{6}(n,s)$ and $\max\{C_{5}(n,s),C_{6}(n,s)\}$ are monotonically increasing coefficients. Then we discuss the
second term in the brace of the equation (\ref{mon_proof}), which is denoted by $v_{1}(s,\rho)$. Taking logarithm of
$v_{1}(s,\rho)$ yields
\begin{equation}
 \begin{aligned}
 \ln(v_{1}(s,\rho))&=s\ln(1-\rho)+\frac{2+2s}{3+s}\ln\left[B\left(\frac{1-s}{4},2\right)\right]
 \\
 &+\frac{1-s}{3+s}\ln\left(\frac{1}{a(s)}\left[2^{a(s)}(1-\rho)^{-a(s)}-(1+\rho)^{-a(s)}\right]\right)
 \\
 &=v_{1,1}(s,\rho)+v_{1,2}(s,\rho),
 \end{aligned}
\end{equation}
where
\begin{equation}
 \begin{aligned}
  v_{1,1}(s,\rho)&=\frac{1-s}{3+s}\ln\left(\frac{1}{a(s)}\left[2^{a(s)}(1+\rho)^{a(s)}-(1-\rho)^{a(s)}\right]\right)
  \\
  v_{1,2}(s,\rho)&=\left(s-a(s)\frac{1-s}{3+s}\right)\ln(1-\rho)+\frac{2+2s}{3+s}\ln\left[B\left(\frac{1-s}{4},2\right)
  \right]-a(s)\frac{1-s}{3+s}\ln(1+\rho).
  \\
 \end{aligned}
\end{equation}
{After careful calculations, we obtain the derivative of  $v_{1,1}(s,\rho)$ with respect to $s$
\begin{equation}
 \begin{aligned}
    ({v_{1,1}(s,\rho)})'_{s}
    =&-\frac{4}{(s+3)^2}\ln\left(\frac{2(1-s)}{s^2+(4n+2)s+4n-3}\right)
  \\
     &-\frac{4}{(s+3)^2}\ln\left[(2+2\rho)^{a(s)}-(1-\rho)^{a(s)}\right]
     +\frac{s^2-2s-8n+1}{(s+3)(s^2+(4n+2)s+4n-3)}
  \\
     &+\frac{1-s}{s+3}a'(s)\frac{(2+2\rho)^{a(s)}\ln(2+2\rho)
     -(1-\rho)^{a(s)}\ln(1-\rho)}{(2+2\rho)^{a(s)}-(1-\rho)^{a(s)}}
  \\
 \geq& \frac{4}{(s+3)^2}\ln\left[\frac{4n-3}{2(1-s)}\right]
     -\frac{4}{(s+3)^2}a(s)\ln(2+2\rho)
  \\
    &+\frac{s^2-2s-8n+1}{(s+3)(s^2+(4n+2)s+4n-3)}
     +\frac{-s^2+2s+8n-1}{2(1-s)(s+3)}\ln(2+2\rho)
  \\
  \geq&\frac{4}{(s+3)^2}\ln\left[\frac{4n-3}{2(1-s)}\right]
      +\left(\frac{1}{2(1-s)}+\frac{4n}{(s+3)^2}\right)\ln(2+2\rho)
  \\
      &+\frac{s^2-2s-8n+1}{(s+3)(s^2+(4n+2)s+4n-3)}
  \\
  \geq& \frac{1}{4}\ln\left(\frac{12n-9}{4}\right)+\left(\frac{1}{2}+\frac{n}{4}\right)\ln2
  -\frac{72}{160}-\frac{26}{160n-\frac{200}{3}}\geq0,
 \end{aligned}
\end{equation}
since $\frac{s^2-2s-8n+1}{(s+3)(s^2+(4n+2)s+4n-3)}$ is an increasing function.
Thus $v_{1,1}(s,\rho)$ is  increasing with respect to $s$.
For $v_{1,2}(s,\rho)$, we have
\begin{equation}
 \begin{aligned}
  v_{1,2}(s,\rho)
  =&\left(s-a(s)\frac{1-s}{3+s}\right)\ln(1-\rho)+\frac{2+2s}{3+s}\left(\ln\Gamma(n)
  +\ln\prod_{k=0}^{n-1}\frac{1}{1-s+4k}\right)
  \\
  &+\left(\frac{2+2s}{3+s}\ln(4^{n})-a(s)\frac{1-s}{s+3}\ln(1+\rho)\right)
  \\
  :=&A_1+A_2+A_3.
 \end{aligned}
\end{equation}
It is easily to obtain $A_1$ and $A_2$ are increasing with respect to $s$. We find
\begin{equation}
 \begin{aligned}
  A_3=\left[\frac{8ns+8n}{s+3}-a(s)\frac{1-s}{s+3}\right]\ln2
    +a(s)\frac{1-s}{s+3}\left[\ln2-\ln(1+\rho)\right]
 \end{aligned}
 \end{equation}
Since $\left[\frac{8ns+8n}{2(s+3)}-a(s)\frac{1-s}{s+3}\right]'=\frac{-s^2-6s-9+8n}{2(s+3)^2}\geq0$
and $\left[a(s)\frac{1-s}{s+3}\right]'=\frac{s^2+6s+8n+9}{2(s+3)^2}>0$, $A_3$ is  increasing with respect to $s$
Combining the monotonicity of $v_1(s,\rho)$ and $\max\{C_{5}(n,s),C_{6}(n,s)\}$ with respect to $s$,
$v(s,\rho)$ will when $s$ grows for $s>\frac{1}{3}$. Similarly, we can derive the same result for $s\leq \frac{1}{3}$. Thus we obtain the desired result.}} \qedsymbol

When $s\rightarrow1$, the upper bounds for the expected stopping steps can not work, since fractional Laplacian
degenerates into the classical Laplace operator so that the L\'{e}vy flight becomes Brownian motion. Though the Brownian motion originated at $\texttt{{\rm{x}}}$ will reach boundary $\mathbb{B}_{r}$ in the probability sense, the expected stopping steps are infinite. We also have the following theorem.
\begin{theorem}
When $s\rightarrow1$, $G(\texttt{{\rm{x}}},\texttt{{\rm{y}}})$ in (\ref{Greenfunction}) is the Green function for the
classical Laplace equation with ball boundary.
\end{theorem}
\begin{proof}
 When $n=2$, we have
 \begin{equation}
 \begin{aligned}
  G(\texttt{{\rm{x}}},\texttt{{\rm{y}}})&=\kappa(2,1)\int_{0}^{r^{*}(\texttt{{\rm{x}}},\texttt{{\rm{y}}})}
  \frac{1}{1+t}{\rm d}t
  =\kappa(2,1)\log\left(\frac{(r^2-|\texttt{{\rm{x}}}|^2)(r^2-|\texttt{{\rm{y}}}|^2)+r^2|\texttt{{\rm{x-y}}}|^2}
  {r^2|\texttt{{\rm{x}}}-\texttt{{\rm{y}}}|^2}\right)
  \\
  &=\frac{1}{4\pi}\log\left(\frac{r^4+|\texttt{{\rm{x}}}|^2|\texttt{{\rm{y}}}|^2-2r^2
  |\texttt{{\rm{x}}}||\texttt{{\rm{y}}}|\cos\langle\texttt{{\rm{x}}},\texttt{{\rm{y}}}\rangle}{r^2|
  \texttt{{\rm{x}}}-\texttt{{\rm{y}}}|^2}\right).
 \end{aligned}
 \end{equation}
   When $n\geq 3$, we have
 \begin{equation}
  \begin{aligned}
  G(\texttt{{\rm{x}}},\texttt{{\rm{y}}})&=\kappa(n,1)|\texttt{{\rm{x}}}-\texttt{{\rm{y}}}|^{2-n}
  \int_{0}^{r^{*}(\texttt{{\rm{x}}},\texttt{{\rm{y}}})}\frac{1}{(1+t)^{\frac{n}{2}}}{\rm d}t
  \\
  &=\frac{2\kappa(n,1)}{n-2}|\texttt{{\rm{x}}}-\texttt{{\rm{y}}}|^{2-n}\left[
  1-\left(\frac{(r^2-|\texttt{{\rm{x}}}|^2)(r^2-|\texttt{{\rm{y}}}|^2)}{r^2|\texttt{{\rm{x}}}-\texttt{{\rm{y}}}|^2}
  +1\right)^{-\frac{n-2}{2}}\right]
  \\
  &=\frac{2\kappa(n,1)}{n-2}\left[\frac{1}{(|\texttt{{\rm{x}}}|^2+|\texttt{{\rm{y}}}|^2-2|\texttt{{\rm{x}}}|
  |\texttt{{\rm{y}}}|\cos{\langle\texttt{{\rm{x}}},\texttt{{\rm{y}}}\rangle})^{\frac{n-2}{2}}}-
  \frac{r^{n-2}}{(r^4+|\texttt{{\rm{x}}}|^2|\texttt{{\rm{y}}}|^2)-2r^2\cos{\langle\texttt{{\rm{x}}},
  \texttt{{\rm{y}}}\rangle}^{\frac{n-2}{2}}}\right]
  \end{aligned}
 \end{equation}
\end{proof}
\begin{remark}
 When $f(\texttt{{\rm{x}}})=(d(\texttt{{\rm{x}}}))^{-2s}$ which doesn't satisfy the condition in Theorem
 \ref{integralrepre}, we know that $u(\texttt{{\rm{x}}})$ in (\ref{solutionfor4.1}) is still the solution
 of the problem (\ref{eq:4possionball}) {in the sense of mild solution}, while $u(\texttt{{\rm{x}}})$ may
 not satisfy the regularity in Theorem \ref{integralrepre}.
\end{remark}

\section{Numerical examples} \label{sec:num-exm}

In this section, numerical examples are carried out by using Scheme I (quadrature method \ref{righthandside1} and
\ref{boundarycondition1}) discussed in Section 2
and Algorithm \ref{algo:mod-walk-on-spheres} introduced in Section \ref{sec:mod-walk-on-spheres} on an i5-8250U CPU.\

 In the experiments, we consider two special cases of
equation (\ref{eq:possionproblem}): the homogeneous equation with inhomogeneous boundary value condition, and nonconstant
source term with homogeneous boundary value condition.

We set step sizes $h=h_{\rho}=h_{\theta}=h_{t}=h_{r},\,r=1,2,\ldots,n-2$. In addition, $E(h)=|u_{2h}-u_{h}|$ denotes
posteriori error estimates in Scheme I and $E$ denotes the absolute error in the modified walk-on-sphere method. Then
the convergence order is given by $rate=\log_{2}\frac{E(2h)}{E(h)}$.

\begin{example}\label{example1}
Let $\Omega$ be a unit ball in $\mathbb{R}^{n}$ centered at the origin
\begin{equation}\label{eq:5possionproblem1}
\left\{
\begin{aligned}
 &(-\triangle)^{s}u(\texttt{{\rm{x}}})=0,\, &&  \texttt{{\rm{x}}}\in\Omega,
 \\
 &u(\texttt{{\rm{x}}})=g(\texttt{{\rm{x}}}),\, && \texttt{{\rm{x}}}\in \mathbb{R}^{n}\backslash \Omega,
\end{aligned}
\right.
\end{equation}
where  $
  g(\texttt{{\rm{x}}})=\exp(-|\texttt{{\rm{x}}}-\texttt{{\rm{x}}}'|^2).$
\end{example}

 In this example, we take $s=0.25$, $0.5$, $0.75$ and set different step sizes
$\frac{1}{32}$, $\frac{1}{64}$, $\frac{1}{128}$, $\frac{1}{256}$, $\frac{1}{512}$ in two spacial dimensions and
$\frac{1}{8}$, $\frac{1}{16}$, $\frac{1}{32}$, $\frac{1}{64}$, $\frac{1}{128}$ in three spacial dimensions for Scheme I.
For the modified walk-on-sphere method, the number of samples are set by $1000$, $10000$, $100000$. We evaluate $u(0.6,0.6)$ with
$\texttt{{\rm{x}}}'=(3,0)$ in two spacial dimensions. The numerical results of Scheme I and modified walk-on-sphere method are
presented in Tables $1$ and $2$, respectively.

\begin{table}[!htbp]
 \centering
 \caption{Numerical results of Example \ref{example1} using Scheme I in 2D.}\label{exmp1tab1}
 \begin{tabular}{cccccc}
   \hline
    $s$      &$\frac{1}{h}$   &approximation     &$E(h)$         &rate   &CPU time(secs.) \\
   \hline
                  &32              &0.0234077           &5.6370E-06     &$-$      &0.0631          \\
    $s=0.25$      &64              &0.0234021           &8.9306E-07     &2.6581   &0.1039           \\
                  &128             &0.0234012           &2.2166E-07     &2.0104   &0.3229           \\
                  &256             &0.0234009           &5.5441E-08     &1.9994   &1.1362          \\
                  &512             &0.0234009           &1.3863E-08     &1.9996   &4.4059           \\
   \hline
                  &32              &0.0187671           &8.2695E-06     &$-$      &0.0595          \\
    $s=0.50$      &64              &0.0187558           &4.6391E-07     &4.1559   &0.1041        \\
                  &128             &0.0187583           &1.1117E-07     &2.0612   &0.3031           \\
                  &256             &0.0187582           &2.8056E-08     &1.9863   &1.0778          \\
                  &512             &0.0187582           &7.0603E-09     &1.9905   &4.1187           \\
   \hline
                  &32              &0.0099238           &1.5558E-05     &$-$      &0.0599          \\
    $s=0.75$      &64              &0.0099082           &3.7635E-07     &5.3694   &0.1407        \\
                  &128             &0.0099079           &8.3237E-08     &2.1768   &0.7417           \\
                  &256             &0.0099078           &2.1912E-08     &1.9255   &1.4891          \\
                  &512             &0.0099077           &5.7075E-09     &1.9407   &5.5057           \\
   \hline
 \end{tabular}
\end{table}
\begin{table}[!htbp]
 \centering
 \caption{Numerical results of Example \ref{example1} using modified walk-on-sphere method in 2D.}\label{exmp1tab2}
 \begin{tabular}{cccccc}
   \hline
    $s$      &samples    &approxiamtion       &average no. steps  &variance      &CPU time(secs.) \\
   \hline
                  &1000       &0.0230409             &1.7470             &9.4958E-03    &0.0078       \\
    $s=0.25$      &10000      &0.0231043             &1.7515             &9.0476E-03    &0.0835       \\
                  &100000     &0.0234345             &1.7543             &8.4807E-03    &0.6804       \\
   \hline
                  &1000       &0.0185969             &2.9460             &1.0672E-02    &0.0171       \\
    $s=0.50$      &10000      &0.0188054             &3.0495             &5.8763E-03    &0.1179       \\
                  &100000     &0.0187276             &3.0142             &5.7382E-03    &1.1908       \\
   \hline
                  &1000       &0.0098083             &6.4760             &3.2685E-03    &0.0295       \\
    $s=0.75$      &10000      &0.0097991             &6.2111             &2.1455E-03    &0.2601       \\
                  &100000     &0.0098974             &6.1990             &1.7594E-03    &2.5103       \\
   \hline
 \end{tabular}
\end{table}
  Table \ref{exmp1tab1} shows the convergent order coincides with the theoretical analysis. It can be seen from Table
  \ref{exmp1tab2} that the simulation results by the Monte Carlo method are close to the approximations in Table
  \ref{exmp1tab1}.

Next, we evaluate $u(0.5,0.5,0.5)$ with $\texttt{{\rm{x}}}'=(3,0,0)$ in three spacial dimensions. Numerical results
are given in Tables \ref{exmp1tab3} and \ref{exmp1tab4}. Table \ref{exmp1tab3} shows that although Scheme I achieves
same convergent order while the CPU time in three spacial dimensions grows a lot. Compared with the Scheme I,
modified walk-on-sphere method presented in Table \ref{exmp1tab4} saves much more time.

\begin{table}[!htbp]
 \centering
 \caption{Numerical results of Example \ref{example1} using Scheme I in 3D.}\label{exmp1tab3}
 \begin{tabular}{cccccc}
   \hline
    $s$      &$\frac{1}{h}$   &approximation     &$E(h)$         &rate   &CPU time(secs.) \\
   \hline
                  &8              &0.0084161           &7.3967E-05     &$-$      &0.0983          \\
    $s=0.25$      &16             &0.0079807           &9.5757E-05     &3.4423   &0.2698          \\
                  &32             &0.0080208           &2.1302E-05     &2.1448   &1.3501          \\
                  &64             &0.0080298           &4.9834E-06     &2.0779   &14.830          \\
                  &128            &0.0080320           &1.2291E-06     &2.0155   &103.99          \\
   \hline
                  &8              &0.0066376           &7.3967E-05     &$-$      &0.0752          \\
    $s=0.50$      &16             &0.0065636           &9.5757E-06     &-0.3725   &0.2861          \\
                  &32             &0.0066594           &2.1377E-05     &2.1684   &1.2785          \\
                  &64             &0.0066807           &4.6231E-06     &2.0957   &9.4483          \\
                  &128            &0.0066856           &1.0635E-06     &2.0195   &78.617          \\
   \hline
                  &8              &0.0033824           &2.7559E-04     &$-$      &0.0792          \\
    $s=0.75$      &16             &0.0036580           &1.5625E-04     &0.8185   &0.2571          \\
                  &32             &0.0038143           &3.4898E-05     &2.1627   &1.3584          \\
                  &64             &0.0038491           &8.0415E-06     &2.1176   &10.167          \\
                  &128            &0.0038572           &1.9762E-06     &2.0247   &85.237          \\
   \hline
 \end{tabular}
\end{table}
\begin{table}[!htbp]
 \centering
 \caption{Numerical results of Example \ref{example1} for modified walk-on-sphere method in 3D.}\label{exmp1tab4}
 \begin{tabular}{cccccc}
   \hline
    $s$      &samples    &approximation       &average no. steps  &variance       &CPU time(secs.) \\
   \hline
                  &1000       &0.0083326            &1.8890             &2.2310E-03    &0.0340       \\
    $s=0.25$      &10000      &0.0080532            &1.9248             &1.9605E-03    &0.1520       \\
                  &100000     &0.0080475            &1.9259             &1.8473E-03    &1.3697       \\
   \hline
                  &1000       &0.0068445            &3.9280             &3.1425E-03    &0.0476       \\
    $s=0.50$      &10000      &0.0067445            &3.8514             &1.3806E-03    &0.3632       \\
                  &100000     &0.0066647            &3.8748             &1.2729E-03    &2.7069       \\
   \hline
                  &1000       &0.0039063            &10.327             &1.8671E-03    &0.0936       \\
    $s=0.75$      &10000      &0.0037649            &10.010             &5.4232E-04    &0.7213       \\
                  &100000     &0.0038088            &10.110             &4.1431E-04    &7.4822       \\
   \hline
 \end{tabular}
\end{table}

\begin{example}\label{example2}
Consider equation (\ref{eq:5possionproblem1}) with $\Omega$ being a unit ball in $\mathbb{R}^{n}$ centered at the origin,
\begin{equation}
 g(\texttt{{\rm{x}}})=\left\{
 \begin{aligned}
 &\frac{1}{\pi}\log{|x-x'|}, \, &&n=1,\\
   &a(n,s)|\texttt{{\rm{x}}}-\texttt{{\rm{x}}}'|^{-n+2s},\, &&n \geq 2,
  \end{aligned}
  \right.
 \end{equation}
and
\begin{equation}
 \begin{aligned}
 &a(n,s)=\frac{\Gamma(\frac{n}{2}-s)}{2^{2s}\pi^{\frac{n}{2}}\Gamma(s)}.
 \end{aligned}
\end{equation}
Here $g(\texttt{{\rm{x}}})$ is just Green's function and $\texttt{{\rm{x}}}'\in \mathbb{R}^{n}\backslash \Omega$,
where $\texttt{{\rm{x}}}'$ can be any points out of the unit ball.
The exact solution to (\ref{example1}) is given below \rm{\cite{Bucur2016}},
\begin{equation}
 u(\texttt{{\rm{x}}})=\left\{
  \begin{aligned}
     &\frac{1}{\pi}\log{|x-x'|} ,\, &&n=1,\\
     &a(n,s)|\texttt{{\rm{x}}}-\texttt{{\rm{x}}}'|^{-n+2s} ,\, &&n \geq 2.
  \end{aligned}
   \right.
\end{equation}
\end{example}
We use the modified walk-on-sphere method to simulate the solution.
The number of samples are set by $1000$, $10000$, and $100000$ for modified walk-on-sphere method. Though $g(\texttt{{\rm{x}}})$
does not satisfy the condition in Theorem \ref{integralrepre}, modified walk-on-sphere method still takes effect since the
representation formula is finite. The value of $u(\frac{1}{2})$ with $x'=2$ in one spacial dimension is showed in
Table \ref{exmp2tab1_1}.
We also evaluate $u(0.6,0.6)$ with $\texttt{{\rm{x}}}'=(\sqrt{2},\sqrt{2})$ in two spacial dimensions.
Table \ref{exmp2tab1} gives the numerical results. The average number of step in Table \ref{exmp2tab1} is
basically the same as that in Table \ref{exmp1tab2}, indicating that the average number of step
isn't related to $g(\texttt{{\rm{x}}})$. When $10^{5}$ samples are used in modified walk-on-sphere method, Figure \ref{5fig1}
shows that the larger the $s$ is, the smaller the errors will be, which is caused by the singularity of $g(x)$. And the average
number of steps will increase when $s$ tends to 1, which explains why the CPU time will become longer when $s$ grows.
\begin{table}[!htbp]
 \centering
 \caption{Numerical results of Example \ref{example2} using modified walk-on-sphere method method in 1D.}\label{exmp2tab1_1}
 \begin{tabular}{cccccc}
   \hline
    $s$      &samples         &$E$       &average no. steps  &variance      &CPU time(secs.) \\
   \hline
                  &1000       &8.8457E-03   &1.2910             &3.1650E-01    &0.0129       \\
    $s=0.25$      &10000      &2.2407E-03   &1.2929             &2.4775E-01    &0.0912       \\
                  &100000     &8.4281E-04   &1.2915             &2.3637E-01    &0.8159       \\
   \hline
                  &1000       &6.9838E-03   &1.5000             &2.8124E-01    &0.0150       \\
    $s=0.50$      &10000      &3.7475E-03   &1.5290             &2.7348E-01    &0.0982       \\
                  &100000     &4.8992E-04   &1.5246             &2.6451E-01    &0.8901       \\
   \hline
                  &1000       &7.8929E-03   &1.7200             &3.7211E-01    &0.0253       \\
    $s=0.75$      &10000      &3.1675E-03   &1.7069             &3.7141E-01    &0.1236       \\
                  &100000     &2.8304E-05   &1.6879             &3.5764E-01    &1.1568       \\
    \hline
 \end{tabular}
\end{table}
\begin{table}[!htbp]
 \centering
 \caption{Numerical results of Example \ref{example2} using modified walk-on-sphere method in 2D.}\label{exmp2tab1}
 \begin{tabular}{cccccc}
   \hline
    $s$      &samples         &$E$       &average no. steps  &variance      &CPU time(secs.) \\
   \hline
                  &1000       &7.7764E-03   &1.7220             &1.5125E-01    &0.0191       \\
    $s=0.25$      &10000      &3.3265E-03   &1.7338             &3.3816E-02    &0.1893       \\
                  &100000     &1.7565E-03   &1.7338             &1.2221E-02    &1.6108       \\
   \hline
                  &1000       &2.3412E-03   &3.0770             &1.4768E-02    &0.0292       \\
    $s=0.50$      &10000      &1.3937E-04   &3.0016             &8.2579E-03    &0.2636       \\
                  &100000     &7.8162E-05   &3.0004             &6.7842E-03    &2.5823       \\
   \hline
                  &1000       &9.5021E-03   &6.5530             &4.0156E-02    &0.0813       \\
    $s=0.75$      &10000      &1.0741E-04   &6.2859             &3.6001E-02    &0.5657       \\
                  &100000     &2.2905E-05   &6.2344             &2.9957E-02    &5.5934       \\
    \hline
 \end{tabular}
\end{table}
\begin{figure}[!htbp]
    \begin{minipage}[t]{0.5\linewidth}
    \centering
    \includegraphics[height=6cm,width=8cm]{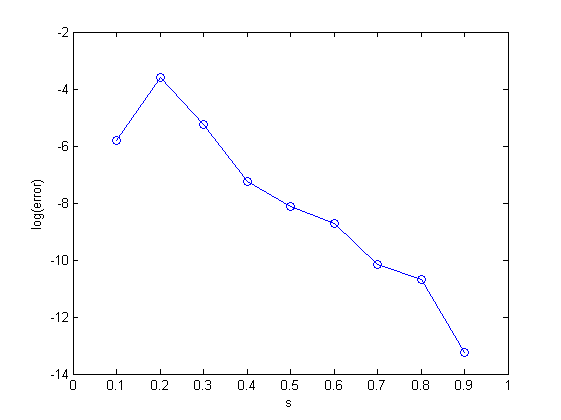}
    \end{minipage}
    \begin{minipage}[t]{0.5\linewidth}
    \centering
    \includegraphics[height=6cm,width=8cm]{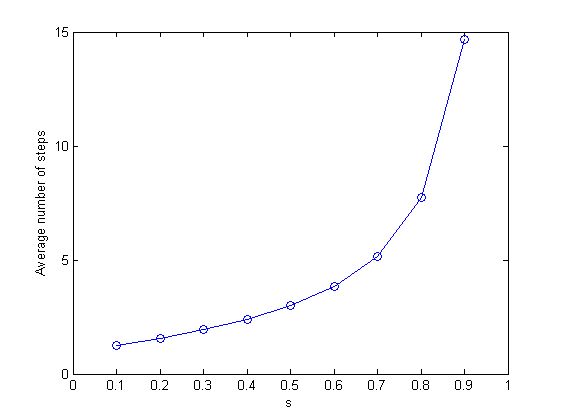}
    \end{minipage}
 \captionsetup{font={footnotesize}}
 \caption{Result for Example \ref{example1} with the  walk-one-spheres  method based on $10^5$ samples in two dimensions.
 When $s$ becomes bigger, the error will be smaller whereas the average number of steps will increase.}\label{5fig1}
\end{figure}

Next, we evaluate $u(0.5,0.5,0.5)$ with $\texttt{{\rm{x}}}'=(\frac{2\sqrt{3}}{3},\frac{2\sqrt{3}}{3},\frac{2\sqrt{3}}{3})$
in three spacial dimensions. The numerical results are given in Table \ref{exmp2tab2}. The consuming time does not grow
too much as the dimension increases. Figure \ref{5fig2} also indicates the relation between the average number of step
and index $s$ remains, which coincides with theoretical analysis. 

\begin{table}[!htbp]
 \centering
 \caption{Numerical results of Example \ref{example2} for modified walk-on-sphere method in 3D.}\label{exmp2tab2}
 \begin{tabular}{cccccc}
   \hline
    $s$      &samples         &$E$       &average no. steps  &variance      &CPU time(secs.) \\
   \hline
                  &1000       &5.0191E-03   &1.8500             &1.7617E-04    &0.0213       \\
    $s=0.25$      &10000      &9.5331E-04   &1.9544             &1.2973E-04    &0.1645       \\
                  &100000     &3.9997E-04   &1.9544             &3.5799E-05    &1.5001       \\
   \hline
                  &1000       &2.1632E-03   &4.0730             &8.7991E-05    &0.0318       \\
    $s=0.50$      &10000      &6.4897E-04   &3.8937             &5.2222E-05    &0.2899       \\
                  &100000     &2.4870E-05   &3.8930             &2.2081E-05    &2.7808       \\
   \hline
                  &1000       &3.1106E-03   &9.8970             &3.3897E-05    &0.1346       \\
    $s=0.75$      &10000      &5.4258E-05   &10.181             &2.6497E-05    &0.8003       \\
                  &100000     &3.6118E-05   &10.200             &2.0900E-05    &7.4063       \\
   \hline
 \end{tabular}
\end{table}

\begin{figure}[!htbp]   
    \begin{minipage}[t]{0.5\linewidth}
    \centering
    \includegraphics[height=6cm,width=8cm]{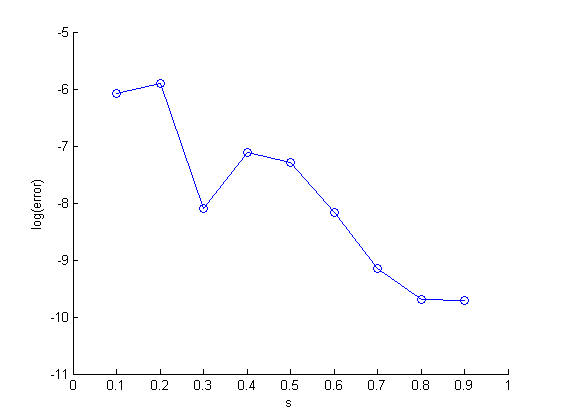}
    \end{minipage}
    \begin{minipage}[t]{0.5\linewidth}
    \centering
    \includegraphics[height=6cm,width=8cm]{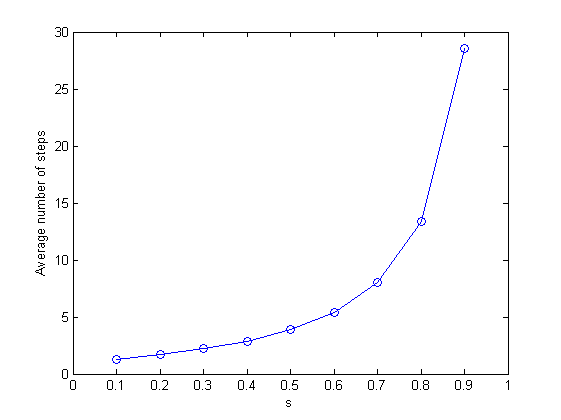}
    \end{minipage}
 \captionsetup{font={footnotesize}}
 \caption{Result for Example \ref{example1} with the  walk-one-spheres  method based on $10^5$ samples in three dimensions.
The error and average number of steps keep the same tendency as those in 2D when $s$ tends to $1$ from $0$.
}\label{5fig2}
\end{figure}
For higher dimensional cases, we first evaluate $u(\texttt{{\rm{x}}})$ with $\texttt{{\rm{x}}}=\frac{1}{4}*\textbf{ones}(4)$ and $\texttt{{\rm{x}}}'=\textbf{ones}(4)$ in four spacial dimensions and $u(\texttt{{\rm{x}}})$ with $\texttt{{\rm{x}}}=
\frac{1}{5}*\textbf{ones}(5)$ and $\texttt{{\rm{x}}}'=\frac{2\sqrt{5}}{5}*\textbf{ones}(5)$ in five spacial dimensions,
where $\textbf{ones}(n)$ is an $n$-dimensional vector with all ones. The number of samples are set by $10^5$. Since when
$s\in(0,\frac{1}{2})$, $g(x)$ has singularity, we mainly give the numerical results for $s\in(\frac{1}{2},1)$ in Table {\ref{exmp2tab3}}.

\begin{table}[!htbp]
 \centering
 \caption{Numerical results of Example \ref{example2} for modified walk-on-sphere method in 4D and 5D by $10^5$ samples.}\label{exmp2tab3}
 \begin{tabular}{cccccc}
   \hline
    $s$            &$n$      &$E$       &average no. steps  &variance      &CPU time(secs.) \\
   \hline
    $s=0.25$      &$n=4$      &1.1203E-02   &1.5387             &2.1507E-01    &5.8991       \\
    $s=0.50$      &$n=4$      &2.0822E-03   &3.3463             &8.8047E-02    &12.082       \\
    $s=0.60$      &$n=4$      &1.4179E-03   &5.0610             &4.8773E-02    &19.741       \\
    $s=0.70$      &$n=4$      &6.4099E-04   &8.2784             &1.9234E-02    &32.903       \\
    $s=0.80$      &$n=4$      &3.4514E-04   &15.663             &8.9541E-03    &60.322       \\
    $s=0.90$      &$n=4$      &2.8155E-04   &38.629             &2.9072E-03    &208.21       \\
   \hline
    $s=0.25$      &$n=5$      &1.7871E-03   &1.5178             &4.6222E-02    &8.7034       \\
    $s=0.50$      &$n=5$      &1.6622E-03   &3.4818             &2.8613E-02    &19.241       \\
    $s=0.60$      &$n=5$      &8.2384E-04   &5.4826             &1.6533E-02    &30.330       \\
    $s=0.70$      &$n=5$      &6.3443E-04   &9.4176             &8.4255E-03    &52.100       \\
    $s=0.80$      &$n=5$      &3.4322E-04   &18.598             &2.2152E-03    &100.84       \\
    $s=0.90$      &$n=5$      &2.9972E-04   &48.296             &1.0710E-03    &271.01       \\
   \hline
 \end{tabular}
\end{table}
We then evaluate $u(\texttt{{\rm{x}}})$ with $\texttt{{\rm{x}}}=\frac{1}{10}*\textbf{ones}(10)$ and $\texttt{{\rm{x}}}'=
\frac{\sqrt{10}}{5}*\textbf{ones}(10)$ in ten spacial dimensions. The number of samples are set by $10^5$ and the numerical
results is given in Table \ref{exmp2tab4}. Compared with the computational time in lower dimensions, the time in ten dimensions
only increases in multiple, which shows the efficiency of the algorithm.
\begin{table}[!htbp]
 \centering
 \caption{Numerical results of Example \ref{example2} for modified walk-on-sphere method in 10D by $10^5$ samples.}\label{exmp2tab4}
 \begin{tabular}{cccccc}
   \hline
    $s$            &$n$      &$E$       &average no. steps  &variance      &CPU time(secs.) \\
   \hline
    $s=0.25$      &$n=10$      &1.1617E-03   &1.4501             &2.4291E-01    &23.991       \\
    $s=0.50$      &$n=10$      &6.8564E-04   &3.6944             &7.5413E-02    &52.489       \\
    $s=0.60$      &$n=10$      &5.3812E-04   &6.4110             &5.2138E-04    &94.023       \\
    $s=0.70$      &$n=10$      &2.9153E-04   &12.266             &2.4223E-04    &178.82       \\
    $s=0.80$      &$n=10$      &2.4108E-04   &27.366             &6.1517E-05    &395.73       \\
    $s=0.90$      &$n=10$      &1.2341E-04   &80.842             &5.4785E-05    &1168.0       \\
   \hline
 \end{tabular}
\end{table}

\begin{example}\label{example3}
Consider the following fractional Poisson equation with vanishing Dirichlet boundary condition
 \begin{equation}\label{eq:5possionproblem2}
  \left\{
  \begin{aligned}
    &(-\Delta)^{s}u(\texttt{{\rm{x}}})=f(\texttt{{\rm{x}}}) ,\,&&  \texttt{{\rm{x}}}\in\Omega,
    \\
  &u(\texttt{{\rm{x}}})=0 ,\, &&  \texttt{{\rm{x}}}\in \mathbb{R}^{n}\backslash \Omega,
  \end{aligned}
  \right.
 \end{equation}
 where $\Omega$ is a unit ball in $\mathbb{R}^{n}$ and
 \begin{equation}
  f(\texttt{{\rm{x}}})=\left\{
  \begin{aligned}
  &\Gamma\left(\frac{s}{2}+2\right)x ,\, &&n=1,\\
  &2^{2s}\Gamma(2+s)\Gamma \left(\frac{n}{2}+s\right){\Gamma\left(\frac{n}{2}\right)}^{-1}
\left(1-\left(1+\frac{2s}{n}\right)|\texttt{{\rm{x}}}|^2\right), &&n \geq 2.
  \end{aligned}
  \right.
 \end{equation}
 The exact solution to (\ref{eq:5possionproblem2}) is given in \rm\cite{Dyda2012}
 \begin{equation}
  u(\texttt{{\rm{x}}})=\left\{
   \begin{aligned}
     &x(1-x^2)^s ,\, &&n=1,\\
     &(1-|\texttt{{\rm{x}}}|^2)^{1+s},\, &&n \geq 2.
   \end{aligned}
   \right.
 \end{equation}
\end{example}

We use the modified walk-on-sphere method to simulate the solution and the number of samples are set by $1000$, $10000$,
and $100000$ for modified walk-on-sphere method. We evaluated $u(\frac{1}{2})$ in one spacial dimension, which is showed
in Table \ref{exmp3tab1_1}. Since we need to approximate the integral or the hypergeometric function when $\frac{1}{2}<s<1$,
the computational time is a bit longer. {$u(0.6,0.6)$ is also evaluated in two spacial dimensions by both Scheme I
(\ref{boundarycondition1}) and the modified walk-on-sphere method. The numerical results are presented
in Tables \ref{exmp3tab1} and \ref{exmp3tab12}. It is obvious that Scheme I has bigger errors and costs more computational
time.} Comparing Table \ref{exmp1tab2} and Table \ref{exmp2tab1}, it is obvious that the average number of step is independent
of $f(\texttt{{\rm{x}}})$ and $g(\texttt{{\rm{x}}})$. Figure \ref{5fig3} shows that there is no obvious trend in absolute
error when $s$ changes. As expected, we again observe that when $\texttt{{\rm{x}}}$ is closed to the origin, the average
number of steps will become smaller. In particular, when $\texttt{{\rm{x}}}$ is at the origin, the number of steps is one.

\begin{table}[!htbp]
 \centering
 \caption{Numerical results of Example \ref{example3} using modified walk-on-sphere method in 1D.}\label{exmp3tab1_1}
 \begin{tabular}{cccccc}
   \hline
    $s$      &samples    &$E$       &average no. steps  &variance      &CPU time(secs.) \\
   \hline
                  &1000       &9.9759E-03   &1.2540             &1.1209E-01    &0.0159       \\
    $s=0.25$      &10000      &8.7500E-04   &1.2952             &1.1191E-01    &0.1131       \\
                  &100000     &4.6390E-04   &1.2879             &1.0927E-01    &0.8675       \\
   \hline
                  &1000       &9.4393E-03   &1.5410             &1.7148E-01    &0.0178       \\
    $s=0.50$      &10000      &2.9952E-03   &1.5316             &1.6840E-01    &0.1139       \\
                  &100000     &2.0324E-04   &1.5281             &1.6630E-01    &0.9490       \\
   \hline
                  &1000       &3.5671E-03   &1.7000             &8.6924E-01    &0.8696       \\
    $s=0.75$      &10000      &1.6581E-03   &1.6945             &2.3160E-01    &6.3171       \\
                  &100000     &4.2174E-04   &1.6838             &2.7423E-01    &45.534       \\
   \hline
 \end{tabular}
\end{table}

\begin{table}[!htbp]
 \centering
 \caption{Numerical results of Example \ref{example3} for modified walk-on-sphere method in 2D.}\label{exmp3tab1}
 \begin{tabular}{cccccc}
   \hline
    $s$      &samples         &$E$       &average no. steps  &variance      &CPU time(secs.) \\
   \hline
                  &1000       &6.9086E-03   &1.7620             &1.1539E-01    &0.0216       \\
    $s=0.25$      &10000      &3.1608E-03   &1.7716             &1.0973E-01    &0.1662       \\
                  &100000     &1.3496E-04   &1.7606             &8.7965E-02    &1.5673       \\
   \hline
                  &1000       &6.4914E-03   &2.8210             &1.9263E-01    &0.0332       \\
    $s=0.50$      &10000      &2.6209E-03   &2.9709             &1.7505E-01    &0.2636       \\
                  &100000     &1.3063E-04   &2.9997             &1.7088E-01    &2.8465       \\
   \hline
                  &1000       &9.5021E-03   &5.9050             &2.2955E-01    &0.0538       \\
    $s=0.75$      &10000      &1.0741E-04   &6.1569             &2.1898E-01    &0.5109       \\
                  &100000     &2.2905E-05   &6.1818             &1.8771E-01    &5.6996       \\
   \hline
 \end{tabular}
\end{table}

\begin{table}[!htbp]
 \centering
 \caption{{Numerical results of Example \ref{example1} using Scheme I (\ref{boundarycondition1}) in 2D.}}\label{exmp3tab12}
 \begin{tabular}{cccccc}
   \hline
    $s$      &$\frac{1}{h}$        &$E(h)$         &rate   &CPU time(secs.) \\
   \hline
                  &32              &3.4047E-02     &$-$      &0.9227          \\
    $s=0.25$      &64              &2.5291E-02     &0.5159   &10.797          \\
                  &128             &1.9142E-02     &0.5099   &53.474          \\
                  &256             &1.4927E-02     &0.5462   &246.99          \\
                  &512             &1.2011E-02     &0.5317   &1922.1          \\
   \hline
                  &32              &8.6860E-03     &$-$      &0.6604          \\
    $s=0.50$      &64              &4.7663E-03     &0.4627   &4.2421          \\
                  &128             &2.7036E-03     &0.9262   &36.764          \\
                  &256             &1.6377E-03     &0.9524   &335.90          \\
                  &512             &1.0602E-03     &0.8392   &1712.6          \\
   \hline
                  &32              &1.1589E-02     &$-$      &0.7423          \\
    $s=0.75$      &64              &4.6710E-03     &0.9681   &5.3990          \\
                  &128             &1.8928E-03     &1.3157   &39.301          \\
                  &256             &8.2243E-04     &1.3769   &441.43          \\
                  &512             &3.9633E-04     &1.3286   &2169.6         \\
   \hline
 \end{tabular}
\end{table}
\begin{figure}[!htbp]   
    \begin{minipage}[t]{0.5\linewidth}
    \centering
    \includegraphics[height=6cm,width=8cm]{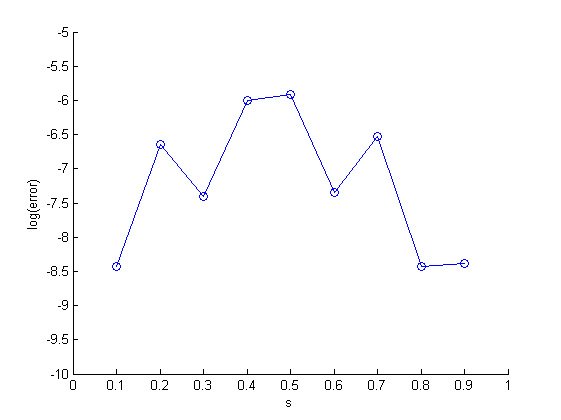}
    \end{minipage}
    \begin{minipage}[t]{0.5\linewidth}
    \centering
    \includegraphics[height=6cm,width=8cm]{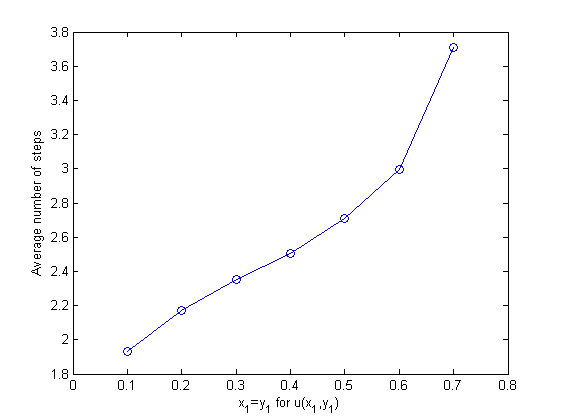}
    \end{minipage}
 \captionsetup{font={footnotesize}}
 \caption{Result for Example \ref{example2} with the  walk-one-sphere  method based on $10^5$ samples
 in two dimensions. From the left to right, we see there is no obvious relation between errors and $s$, meanwhile,
  when $\texttt{{\rm{x}}}$ is near the origin, the average number of steps will be small.
}\label{5fig3}
\end{figure}

Next, we evaluate $u(0.5,0.5,0.5)$ in three spacial dimensions. Table \ref{exmp3tab2} shows that the CPU time of
modified walk-on-sphere method does not increase too much in three dimensions compared with the time in two dimensions.
Figure \ref{5fig4} gives the same result derived in Figure \ref{5fig3}. When $\texttt{{\rm{x}}}$ approaches the
origin, the average number of steps will be small, which coincides with the theoretical analysis.

\begin{table}[!htbp]
 \centering
 \caption{Numerical results of Example \ref{example3} for modified walk-on-sphere method in 3D.}\label{exmp3tab2}
 \begin{tabular}{cccccc}
   \hline
    $s$      &samples         &$E$       &average no. steps  &variance      &CPU time(secs.) \\
   \hline
                  &1000       &4.2582E-03   &1.8890             &7.4618E-02    &0.0327       \\
    $s=0.25$      &10000      &8.6731E-04   &1.9480             &7.2944E-02    &0.1464       \\
                  &100000     &1.5456E-04   &1.9233             &7.2383E-02    &1.4511       \\
   \hline
                  &1000       &7.2182E-03   &3.9970             &1.1093E-01    &0.0348       \\
    $s=0.50$      &10000      &1.7504E-04   &3.8960             &1.0408E-01    &0.3011       \\
                  &100000     &1.3108E-04   &3.9187             &1.0043E-01    &2.9787       \\
   \hline
                  &1000       &9.3257E-04   &10.357             &1.2448E-01    &0.0810       \\
    $s=0.75$      &10000      &5.2534E-04   &10.067             &1.1665E-01    &0.7787       \\
                  &100000     &2.5671E-04   &10.132             &1.1586E-01    &7.7002       \\
   \hline
 \end{tabular}
\end{table}

\begin{figure}[!htbp]
    \begin{minipage}[t]{0.5\linewidth}
    \centering
    \includegraphics[height=6cm,width=8cm]{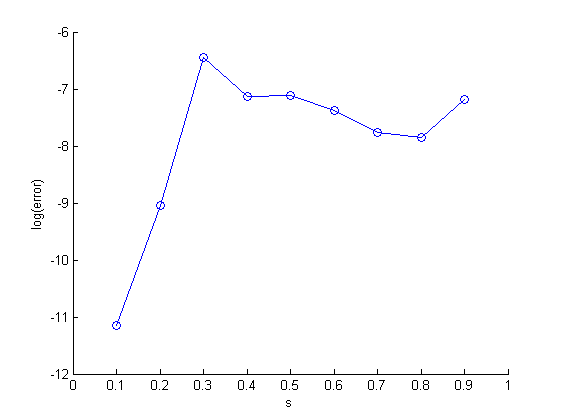}
    \end{minipage}
    \begin{minipage}[t]{0.5\linewidth}
    \centering
    \includegraphics[height=6cm,width=8cm]{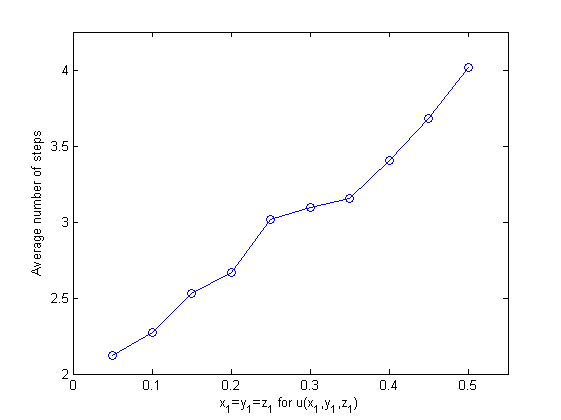}
    \end{minipage}
 \captionsetup{font={footnotesize}}
 \caption{Result for Example \ref{example2} with the  walk-one-sphere  method based on $10^5$ samples
 in three dimensions. From the left to right, we derive the same results as those in two dimensions.
}\label{5fig4}
\end{figure}
For higher dimensional cases, we still first evaluate $u(\texttt{{\rm{x}}})$ with $\texttt{{\rm{x}}}=\frac{1}{4}*\textbf{ones}(4)$ in four spacial dimensions and $u(\texttt{{\rm{x}}})$ with $\texttt{{\rm{x}}}=\frac{1}{5}*\textbf{ones}(5)$ in five spacial dimensions. The number of samples are set by $10^5$ and the numerical results is given in Table \ref{exmp3tab3}. It is noticed that the average number of steps will increase when the fractional order $s$ increases.
\begin{table}[!htbp]
 \centering
 \caption{Numerical results of Example \ref{example3} for modified walk-on-sphere method in 4D and 5D by $10^5$ samples.}\label{exmp3tab3}
 \begin{tabular}{cccccc}
   \hline
    $s$            &$n$        &$E$       &average no. steps  &variance      &CPU time(secs.) \\
   \hline
    $s=0.20$      &$n=4$      &1.0203E-03   &1.3759             &7.0657E-02    &10.931       \\
    $s=0.40$      &$n=4$      &3.8298E-04   &2.3480             &1.3780E-01    &18.279       \\
    $s=0.60$      &$n=4$      &5.9519E-04   &5.0809             &1.9744E-01    &40.125       \\
    $s=0.80$      &$n=4$      &6.2239E-04   &15.702             &2.5873E-01    &121.50       \\
   \hline
    $s=0.20$      &$n=5$      &9.3904E-04   &1.3557             &5.8018E-02    &16.349       \\
    $s=0.40$      &$n=5$      &5.6123E-04   &2.3758             &1.1718E-01    &28.644       \\
    $s=0.60$      &$n=5$      &9.7228E-04   &5.4498             &1.6952E-01    &65.567       \\
    $s=0.80$      &$n=5$      &2.5286E-03   &18.684             &2.2059E-01    &221.95       \\
   \hline
 \end{tabular}
\end{table}

We then evaluate $u(\texttt{{\rm{x}}})$ with $\texttt{{\rm{x}}}=\frac{1}{10}*\textbf{ones}(10)$ in ten spacial dimensions. The number of samples are set by $10^5$. Unlike the homogeneous equation in Examples \ref{example1} and \ref{example2}, we need to sample $\texttt{{\rm{Y}}}$ in every step so
that it will cost more computational time. However, based on the numerical results given in Table \ref{exmp3tab4} the algorithm is still fast and efficient.
\begin{table}[!htbp]
 \centering
 \caption{Numerical results of Example \ref{example3} for modified walk-on-sphere method in 10D by $10^5$ samples.}\label{exmp3tab4}
 \begin{tabular}{cccccc}
   \hline
    $s$            &$n$        &$E$       &average no. steps  &variance      &CPU time(secs.) \\
   \hline
    $s=0.10$      &$n=10$      &2.1003E-04   &1.0953             &1.9213E-02    &34.905       \\
    $s=0.30$      &$n=10$      &1.6807E-03   &1.6611             &5.9872E-02    &49.659       \\
    $s=0.50$      &$n=10$      &6.3033E-03   &3.6920             &9.9123E-02    &110.85       \\
    $s=0.70$      &$n=10$      &2.8531E-03   &12.230             &1.3642E-01    &344.63       \\
    $s=0.90$      &$n=10$      &1.8440E-03   &80.954             &1.6415E-01    &2955.7       \\
   \hline
 \end{tabular}
\end{table}

{Next, we gives the experiment of fractional Poisson equation with square boundary.
\begin{example}\label{example4}
Consider the following fractional Poisson equation with vanishing Dirichlet boundary condition
 \begin{equation}\label{eq:5possionproblem2}
  \left\{
  \begin{aligned}
    &(-\Delta)^{s}u(\texttt{{\rm{x}}})=f(\texttt{{\rm{x}}}) ,\,&&  \texttt{{\rm{x}}}\in[0,1]^{n},
    \\
  &u(\texttt{{\rm{x}}})=0 ,\, &&  \texttt{{\rm{x}}}\in \mathbb{R}^{n}\backslash [0,1]^{n},
  \end{aligned}
  \right.
 \end{equation}
 where $f(x)=1$.
\end{example}
We evaluate $u(\texttt{{\rm{x}}})$ at points $\texttt{{\rm{x}}}_{1}=\frac{1}{1000}*\textbf{ones}(10)$ and $\texttt{{\rm{x}}}_{2}=\frac{1}{10}*\textbf{ones}(10)$ in ten spacial dimensions, respectively.
The number of samples are set by $10^5$. It is reasonable that numerical results are closed to the exact
solution, since when $\texttt{{\rm{x}}}$ is near boundary, numerical solutions are almost equal to 0.
Average number of steps will increase, when $s$ grows, while it seems no relations between average
number of steps and location $\texttt{{\rm{x}}}$.}
\begin{table}[!htbp]
 \centering
 {
 \caption{Numerical results of Example \ref{example4} for modified walk-on-sphere method in 10D.}\label{exmp4tab1}
 \begin{tabular}{cccccc}
   \hline
        $location$             &$s$      &approximation  &average no. steps  &variance      &CPU time(secs.) \\
   \hline
                               &0.25     &7.711E-03      &1.9009             &1.9791E-05    &61.980\\
  $\texttt{{\rm{x}}}_{1}$      &0.50     &5.244E-05      &5.7405             &1.3905E-09    &125.64\\
                               &0.75     &3.227E-07      &26.661             &6.0199E-14    &409.68 \\
   \hline
                               &0.25     &2.430E-01      &1.8899             &1.9093E-02    &62.886\\
  $\texttt{{\rm{x}}}_{2}$      &0.50     &5.250E-02      &5.7755             &1.3489E-03    &120.84\\
                               &0.75     &1.023E-02      &26.781             &6.0105E-05    &409.18\\
   \hline
 \end{tabular}
 }
\end{table}
\begin{remark}
  Since the numerical experiments for modified walk-on-sphere method contain randomness, {the variance sometimes does not converge (e.g. Table \ref{exmp3tab4}).}

As discussed in Section \ref{sec:bound-steps-walks}, the average number of steps depends only on the domain $\Omega$, the
point $\texttt{{\rm{x}}}$ at which the solution we want to arrive, and $s$. Combining Figures \ref{5fig1}-\ref{5fig4},
we conclude that when $\Omega$ is a ball, the steps will increase if $\texttt{{\rm{x}}}$ is far from
the center of the sphere or $s$ becomes larger.
\\
\end{remark}

\section{Conclusion}

We propose  a modified
walk-on-sphere method for the fractional Laplacian problem on general domains in  high dimensions. Based on the
probabilistic representation of the problem, we carefully compute the probabilities of the random walks,
using proper quadrature rules and the modified walk-on-sphere method to sample from the probabilities. We show that
the quadrature rules are of second-order convergent   when the boundary data $g(\texttt{{\rm{x}}})\in  C_b^2
(\mathbb{R}^{n}\backslash \mathbb{B}_{r})$ and the forcing $f=0$. {When $f(\texttt{{\rm{x}}})\in  C_b^2(\mathbb{B}_{r}\backslash S_{h})$
and $g(\texttt{{\rm{x}}})=0$, we derive the numerical method in two dimensions, while the convergent order is only
$\mathcal{O}(h^{2s\wedge1})$ because of the poor property of Green function and it will cost more computational time.
So, it is necessary to propose much more efficient method for the problem.} Thus, for problems in higher dimensions, we
apply  an efficient rejection sampling method based on truncated Gaussian distribution. Also,  we estimate the mean of the
number of walks for the problem in  a ball in $n\,(n\geq2)$ dimensions and $s\in(0,1)$ and  show that the mean of the number
of walks is increasing in $s$ and the distance of the initial point to the origin.
Numerical results  verify the theoretical analysis and show the efficiency of the proposed method.  Extensions
to fractional advection-diffusion equations are currently ongoing.

\begin{figure}[!htbp]   
    \centering
    \includegraphics[height=10cm,width=16cm]{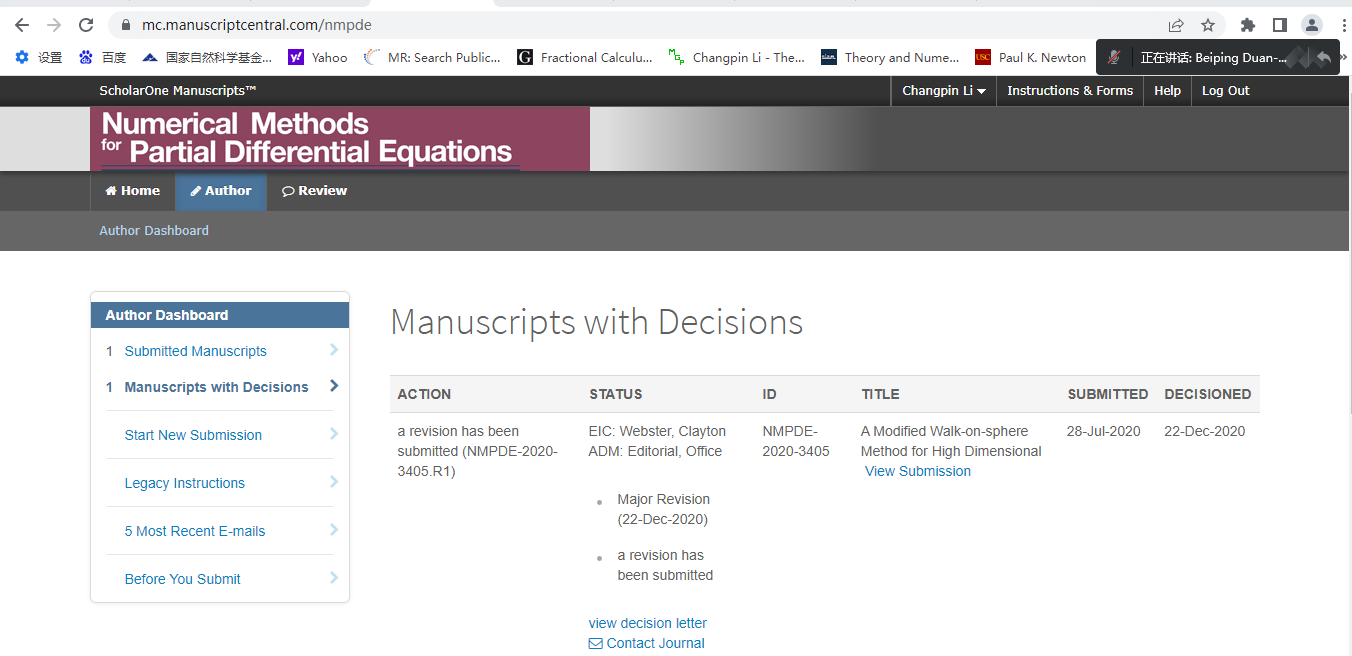}
\end{figure}


{\small
\begin{thebibliography}{3}
\addtolength{\itemsep}{-1.5 em}
\setlength{\itemsep}{2.8pt}

\bibitem{AcoBor15}
{G.~Acosta and J.~P. Borthagaray}, { A fractional {L}aplace equation: regularity of solutions and finite element
approximations}, SIAM J. Numer. Anal., 55 (2017), pp.~472--495.

\bibitem{AcoBBM18}
{ G.~Acosta, J.~P. Borthagaray, O.~Bruno, and M.~Maas}, {Regularity theory and high order numerical methods for the
(1{D})-fractional {L}aplacian}, Math. Comput., 87 (2018), pp.~1821--1857.

\bibitem{Adams75}
{R.~A. Adams}, {Sobolev Spaces}, Academic Press, New York, 1975.

\bibitem{AinsworthG18}
{M.~Ainsworth and C.~Glusa}, {Towards an efficient finite element method for the integral fractional {L}aplacian
on polygonal domains}, in Contemporary computational mathematics--a celebration of the 80th birthday
of {I}an {S}loan, Springer, Cham, 2018, pp.~17--57.

\bibitem{Applebaum2009}
  D. Applebaum, L\'{e}vy Processes and Stochastic Calculus, 
  Cambridge University Press, Cambridge, UK, 2009.


\bibitem{BonBNe18}
{A.~Bonito, J.~P. Borthagaray, R.~H. Nochetto, E.~Ot\'{a}rola, and A.~J.
	Salgado}, {Numerical methods for fractional diffusion}, Comput. Vis.
Sci., 19 (2018), pp.~19--46.

\bibitem{Bucur2016}
  C. Bucur, Some observations on the Green function for the ball in the fractional Laplace framework,
  Commun. Pure Appl. Anal., 15(2) (2016), pp.~657--699.

\bibitem{Cai&Li2019}
  M. Cai and C. P. Li, On Riesz derivative, Fract. Calc. Appl. Anal., 22(2) (2019), pp.~287--301.

\bibitem{Chen&Song1998}
  Z. Q. Chen and R. M. Song, Estimates on Green functions and Poisson kernels for symmetric stable processes,
  Math. Ann., 312(3) (1998), pp.~465--501.

\bibitem{Chopin2009}
 N. Chopin, Fast simulation of truncated Gaussian distributions, Stat. Comput., 21(2) (2011), pp.~275--288.

\bibitem{Das2011}
S. Das, Functional Fractional Calculus, Springer-Verlag, Berlin, 2011.

\bibitem{DeLaurentis&Romero1990}
J. M. Delaurentis and L. A. Romero, A Monte Carlo method for Poisson's equation, J. Comput. Phys., 90 (1990), pp.~123--140.

\bibitem{DElDGe21}
M. D'Elia, Q. Du, C. Glusa, M. Gunzburger, X. Tian, and Z. Zhou, Numerical methods for nonlocal and fractional models,
   Acta Numerica, (2021).

\bibitem{Marta13}
M.~D'Elia and M.~Gunzburger, The fractional {L}aplacian operator on bounded domains as a special case of the nonlocal diffusion
operator, Comput. Math. Appl., 66 (2013), pp.~1245--1260.

\bibitem{Duo&Wyk&Zhang2018}
S. W. Duo, H. W. Van Wyk, and Y. Z. Zhang, A novel and accurate finite difference method for the fractional Laplacian and the
fractional Poisson problem, J. Comput. Phys., 355 (2018), pp.~233--252.

\bibitem{Dyda2012}
  B. Dyda, Fractional calculus for power functions and eigenvalues of the fractional Laplacian, Fract.
  Calc. Appl. Anal., 15(4) (2012), pp.~536--555.

\bibitem{Elepov&Mihailov1973}
  B. S. Elepov and G. A. Mihailov, The `walk on spheres' algorithm for the equation $\delta{u}-cu=-g$,
  Sov. Math. Dokl., 14 (1973), pp.~1276--1280.

\bibitem{Gao&Duan2014}
  T. Gao, J. Q. Duan, X. F. Li, and R. M. Song, Mean exit time and escape probability for dynamical systems driven by L\'{e}vy
  noises, SIAM J. Sci. Comput., 36(3) (2014), pp.~A887--A906.



\bibitem{HaoZhang20}
 Z. Hao and Z. Zhang, Optimal regularity and error estimates of a spectral Galerkin method for fractional advection-diffusion-reaction
 equation, SIAM J. Numer. Anal., 58(1) (2020), pp.~211--233.

\bibitem{Herrmann2011}
  R. Herrmann, Fractional Calculus: An Introduction for Physicists, {World Scientific},
  Singapore, 2011.

\bibitem{Hilfer2000}
  R. Hilfer, Applications of Fractional Calculus in Physics, {World Scientific}, River Edge, NJ, 2000.

\bibitem{Huang&Oberman2014}
  Y. H. Huang and A. Oberman, Numerical methods for the fractional Laplacian:
  a finite difference-quadrature approach, SIAM J. Numer. Anal., 52(6) (2014), pp.~3056--3084.

\bibitem{Hwang&Mascagni2003}
  C. O. Hwang, M. Mascagni and J. A. Given, A Feynman-Kac path-integral implementation for Poisson's equation using an
  h-conditioned Green's function, Math. Comput. Simul., 62(3-6) (2003), pp.~347--355.

\bibitem{kloeden}
P. Kloeden, E. Platen, Numerical Solution of Stochastic Differential Equations, Springer, New York, 1992.

\bibitem{Kwasnicki2017}
  M. Kwa\'{s}nicki, Ten equivalent definitions of the fractional Laplace operator, Fract.
  Calc. Appl. Anal., 20(1) (2017), pp.~7--51.

\bibitem{Kyprianou&Osojnik2018}
  A. E. Kyprianou, A. Osojnik and T. Shardlow, Unbiased ``walk-on-spheres'' Monte Carlo methods
  for the fractional Laplacian, IMA J. Numer. Anal., 38 (2018), pp.~1550--1578.

\bibitem{lay}
H. A. Lay, Z. Colgin, V. Reshniak, Abdul Q. M. Khaliq, On the implementation of multilevel Monte Carlo simulation of the
stochastic volatility and interest rate model using multi-GPU clusters, Monte Carlo Meth. Appl., 24(4) (2018), pp.~309--321.

\bibitem{Li&Cai2019}
 C. P. Li and M. Cai, Theory and Numerical Approximations of Fractional Integrals and Derivatives, SIAM, Philadelphia, 2019.

\bibitem{Li&Li2020}
 C. P. Li and Z. Q. Li, Asymptotic behaviors of solution to Caputo-Hadamard fractional partial differential equation with
 fractional Laplacian, Int. J. Comput. Math., 2020, DOI:10.1080/00207160.2020.174454.

\bibitem{Li&Yi2019}
 C. P. Li and Q. Yi, Modeling and Computing of Fractional Convection Equation, Commun. on Appl. Math. and Comput., 1 (2019), pp.~565-595.

\bibitem{Lischke&Pang2020}
 A. Lischke, G. F. Pang, M. Gulian, F. Y. Song, C. Glusa, X. N. Zheng, Z. P. Mao, W. Cai, M. M. Meerschaert, M. Ainsworth,
and G. E. Karniadakis, What is the fractional Laplacian? A comparative review with new results, J. Comput. Phys.,
404 (2020), 109009.

\bibitem{Luca&David2019}
 M. Luca and L. David, Extremely efficient acceptance-rejection method for simulating uncorrelated Nakagami fading
channels, Commun. 
Statistics-Simulat.
Comput., 48(6) (2019), pp.~1798--1814.


\bibitem{Muller1956}
M. E. Muller, Some continuous Monte Carlo Methods for the Dirichlet problem, Ann. Math. Stat., 27(3) (1956), pp.~569--589.



\bibitem{Oldham&Spanier1974}
  K. B. Oldham and J. Spanier, The Fractional Calculus: Theory and Applications of
  Differentiation and Integration to Arbitrary Order, {Academic Press}, New York, London, 1974.

\bibitem{Pozrikidis2016}
  C. Pozrikidis, The Fractional Laplacian, {CRC Press}, Boca Raton, 2016.

\bibitem{Ros-Oton&Serra2014}
  X. Ros-Oton and J. Serra, The Dirichlet problem for the fractional Laplacian: regularity up to the boundary, J. Math.
  Pures Appl., 101(3) (2014), pp.~275--302.

\bibitem{Sabelfeld1991}
  K. K. Sabelfeld, Monte Carlo Methods in Boundary Value Problems, Springer, Berlin, 1991.


\bibitem{Zhang19}
{Z.~Zhang}, { Error estimate of spectral Galerkin methods for a linear fractional reaction-diffusion equation},
J. Sci. Comput., 78(2) (2019), pp.~1087--1110.

\end{thebibliography}
}
\end{document}